\documentclass[reqno,11pt]{amsart}

\usepackage{amsthm,amsfonts,amssymb,euscript,mathrsfs,graphicx,color,amsmath,latexsym,marginnote}
\usepackage{soul}

\usepackage{hyperref}

\graphicspath{{../outputs/structural_unit_fine_convergence/figures/}}

\usepackage[margin=1.4in]{geometry}
\usepackage{bm}
\usepackage{comment}

\numberwithin{equation}{section}

\def\C{{\mathbb{C}}}

\def\hat{\widehat}

\def\bar{\overline}
\def\R{{\mathbb R}}
\def\SS{{\mathbb S}}

\def\H{{\bf H}}

\def\R{{\bf R}}

\def\Z{{\mathbb Z}}

\def\g{{\bf g}}

\def\bar{\overline}

\def\R{\mathbb{R}}

\definecolor{bluegreen}{rgb}{0.0, 0.3, 0.9}

\newcommand{\norm}[1]{\left\lVert #1 \right\rVert}

\newtheorem{theorem}{Theorem}[section]
\newtheorem{lemma}[theorem]{Lemma}
\newtheorem{proposition}[theorem]{Proposition}
\newtheorem{corollary}[theorem]{Corollary}

\newtheorem{remark}[theorem]{Remark}

\begin{document}

\title[On the global asymptotic stability for the 3D Peskin Problem at critical regularity]{On the global asymptotic stability for the 3D Peskin Problem at critical regularity}

\author{Eduardo García-Juárez}
\address{Departamento de An\'alisis Matem\'atico \& IMUS, Universidad de Sevilla, C/Tarfia s/n, Campus Reina Mercedes, 41012, Sevilla, Spain}
\email{egarcia12@us.es}

\author{Susanna V. Haziot}
\address{Department of Mathematics, Princeton University, Fine Hall, Princeton, NJ 08544, USA}
\email{susanna.haziot@princeton.edu}

\author{Po-Chun Kuo}
\address{Department of Mathematics, Purdue University, West Lafayette, IN 47907, USA}
\email{kuo130@purdue.edu}

\author{Yoichiro Mori}
\address{Department of Mathematics, Department of Biology, University of Pennsylvania, Philadelphia, PA 19104, USA}
\email{y1mori@sas.upenn.edu}

\author{Han Zhou}
\address{Department of Mathematics, University of Pennsylvania, Philadelphia, PA 19104, USA}
\email{hzhou24@sas.upenn.edu}

\setcounter{tocdepth}{1}

\begin{abstract}
    We prove global well-posedness and asymptotic stability for the three-dimensional Peskin problem, which models a closed, elastic membrane immersed in an incompressible Stokes fluid. We work with initial data in the optimal regularity space $W^{1,\infty}(\mathbb{S}^2)$, which may contain infinitely many corners. These initial configurations are instantly desingularized by the flow’s parabolic smoothing effect, becoming smooth for all $t > 0$.  Then we establish that the solutions converge exponentially in the $C^1$ topology to a translated and dilated conformal sphere. 
    
    The stability is achieved by combining our nonlinear estimates with an exact structural decoupling of the 10-dimensional manifold of conformal steady states, demonstrating that the infinite-dimensional dissipative perturbation is strictly controlled. The core of our analysis is a functional framework on $\mathbb{S}^2$ that uses spectral Littlewood-Paley projections to control the highly singular multilinear operators arising from the fluid nonlinearity.
\end{abstract}

\maketitle

\tableofcontents

\section{Introduction}\label{sec:Intro}

The Peskin problem is a free-boundary fluid-structure interaction model for an elastic membrane immersed in an incompressible viscous fluid. 
Originating in Peskin's immersed boundary method for cardiac flow \cite{Peskin77, Peskin02}, the mathematical model was introduced for two-dimensional fluids in  \cite{LinTong19, MoriRodenbergSpirn19}. In this paper we deal with the three-dimensional Peskin problem, where a two-dimensional, closed, and Hookean membrane is immersed in an incompressible Stokes flow. The problem admits a boundary integral formulation (see \cite{GJPChunMoriStrain25} for the derivation) and reduces to a nonlocal and nonlinear evolution equation for the elastic interface:
\begin{equation}\label{Intro1}
(\partial_tX^i)(t,x)=\int_{\SS^2}G_{ij}(X(t,x)-X(t,p))\Delta_{\SS^2}X^j(t,p)\,d\mu(p),
\end{equation}
where $i\in\{1,2,3\}$, $x\in\SS^2$, and the map $X=(X^1,X^2,X^3):[0,T]\times\SS^2\to\R^3$ represents the membrane. Here $\Delta_{\mathbb{S}^2}$ denotes the Laplace-Beltrami operator on the sphere, $\mu$ is the induced measure, and $G_{ij}:\R^3\to\R$ denotes the Stokeslet kernel,
\begin{equation}\label{Intro2}
G_{ij}(y):=\frac{1}{8\pi}\Big(\frac{\delta_{ij}}{|y|}+\frac{y^iy^j}{|y|^3}\Big).
\end{equation}
We consider solutions that are perturbations of the identity map,
\begin{equation}\label{Intro3}
X(t,q)=X_0(q)+Y(t,q),\qquad X^j(t,q)=X^j_0(q)+Y^j(t,q),\qquad X^j_0(q):=q^j.
\end{equation}
To measure the perturbations we use Besov spaces of functions on the sphere. These spaces can be defined in two ways, using either spectral theory and spherical harmonics on the sphere or local coordinates. 

Let $\delta_0=1/8$ and let $\mathcal{P}'_j$ denote the spectral Littlewood-Paley projections, defined in Section \ref{LiPaley} below.
The space of solutions is defined by
\begin{equation}\label{Intro7}
\begin{split}
&Z[0,T]:=\big\{f\in C([0,T]:L^\infty(\SS^2)):\|f\|_{Z[0,T]}<\infty\big\},
\end{split}
\end{equation}
where 
\begin{equation*}
    \begin{aligned}
\|f\|_{Z[0,T]}:=\sup_{t\in[0,T]}\big\{\|f(t)\|_{L^\infty}+\|\nabla f(t)\|_{L^\infty}+\sup_{j\in\Z_+}2^j(1+2^jt)^{\delta_0}\|\mathcal{P}'_jf(t)\|_{L^\infty}\big\}.
    \end{aligned}   
\end{equation*}
We are now ready to state our main theorem:

\begin{theorem}\label{MainTHM} (i) There is a constant $\varepsilon_0>0$ such that if 
\begin{equation}\label{Intro4}
\|Y_0\|_{L^\infty}+\|\nabla Y_0\|_{L^\infty}\leq\varepsilon\leq\varepsilon_0
\end{equation}
then there is a unique solution $X=X_0+Y$ of the Peskin evolution equation \eqref{Intro1} on the time interval $[0,1]$, with initial data $Y(0)=Y_0$ and with the property that
\begin{equation}\label{Intro5}
\|Y\|_{Z[0,1]}\lesssim \varepsilon.
\end{equation}

(ii) Moreover, under the initial data assumption \eqref{Intro4}, the solution $X=X_0+Y$ extends to a unique global in time solution $X\in C((0,\infty): C^1(\SS^2))$ which converges in $C^1$ as $t\to\infty$ to a translated and dilated conformal steady state  (see subsection \ref{confmap} below). More precisely, there exist final state parameters $(r_\infty, b_\infty, \Lambda_\infty) \in (0, \infty) \times \mathbb{R}^3 \times \mathrm{SO}^+(3, 1)$ close to $(1, 0, I_4)$ such that
\begin{equation*}
    \|X(t) - \mathcal{S}_{r_\infty, b_\infty, \Lambda_\infty}\|_{C^1(\mathbb{S}^2)} \le C_\lambda \varepsilon e^{-a_2/r_\infty t}, \qquad t \ge 1,
\end{equation*}
where $a_2=8/35$.
Here $\mathcal{S}_{r, b, \Lambda}(\alpha) := b + r S_\Lambda(\alpha)$
is the translated and dilated conformal steady state. The final radius \(r_\infty\) is fixed by the conserved enclosed volume.
\end{theorem}

\begin{remark}
Theorem \ref{MainTHM} is stated for perturbations of the unit sphere in order to simplify the exposition and make the computations  in the proof more transparent. Nevertheless, the same argument applies to perturbations of any sphere in the steady-state manifold.    
\end{remark}

\begin{remark}
The asymptotic stability result in Theorem \ref{MainTHM} 
implies that, in a neighborhood of the unit sphere, the only steady states are the conformal spheres. However, this result does not exclude the possibility of other steady states far from the unit sphere. We mention that, in fact, the conformal spheres are  the only steady states of the three-dimensional Peskin problem\footnote{Changyou Wang, Purdue University, personal communication.}. The result is a consequence of the conformal invariance of the Dirichlet energy in two dimensions and Hopf's theorem.
\end{remark}

Notice that our choice of initial data described in \eqref{Intro4} allows for the presence of small corners. These corners are instantly desingularized, due to the smoothing factor $(1+2^jt)^{\delta_0}$ in the definition \eqref{Intro7}. In fact, one can show that the solution becomes smooth at positive times $t>0$.
The rate of convergence of the solution to the steady state manifold is exponential, see Proposition \ref{prop:decay_exact}.

\subsection{Previous and related results}

The rigorous theory of the Peskin problem begins with the independent works of Lin--Tong
\cite{LinTong19} and Mori--Rodenberg--Spirn \cite{MoriRodenbergSpirn19}. In these works, a one-dimensional closed elastic filament is immersed in a two-dimensional incompressible fluid. The tension law is linear and the fluid inertia is neglected, using the Stokes equations for the fluid motion. The relative simplicity of the model made it possible to initiate the rigorous study of  fluid-structure problems inspired by Peskin's immersed boundary method, including the long-time behavior. The works \cite{LinTong19, MoriRodenbergSpirn19} showed that, for initial interfaces of subcritical regularity and that satisfy the chord-arc condition, the contour equation behaves as a semilinear, nonlocal, parabolic evolution equation of order one.
They proved local well-posedness and smoothing in that setting, together with convergence towards the steady states for small initial data. These states were characterized as evenly stretched circles, showing the importance of the parametrization as opposed to fluid interfaces. We also mention the work by Tong 
\cite{Tong21}, where the author studied a regularized version of the problem, providing a bridge between analytical contour dynamics and numerical immersed-boundary approximations. 

Subsequent work pushed the theory toward critical spaces. García-Juárez--Mori--Strain \cite{GJMoriStrain23} obtained global well-posedness for initial interfaces near the evenly stretched circle in the Wiener algebra, for the problem with viscosity contrast, where the interior and exterior fluids may have different viscosities. 
Chen--Nguyen \cite{ChenNguyen23} proved local well-posedness for initial interfaces belonging to the closure of $C^2$ in the critical Besov space $\dot{B}_{\infty,\infty}^1$, and global for small non-Lipschitz initial data.
Cameron--Strain \cite{CameronStrain24} extended the local-in-time theory for the fully nonlinear model with general  tension law, working in the critical Besov space $\dot{B}_{2,1}^{3/2}$. García-Juárez--Haziot \cite{GJHaziot25} then obtained global well-posedness for the nonlinear tension model in critical spaces allowing initial interfaces with corners.
The work by Chen-Hu-Nguyen \cite{ChenHuNguyen26} develops the Schauder method for a class of nonlocal quasilinear parabolic equations in critical spaces, including the Peskin problem with nonlinear law.

In parallel, several works focused on broadening the conditions that guarantee global existence. Gancedo--Granero-Belinchón--Scrobogna \cite{GancedoGraneroBelinchonScrobogna23} introduced a toy model capturing the normal motion while discarding tangential stretching. 
Tong 
\cite{Tong24} obtained the tangential Peskin problem, that considers a long straight elastic string deforming tangentially in its own Stokes flow. Building upon this latter work and finding maximum principles for geometric quantities such as the curvature, Tong--Wei
\cite{TongDongyi24} showed global well-posedness under an explicit medium-size condition. This condition only imposes a size constraint on the normal deviation of the initial data.
Very recently, Tong--Wei \cite{TongWei25} extended the 2D Peskin problem to fluids with inertia, namely an elastic filament immersed in fluids governed by the Navier--Stokes equations. In a different direction, 
Kandalam--Spirn \cite{KandalamSpirn26} introduced the anchored Peskin problem, by considering the Peskin problem in the half-plane with an elastic filament whose endpoints are anchored to a no-slip wall.

The three-dimensional setting addressed in this paper is less explored due to its geometrical complexity.
The formulation of the 3D Peskin problem with a general membrane elasticity law was derived by García-Juárez--Po-Chun--Mori--Strain in \cite{GJPChunMoriStrain25}, where the authors also proved local well-posedness of the problem in subcritical H\"older spaces.

There are many other problems closely related to the Peskin problem in which a material interface is immersed in a Stokes fluid. In the inextensible filament problem, the interface exerts a bending force but is constrained to be locally inextensible, a model that is widely used to simulate membrane vesicle dynamics. Well-posedness for this model in 2D was obtained in \cite{GJPoChunMori25}. See also \cite{Li21} for well-posedness for a related model. In the Stokes droplet problem, the fluid interface exerts a surface tension force. The local-in-time well-posedness theory for this problem in subcritical regularity is now well-understood
(see \cite{PrussSimonett16} and the references therein).
In the context of fluid-fluid interfaces, very recent results have appeared for three-dimensional bubbles in different models (see e.g. \cite{LaiWeinstein23}, \cite{MeyerNiebelChristian25}, \cite{BaldiJulinLaManna26},
\cite{Shao26}, \cite{GigaGu26}).
Mathematically, the Peskin problem is particularly connected
with the Muskat problem, where incompressible fluids flow through porous media. They both admit contour formulations as nonlocal and nonlinear equations of parabolic type. Results with critical regularity and dealing with the movement of closed near-circular interfaces  have also very recently appeared for the two-dimensional Muskat problem 
\cite{AlazardNguyen21,
AlazardNguyen22,
AlazardNguyen23,
ChenNguyenXu21,
GJGomezSerranoHaziotPausader24,
GancedoGJPatelStrain23,
GancedoGJPatelStrain25}.

A distinctive feature of the three-dimensional Peskin problem, compared with the fluid interface problems mentioned above, is the appearance of a larger conformal family of steady states. For  fluid interfaces, only the shape of the interface matters, and the natural neutral directions near a sphere come from Euclidean symmetries and changes of scale. In the Peskin problem, by contrast, the parametrization is part of the physical unknown, since the elastic force is generated by the membrane map itself. The conformal invariance of the Dirichlet energy in two dimensions therefore produces an additional finite-dimensional manifold of steady states, generated by the conformal group of $\SS^2$, which is isomorphic to $\mathbb{SO}^+(3,1)$ (see Section \ref{sec:preliminaries}). This is closely related to the classical role of conformal invariance in harmonic map theory, but it seems to be less common in free-boundary fluid models, where conformal maps are often used as analytical coordinates rather than as a genuine family of material steady states. 

The problems discussed above deal with co-dimension 1 structures immersed in an ambient fluid (one-dimensional interfaces in two-dimensional fluids, or two-dimensional interfaces in three-dimensional fluids). The slender body problem, in which a thin filament (essentially a co-dimension 2 structure) is immersed in a three-dimensional Stokes fluid, is also of interest. Its analytical study was initiated in  \cite{MoriOhmSpirn20, MoriOhmSpirn202}. We refer the reader to \cite{Ohm26,Ohm25} for a recent well-posedness theory.

\subsection{Outline of the paper}

Section \ref{sec:Intro} introduces the boundary-integral formulation of the 3D Peskin problem and presents our main result, Theorem \ref{MainTHM}. Section \ref{sec:preliminaries} provides the necessary harmonic analysis on $\mathbb{S}^2$, establishing the 10-dimensional manifold of conformal steady states generated by the group $\mathbb{SO}^+(3,1)$. In Section \ref{sec:linearization}, we compute the linearized operator $\mathcal{N}_1$ about the unit sphere (Proposition \ref{linearSummary}) and identify its kernel with the tangent space of the steady-state manifold (Lemma \ref{LemSteadyKernel}). Section \ref{sec:nonlinear} uses the harmonic analysis tools developed in Sections \ref{sec:preliminaries} and \ref{sec:appendix} to derive the critical nonlinear estimates (Lemma \ref{LemmaNonlinMain}), which are required to close the fixed-point argument for local existence. The global extension and long-time stability are addressed in Section \ref{sec:global}; here, we decouple the flow by projecting onto the steady-state manifold and its orthogonal complement, ultimately proving exponential convergence to the manifold of conformal spheres (Proposition \ref{prop:decay_exact}). Technical results regarding spectral multipliers on $\mathbb{S}^2$, including Littlewood-Paley projection bounds and linear semigroup estimates, are collected in Section \ref{sec:appendix}. Finally, Section \ref{sec:numerical} provides a numerical verification of the exponential convergence rates for the perturbation $Y$ and the parameter vector $p$.

\section{Preliminaries}\label{sec:preliminaries}

In this section we summarize some aspects of the harmonic analysis on the sphere $\mathbb{S}^2$, which are used in our proof of the main theorem.

\subsection{Riemannian structure on the sphere $\mathbb{S}^2$}\label{coord}

We start with the standard spherical coordinates on $\R^3$,
\begin{equation}\label{coo1}
x^1=r\sin\theta\cos\phi,\qquad x^2=r\sin\theta\sin\phi,\qquad x^3=r\cos\theta,
\end{equation}
where $(r,\theta,\phi)\in(0,\infty)\times (0,\pi)\times(0,2\pi)$. Then we calculate the induced vector-fields
\begin{equation}\label{coo2}
\begin{split}
&\partial_r=\sin\theta\cos\phi\cdot\partial_1+\sin\theta\sin\phi\cdot\partial_2+\cos\theta\cdot\partial_3,\\
&\partial_\theta=r\cos\theta\cos\phi\cdot\partial_1+r\cos\theta\sin\phi\cdot\partial_2-r\sin\theta\cdot\partial_3,\\
&\partial_\phi=-r\sin\theta\sin\phi\cdot\partial_1+r\sin\theta\cos\phi\cdot\partial_2,
\end{split}
\end{equation}
where $\partial_j:=\partial_{x^j}$, $j\in\{1,2,3\}$. We can therefore calculate the Euclidean metric components in spherical coordinates,
\begin{equation}\label{coo3}
\begin{split}
&\g_{\R^3}(\partial_r,\partial_r)=1,\qquad\,\,\g_{\R^3}(\partial_r,\partial_\theta)=0,\qquad\,\g_{\R^3}(\partial_r,\partial_\phi)=0,\\
&\g_{\R^3}(\partial_\theta,\partial_\theta)=r^2,\qquad \g_{\R^3}(\partial_\theta,\partial_\phi)=0,\qquad\g_{\R^3}(\partial_\phi,\partial_\phi)=r^2(\sin\theta)^2.
\end{split}
\end{equation}
Therefore 
\begin{equation}\label{coo4}
\begin{split}
&\sqrt{|\g_{\R^3}|}=\sqrt{|\det\g_{\R^3}|}=r^2\sin\theta ,\\
&d\mu_{\R^3}=\sqrt{|\g_{\R^3}|}\,drd\theta d\phi=r^2\sin\theta\,drd\theta d\phi,\\
&\Delta_{\R^3}=\frac{1}{\sqrt{|\g_{\R^3}|}}\partial_i(\sqrt{|\g_{\R^3}|}\g_{\R^3}^{ij}\partial_j)=\partial_r^2+\frac{2}{r}\partial_r+\frac{1}{r^2\sin\theta}\partial_\theta(\sin\theta \partial_\theta)+\frac{1}{r^2(\sin\theta)^2}\partial_\phi^2.
\end{split}
\end{equation}
The Euclidean metric induces a metric on the sphere $\SS^2$, given by
\begin{equation}\label{coo5}
\g_{\SS^2}(\partial_\theta,\partial_\theta)=1,\qquad \g_{\SS^2}(\partial_\theta,\partial_\phi)=0,\qquad\g_{\SS^2}(\partial_\phi,\partial_\phi)=(\sin\theta)^2,
\end{equation}
with corresponding measure and Laplace-Beltrami operator
\begin{equation}\label{coo6}
d\mu_{\SS^2}=\sin\theta\,d\theta d\phi,\qquad \Delta_{\SS^2}=\frac{1}{\sin\theta}\partial_\theta(\sin\theta \partial_\theta)+\frac{1}{(\sin\theta)^2}\partial_\phi^2.
\end{equation}

\subsubsection{The vector-fields $V_{ij}$} We define three vector-fields on the sphere,
\begin{equation}\label{vf11}
V_{12}:=x^1\partial_2-x^2\partial_1,\qquad V_{23}:=x^2\partial_3-x^3\partial_2, \qquad V_{31}:=x^3\partial_1-x^1\partial_3.
\end{equation}
Notice that these vector-fields span the tangent space of the sphere $\mathbb{S}^2$ at all points. Moreover, using the identities \eqref{coo2} we have
\begin{equation}\label{vf12}
V_{12}=\partial_\phi,\qquad V_{23}=-\sin\phi\partial_\theta-\frac{\cos\phi\cos\theta}{\sin\theta}\partial_\phi,\qquad V_{31}=\cos\phi\partial_\theta-\frac{\sin\phi\cos\theta}{\sin\theta}\partial_\phi.
\end{equation}
Using now the metric formulas \eqref{coo5} and the Koszul formula  it follows that  $\mathrm{div}_{\SS^2}(V_{12})=0$, $\mathrm{div}_{\SS^2}(V_{23})=0$, and $\mathrm{div}_{\SS^2}(V_{31})=0$. In particular 
\begin{equation}\label{vf13}
\int_{\mathbb{S}^2}V_{12}(f)\,d\mu=\int_{\mathbb{S}^2}V_{23}(f)\,d\mu=\int_{\mathbb{S}^2}V_{31}(f)\,d\mu=0.
\end{equation}
for any $f\in H^1(\mathbb{S}^2)$. Finally, we can write the Laplace-Beltrami operator as
\begin{equation}\label{vf14}
\Delta_{\mathbb{S}^2}g=V_{12}^2g+V_{23}^2g+V_{31}^2g
\end{equation}
for any $g\in H^2(\mathbb{S}^2)$, as a consequence of the formulas \eqref{coo6} and \eqref{vf12}. 
The formulas \eqref{vf13} are very useful to integrate by parts on the sphere, while the formula \eqref{vf14} is used in the proof of the key Lemma \ref{LemmaNonlinMain}.

\subsection{Conformal symmetries on $\SS^2$}
\label{confmap}

The group of orientation-preserving conformal diffeomorphisms of the unit sphere $\SS^2$ is isomorphic to the restricted Lorentz group, denoted by $\mathbb{SO}^+(3,1)$. The elements of this group can be represented as $4 \times 4$ real matrices $\Lambda$ that satisfy the fundamental property
\begin{equation}\label{LorentzProp}
\Lambda^{\top} J \Lambda=J, \qquad J=\begin{pmatrix} -1 & \mathbf{0}_3^\top \\ \mathbf{0}_3 & \mathbb{I}_3 \end{pmatrix},
\end{equation}
subject to the additional constraints $\Lambda_{11}>0$ and $\det \Lambda=1$. 
To describe the geometric action of $\mathbb{SO}^+(3,1)$ on the sphere, we identify points on $\SS^2$ with projective coordinates in $\R^4$. For any point $p \in \SS^2$, we define the extended vector $\widetilde{p} = (1, p^\top)^\top \in \R^4$. The action of a matrix $\Lambda \in \mathbb{SO}^+(3,1)$ yields a conformal mapping $\mathcal{S}_\Lambda: \SS^2 \to \SS^2$, defined by the projective transformation
\begin{equation}\label{ConformalAction1}
\mathcal{S}_\Lambda(p)=\frac{\mathbb{P}_3\Lambda \widetilde{p}}{e_0^\top\Lambda \widetilde{p}},
\end{equation}
where $\mathbb{P}_3=(\mathbf{0}_3\; \mathbb{I}_3)$ is the coordinate projection from $\R^4$ to $\R^3$, and $e_0=(1, \mathbf{0}^\top_3)^\top$. 

Because $\mathcal{S}_\Lambda$ is a conformal map, the induced metric tensor is proportional to the standard metric $\g_{\SS^2}$ on the sphere. This conformal structure induces precise scaling properties on differential operators. Specifically, for any sufficiently smooth function $f$, the Laplace-Beltrami operator evaluated on the mapped coordinates transforms by a scalar factor governed by the metric determinants.

\subsubsection{Conformal maps as steady states}

We now demonstrate that the family of conformal maps generated by $\mathbb{SO}^+(3,1)$ characterizes a continuous manifold of steady states for the fluid-structure system. 

\begin{lemma}\label{LemConformalSteady}
	For any $\Lambda \in \mathbb{SO}^+(3,1)$, the conformal map $X(p) = \mathcal{S}_\Lambda(p)$ is a steady state of the Peskin evolution equation \eqref{Intro1}.
\end{lemma}

\begin{proof}
	Let $X(p) = \mathcal{S}_\Lambda(p)$ for an arbitrary $\Lambda \in \mathbb{SO}^+(3,1)$. To establish that this is a steady state, we must show that the nonlinearity vanishes, namely $\mathcal{N}^i(X)(x) = 0$ for all $x \in \SS^2$. By definition, the operator evaluates to
	\begin{equation}\label{SteadyProof1}
	\mathcal{N}^i(X)(x) = \int_{\SS^2} G_{ij}(\mathcal{S}_\Lambda(x) - \mathcal{S}_\Lambda(p)) \Delta_{\SS^2} \mathcal{S}_\Lambda^j(p) \, d\mu(p).
	\end{equation}
	We introduce the change of variables $y = \mathcal{S}_\Lambda(p)$. Because $\mathcal{S}_\Lambda$ is an orientation-preserving conformal map defined on a two-dimensional surface, the scaling factor induced on the Laplace-Beltrami operator is the exact inverse of the scaling factor induced on the surface measure. Consequently, changing the variable of integration yields the exact equivalence
	\begin{equation}\label{SteadyProof2}
	\int_{\SS^2} G_{ij}(\mathcal{S}_\Lambda(x) - \mathcal{S}_\Lambda(p)) \Delta_{\SS^2} \mathcal{S}_\Lambda^j(p) \, d\mu(p) = \int_{\SS^2} G_{ij}(\mathcal{S}_\Lambda(x) - y) \Delta_{\SS^2} X_0^j(y) \, d\mu(y),
	\end{equation}
	where $X_0(y) = y$ represents the identity mapping. 
	We recognize the resulting integral as the $0$-th order term of the Peskin operator, $\mathcal{N}^i_0(X_0)$, evaluated at $\mathcal{S}_\Lambda(x)$. Since the identity map $X_0$ is itself a steady state, $\mathcal{N}^i_0(X_0)(z) = 0$ for any  $z \in \SS^2$ (see \eqref{Pes5}). Therefore, $\mathcal{N}^i(X)(x) = 0$, completing the proof.
    
\end{proof}

\subsubsection{A 10 dimensional family of steady states}

In addition to the conformal symmetries, the Peskin evolution equation \eqref{Intro1} admits steady states corresponding to physical translations and dilations of the sphere. We can verify this by directly evaluating the nonlinearity for these transformations, and using the fact that the Laplace-Beltrami operator kills constants and scales linearly with respect to dilations. To summarize, we have found a 10-dimensional manifold of steady solutions $S_{r,b,\Lambda}:\SS^2\to\R^3$ given by
\begin{equation}\label{mainfamily}
\mathcal{S}_{r,b,\Lambda}(p):=b+r\mathcal{S}_\Lambda(p),
\end{equation}
parametrized over $r\in(0,\infty)$, $b\in\R^3$, and $\Lambda\in\mathbb{SO}^+(3,1)$.

\subsection{Spherical harmonics on $\SS^2$}\label{harm}

To understand the dynamics of solutions of the Peskin problem \eqref{Intro1} with critical regularity, we need spectral analysis of the nonlinearity. For this we use spherical harmonics, following the framework and the notation in \cite[Chapter IV]{Steinbook71}. 

For any integer $k\in\Z_+$ let $\mathcal{P}_k$ denote the space of complex-valued homogeneous polynomials of degree $k$ on $\R^3$, and let $\mathcal{A}_k$ denote the subspace of harmonic polynomials,
\begin{equation}\label{intr1}
\mathcal{A}_k:=\{F\in\mathcal{P}_k:\,\Delta_{\R^3} F=0\}.
\end{equation}
Let $\mathcal{H}_k$ denote the induced space of spherical harmonics,
\begin{equation}\label{intr2}
\mathcal{H}_k:=\{f:\SS^2\to\C:\,\text{ there is }F\in\mathcal{A}_k\text{ such that }f(x)=F(x)\text{ for any }x\in\SS^2\}.
\end{equation}
We have the following:

\begin{lemma}\label{Lem1}
(i) For any $k\in\Z_+$ the Laplace operator $\Delta$ defines a surjective mapping from $\mathcal{P}_k$ to $\mathcal{P}_{k-2}$. Moreover, any $F\in\mathcal{P}_k$ can be written uniquely in the form
\begin{equation}\label{intr3}
F(x)=\sum_{l\in[0,k/2]}|x|^{2l}G_{k-2l}(x)
\end{equation}
for some harmonic polynomials $G_{k-2l}\in\mathcal{A}_{k-2l}$.

(ii) The vector spaces $\mathcal{P}_k$, $\mathcal{A}_k$, $\mathcal{H}_k$ have dimensions
\begin{equation}\label{intr4}
\mathrm{dim}(\mathcal{P}_k)=(k+1)(k+2)/2,\qquad \mathrm{dim}(\mathcal{A}_k)=\mathrm{dim}(\mathcal{H}_k)=2k+1.
\end{equation}

(iii) The collection of finite linear combinations of functions in $\bigcup_{k\geq 0}\mathcal{H}_k$ is dense in $L^2(\SS^2)$. Moreover, the spaces $\mathcal{H}_k$ are mutually orthogonal,
\begin{equation}\label{inr5}
\int_{\SS^2}Y_k(p)\bar{Y_l(p)}\,d\mu_{\SS^2}(p)=0
\end{equation}
if $Y_k\in\mathcal{H}_k$, $Y_l\in\mathcal{H}_l$, $k\neq l$.
\end{lemma}

See \cite[Theorem 2.1, Corollary 2.2, Corollary 2.3, Corollary 2.4]{Steinbook71} for the proof.

Given $F\in\mathcal{A}_k$, the equation $\Delta_{\R^3}F=0$ and the homogeneity assumption give
\begin{equation*}
0=\partial_r^2F+\frac{2}{r}\partial_rF+\frac{1}{r^2} \Delta_{\SS^2}F=\frac{1}{r^2}[\Delta_{\SS^2}F+k(k+1)F].
\end{equation*}
Therefore, if $Y_k\in\mathcal{H}_k$ is a spherical harmonic then
\begin{equation}\label{coo7}
\Delta_{\SS^2}Y_k=-k(k+1)Y_k.
\end{equation}
We can find an orthonormal basis of spherical harmonics for the space $\mathcal{H}_k$ of the form 
\begin{equation}\label{coo8}
\big\{Y_{k,m}(\theta,\phi)=e^{im\phi}c_{k,m}P_{k,m}(\cos\theta):\,m\in\{0,\pm 1,\ldots,\pm k\}\big\},
\end{equation}
where $P_{k,m}:[-1,1]\to\R$ are the associated Legendre polynomials and $c_{k,m}$ are normalization coefficients. 

\subsection{Zonal harmonics and convolutions}\label{zon}

We start with a lemma that describes the zonal spherical harmonics on $\SS^2$.

\begin{lemma}\label{zon2}

(i) For any $k\in\Z_+$ there is a function $Z_k:\SS^2\times\SS^2\to\mathbb{R}$ such that 
\begin{equation}\label{zon5}
Z_k(x,.)\in\mathcal{H}_k\qquad\text{ for any }x\in\SS^2,
\end{equation}
\begin{equation}\label{zon3}
Y(x)=\int_{\SS^2}Y(y)Z_k(x,y)\,d\mu(y)\qquad\text{ for any }Y\in\mathcal{H}_k\text{ and }x\in\SS^2,
\end{equation}
\begin{equation}\label{zon6}
|Z_k(x,y)|\leq Z_k(x,x)=(2k+1)/(4\pi)\qquad\text{ for any }x,y\in\SS^2.
\end{equation}
Moreover, there is a function $F_k:[-1,1]\to\R$ such that
\begin{equation}\label{zon4}
Z_k(x,y)=F_k(x\cdot y)\qquad\text{ for any }x,y\in\SS^2.
\end{equation}
In particular, $Z_k(x,y)=Z_k(y,x)$ and $Z_k(x,y)=Z_k(\rho x,\rho y)$ for any $x,y\in\SS^2$ and any rotation $\rho\in\mathbb{SO}(3)$.

(ii) Assume that $G:\SS^2\times\SS^2\to\C$ is a function such that $G(x,.)\in\mathcal{H}_k$ for any $x\in\SS^2$ and $G(x,y)=G(\rho x,\rho y)$ for any $x,y\in\SS^2$ and any $\rho\in\mathbb{SO}(3)$. Then there is $c\in\C$ such that $G=cZ_k$.
\end{lemma}

See \cite[Lemma 2.8, Corollary 2.9, Corollary 2.13]{Steinbook71} for the proof.

The zonal spherical harmonics can be calculated explicitly using the Poisson kernel. Let
\begin{equation}\label{zon5.5}
p(t,x):=\frac{1}{4\pi}\frac{1-|x|^2}{|x-t|^3}.
\end{equation}
Then, if $t,y\in\SS^2$ and $|r|\leq 1/4$ then
\begin{equation*}
p(t,ry)=\sum_{k\geq 0}r^kZ_k(t,y),
\end{equation*}
or equivalently, in terms of the functions $F_k$ defined in \eqref{zon4},
\begin{equation}\label{zon7}
\frac{1-r^2}{4\pi}(1+r^2-2r \alpha)^{-3/2}=\sum_{k\geq 0}r^kF_k(\alpha),\qquad \alpha\in[-1,1],\,r\in[-1/4,1/4].
\end{equation}
See  \cite[Theorem 2.10]{Steinbook71} for the proof. In particular, if $\alpha\in[-1,1]$ then
\begin{equation}\label{zon8}
F_0(\alpha)=\frac{1}{4\pi},\qquad F_1(\alpha)=\frac{3\alpha}{4\pi},\qquad F_2(\alpha)=\frac{15\alpha^2-5}{8\pi}.
\end{equation}
The functions $F_k$ can also be expressed in terms of the Legendre polynomials $P_k=P_{k,0}$, defined by the series expansion
\begin{equation}\label{zon7.5}
(1+r^2-2r\alpha)^{-1/2}=\sum_{k\geq 0}P_k(\alpha)r^k,\qquad \alpha\in[-1,1],\,r\in[-1/4,1/4].
\end{equation}
Taking an $r$-derivative we get
\begin{equation*}
(\alpha-r)(1+r^2-2r\alpha)^{-3/2}=\sum_{k\geq 0}kP_k(\alpha)r^{k-1}.
\end{equation*}
We multiply this by $2r$ and add up the identity $(1+r^2-2r\alpha)(1+r^2-2r\alpha)^{-3/2}=\sum_{k\geq 0}P_k(\alpha)r^k$. We can now compare with the identity \eqref{zon7} to derive the formula
\begin{equation}\label{zon7.6}
(2k+1)P_k(\alpha)=4\pi F_k(\alpha),\qquad k\in\Z_+,\,\alpha\in[-1,1].
\end{equation}

\subsection{Spectral Littlewood--Paley projections}\label{LiPaley} We fix a smooth even function $\varphi:\mathbb{R}\to[0,1]$ supported in the interval $[-2,2]$ and equal to $1$ in the interval $[-1,1]$. 
For any $a\in\R$ and $d\geq 1$ we define the functions $\varphi_{\leq a},\varphi_{>a}:\mathbb{R}^d\to [-1,1]$ by
\begin{equation}\label{lipa1}
\varphi_{\leq a}(x):=\varphi(|x|/2^{a}),\qquad \varphi_{>a}:=1-\varphi_{\leq a}.
\end{equation}
For any $a\in\Z$ and any bounded interval $I\subseteq\R$ we define
\begin{equation}\label{phi_Int}
\varphi_a(x):=\varphi(|x|/2^a)-\varphi(|x|/2^{a-1}),\qquad \varphi_I(x):=\sum_{a\in\Z\cap I}\varphi_a(x).
\end{equation}
In view of Lemmas \ref{Lem1} and \ref{zon2}, any function $Y\in L^2(\mathbb{S}^2)$ can be uniquely decomposed
\begin{equation}\label{lipa2}
\begin{split}
Y=\sum_{k\geq 0}Y_k,\qquad Y_k\in\mathcal{H}_k,\\
Y_k(x)=\int_{\mathbb{S}^2}Y(y)Z_k(x,y)\,d\mu(y),
\end{split}
\end{equation}
where $Z_k$ are the zonal spherical harmonics defined in  Lemma \ref{zon2}. For any $a\in\Z_+$ we define the spectral Littlewood-Paley operators
\begin{equation}\label{lipa3}
\begin{split}
&\mathcal{P}'_{\leq a}Y(x):=\sum_{k\in\Z_+}\varphi_{\leq a}(k)Y_k(x)=\int_{\mathbb{S}^2}Y(y)\mathcal{K}_{\leq a}(x,y)\,d\mu(y),\\
&\mathcal{P}'_{a}Y:=\sum_{k\in\Z_+}\varphi_{a}(k)Y_k=\int_{\mathbb{S}^2}Y(y)\mathcal{K}_{a}(x,y)\,d\mu(y),\qquad a\geq 1,
\end{split}
\end{equation}
where the kernels $\mathcal{K}_{\leq a}$ and $\mathcal{K}_a$ are defined by
\begin{equation}\label{lipa4}
\mathcal{K}_{\leq a}:=\sum_{k\in\Z_+}\varphi_{\leq a}(k)Z_k,\qquad\qquad\mathcal{K}_{a}:=\sum_{k\in\Z_+}\varphi_{a}(k)Z_k,\,a\geq 1.
\end{equation}
For uniformity of notation, we let $\mathcal{P}'_0:=\mathcal{P}'_{\leq 0}$ and $\mathcal{K}_0:=\mathcal{K}_{\leq 0}$.

Notice that our definitions on $\SS^2$ are analogous to the usual definitions on Euclidean spaces. Some care is needed however, since we need to show that these spectral projections, and more generally operators defined by suitable spectral multipliers (such as those appearing in the linearization of the Peskin equation in Proposition \ref{linearSummary}), are compatible with $L^p$ spaces and with classical derivatives (which appear in the analysis of the nonlinearity). We develop a self-contained functional analysis framework in Section \ref{sec:appendix} below.

\section{Linearization of the Peskin nonlinearity}\label{sec:linearization}

Recall the Peskin nonlinearity
\begin{equation}\label{Pes1}
\mathcal{N}^i(X)(x):=\int_{\SS^2}G_{ij}(X(x)-X(p))\Delta_{\SS^2}X^j(p)\,d\mu(p),
\end{equation}
where $i\in\{1,2,3\}$, $x\in\SS^2$, $X=(X^1,X^2,X^3):\SS^2\to\R^3$ is the solution, and $G_{ij}:\R^3\to\R$ denote the Stokeslet kernel defined in \eqref{Intro2}. 
We consider solutions that are perturbations of the identity map,
\begin{equation}\label{Pes3}
X(q)=X_0(q)+Y(q),\qquad X^j(q)=X^j_0(q)+Y^j(q),\qquad X^j_0(q):=q^j.
\end{equation}
We calculate first the $0$-th order term
\begin{equation}\label{Pes4}
\mathcal{N}^i_0(X)(x):=\int_{\SS^2}G_{ij}(X_0(x)-X_0(p))\Delta_{\SS^2}X^j_0(p)\,d\mu(p).
\end{equation}
In view of \eqref{coo7}, we have $\Delta_{\SS^2}X^j_0(p)=-2p^j$, so
\begin{equation*}
\begin{split}
G_{ij}(X_0(x)-X_0(p))&\Delta_{\SS^2}X^j_0(p)=\frac{-1}{4\pi}\Big(\frac{\delta_{ij}}{|x-p|}+\frac{(x-p)^i(x-p)^j}{|x-p|^3}\Big)p^j\\
&=\frac{-1}{4\pi}\Big(\frac{p^i}{|x-p|}-\frac{x^i-p^i}{2|x-p|}\Big)=\frac{-1}{4\pi}\Big(\frac{3p^i}{|x-p|}-\frac{x^i}{2|x-p|}\Big).
\end{split}
\end{equation*}
for any $x,p\in\SS^2$. Using \eqref{exp21}  with $k=0$ and $k=1$ we calculate 
\begin{equation}\label{Pes5}
\mathcal{N}^i_0(X)(x):=\int_{\SS^2}\frac{1}{4\pi}\Big(\frac{x^i}{2|x-p|}-\frac{3p^i}{|x-p|}\Big)\,d\mu(p)=x^i-x^i=0.
\end{equation}
In particular, the identity map is a steady solution of the Peskin evolution equation.

We calculate now the first order term (linear in $Y$)
\begin{equation}\label{Pes6}
\begin{split}
\mathcal{N}^i_1(Y)&=\mathcal{N}^i_{1,1}(Y)+\mathcal{N}^i_{1,2}(Y),\\
\mathcal{N}^i_{1,1}(Y)(x)&:=\int_{\SS^2}G_{ij}(X_0(x)-X_0(p))\Delta_{\SS^2}Y^j(p)\,d\mu(p)\\
\mathcal{N}^i_{1,2}(Y)(x)&:=\int_{\SS^2}(\partial_lG_{ij})(X_0(x)-X_0(p))(Y^l(x)-Y^l(p))\Delta_{\SS^2}X_0^j(p)\,d\mu(p).
\end{split}
\end{equation}
We decompose $Y=\sum_{k\geq 0}Y_k$, $Y^j=\sum_{k\geq 0}Y^j_k$, $j\in\{1,2,3\}$, where $Y^j_k\in\mathcal{H}_k$ are scalar spherical harmonics. Then we calculate, using \eqref{coo7} and \eqref{Intro2},
\begin{equation}\label{Hes7}
\mathcal{N}^i_{1,1}(Y_k)(x)=\int_{\SS^2}\frac{1}{8\pi}\Big(\frac{\delta_{ij}}{|x-p|}+\frac{(x-p)^i(x-p)^j}{|x-p|^3}\Big)(-k)(k+1)Y^j_k(p)\,d\mu(p).
\end{equation}
To calculate $\mathcal{N}^i_{1,2}(Y_k)$ we write, for any $x,p\in\SS^2$,
\begin{equation*}
\begin{split}
(\partial_lG_{ij})&(X_0(x)-X_0(p))\Delta_{\SS^2}X_0^j(p)=-2p^j(\partial_lG_{ij})(x-p)\\
&=\frac{-1}{4\pi}p^j\Big(-\frac{\delta_{ij}(x-p)^l}{|x-p|^3}+\frac{\delta_{il}(x-p)^j}{|x-p|^3}+\frac{\delta_{jl}(x-p)^i}{|x-p|^3}-3\frac{(x-p)^i(x-p)^j(x-p)^l}{|x-p|^5}\Big)\\
&=\frac{-1}{4\pi}\Big(-\frac{p^i(x-p)^l}{|x-p|^3}+\frac{p^l(x-p)^i}{|x-p|^3}-\frac{\delta_{il}}{2|x-p|}+3\frac{(x-p)^i(x-p)^l}{2|x-p|^3}\Big)\\
&=\frac{1}{4\pi}\Big(\frac{x^i(x-p)^l}{|x-p|^3}-\frac{x^l(x-p)^i}{|x-p|^3}+\frac{\delta_{il}}{2|x-p|}-3\frac{(x-p)^i(x-p)^l}{2|x-p|^3}\Big).
\end{split}
\end{equation*}
Therefore, using the definitions \eqref{exp20},
\begin{equation}\label{Hes8}
\begin{split}
&\mathcal{N}^i_{1,2}(Y_k)(x)\\
&=\frac{1}{4\pi}\int_{\SS^2}\Big(\frac{x^i(x-p)^l}{|x-p|^3}-\frac{x^l(x-p)^i}{|x-p|^3}+\frac{\delta_{il}}{|x-p|}-3\frac{(x-p)^i(x-p)^l}{2|x-p|^3}\Big)(Y^l_k(x)-Y^l_k(p))d\mu(p).
\end{split}
\end{equation}

\subsection{Explicit calculations}\label{explicit}

We would like to calculate now explicitly the linear operators $\mathcal{N}^i_{1,1}$ and $\mathcal{N}^i_{1,2}$. For this we prove the following lemma:

\begin{lemma}\label{PropExp}
We define the linear operators $S,T^{ij},R^i$ on $L^2(\SS^2)$ by the formulas
\begin{equation}\label{exp20}
\begin{split}
&S(Y)(x):=\int_{\SS^2}\frac{1}{|x-p|}Y(p)\,d\mu(p),\\
&T^{ij}(Y)(x):=\int_{\SS^2}\frac{(x^i-p^i)(x^j-p^j)}{|x-p|^3}Y(p)\,d\mu(p),\\
&R^i(Y)(x):=\int_{\SS^2}\frac{x^i-p^i}{|x-p|^3}(Y(x)-Y(p))\,d\mu(p).
\end{split}
\end{equation}
for $i,j\in\{1,2,3\}$ and $x\in\SS^2$. Then, for any $k\in\Z_+$ and $Y_k\in\mathcal{H}_k$ we have
\begin{equation}\label{exp21}
S(Y_k)(x)=\frac{4\pi}{2k+1}\widetilde{Y}_k(x),
\end{equation}
\begin{equation}\label{exp22}
\begin{split}
T^{ij}(Y_k)(x)&=\frac{4\pi\delta_{ij}\widetilde{Y}_k(x)}{2k+3}-\frac{16\pi}{(2k+1)(2k+3)}\Big\{\frac{x^i\partial_j\widetilde{Y}_k(x)+x^j\partial_i\widetilde{Y}_k(x)}{2}-\frac{|x|^2\partial_i\partial_j\widetilde{Y}_k(x)}{2k-1}\Big\},
\end{split}
\end{equation}
\begin{equation}\label{exp22.5}
R^{i}(Y_k)(x)=\frac{4\pi}{2k+1}\partial_i\widetilde{Y}_k,
\end{equation}
for any $x\in\SS^2$, where $\widetilde{Y}_k\in\mathcal{A}_k$ denotes the $k$-homogeneous harmonic extension of $Y_k$. In particular, $S(\mathcal{H}_k)\subseteq \mathcal{H}_k$ and $T^{ij}(\mathcal{H}_k)\subseteq \mathcal{H}_k$ for any $k\in\Z_+$.
\end{lemma}

The rest of this subsection is concerned with the proof of Lemma \ref{PropExp}. We start with a general lemma involving convolution operators on the sphere.

\begin{lemma}\label{zon10}
Assume $L:[-1,1]\to\C$ is a kernel, $k\geq 0$, and $Y\in\mathcal{H}_k$. Then
\begin{equation}\label{zon11}
\int_{\SS^2}L(x\cdot y)Y(y)\,d\mu(y)=C_{L,k}Y(x)
\end{equation}
for any $x\in\SS^2$, where $F_k$ is defined as in \eqref{zon4},
\begin{equation}\label{zon12}
C_{L,k}:=\frac{8\pi^2}{2k+1}\int_{-1}^1F_k(\alpha)L(\alpha)\,d\alpha.
\end{equation}
\end{lemma}

\begin{proof} We start from the identity \eqref{zon3} and write
\begin{equation*}
\int_{\SS^2}L(x\cdot y)Y(y)\,d\mu(y)=\int_{\SS^2}Y(p)\Big[\int_{\SS^2}L(x\cdot y)Z_k(y,p)\,d\mu(y)\Big]d\mu(p).
\end{equation*}
Let
\begin{equation*}
G(x,p):=\int_{\SS^2}L(x\cdot y)Z_k(y,p)\,d\mu(y).
\end{equation*}
Notice that $G(x,.)\in\mathcal{H}_k$ for any $x\in\SS^2$ and $G(\rho x,\rho p)=G(x,p)$ for any $x,p\in\SS^2$ and any rotation $\rho\in\mathbb{SO}(3)$. Therefore, using Lemma \ref{zon2} (ii) $G(x,p)=C_{L,k}Z_k(x,p)$. The constant $C_{L,k}$ can be calculated by setting $x=p=e_3$, so
\begin{equation*}
C_{L,k}Z_k(e_3,e_3)=\int_{\SS^2}L(y\cdot e_3)Z_k(y,e_3)\,d\mu(y)=\int_0^\pi\int_0^{2\pi}F_k(\cos\theta)L(\cos\theta)\sin\theta\,d\theta d\phi,
\end{equation*}
where $F_k$ is as in \eqref{zon4}. The desired conclusions follow since $Z_k(e_3,e_3)=(2k+1)/(4\pi)$. 
\end{proof}

We now prove a formula on the action of multiplication operators on spherical harmonics.

\begin{lemma}\label{zon20}
Assume that $j\in\{1,2,3\}$, $k\in\Z_+$, and $Y_k\in \mathcal{H}_k$, and let $\widetilde{Y}_k\in\mathcal{A}_k$ denote the homogeneous harmonic extension of $Y_k$. Then
\begin{equation}\label{zon21}
\begin{split}
&x^j\widetilde{Y}_k(x)=\widetilde{Y}_{k+1}(x)+|x|^2\widetilde{Y}_{k-1}(x),\\
&\widetilde{Y}_{k+1}:=x^j\widetilde{Y}_k(x)-|x|^2\partial_j\widetilde{Y}_k/(2k+1)\in\mathcal{A}_{k+1},\\
&\widetilde{Y}_{k-1}:=\partial_j\widetilde{Y}_k/(2k+1)\in\mathcal{A}_{k-1}.
\end{split}
\end{equation}
In particular $x^jY_k\in\mathcal{H}_{k+1}+\mathcal{H}_{k-1}$.
\end{lemma}

\begin{proof} We calculate
\begin{equation*}
\Delta_{\R^3}[x^j\widetilde{Y}_k(x)]=2\partial_j\widetilde{Y}_k(x),
\end{equation*}
\begin{equation*}
\Delta_{\R^3}[|x|^2\partial_j\widetilde{Y}_k(x))]=6\partial_j\widetilde{Y}_k(x)+4x_l\partial_l\partial_j\widetilde{Y}_k(x)=(4k+2)\partial_j\widetilde{Y}_k(x),
\end{equation*}
since $x_l\partial_l\partial_j\widetilde{Y}_k(x)=(k-1)\partial_j\widetilde{Y}_k(x)$ due to homogeneity. In particular $\Delta_{\R^3}\widetilde{Y}_{k+1}(x)=0$, and the desired claims in \eqref{zon21} follow.
\end{proof}

\begin{proof}[Proof of Lemma \ref{PropExp}] {\bf{Step 1.}} Assume first that 
\begin{equation}\label{exp1}
\widetilde{L_1}(x,y)=|x-y|^{-1}=(2-2x\cdot y)^{-1/2}=L_1(x\cdot y),\qquad x,y\in\SS^2.
\end{equation}
For $k\in\Z_+$, we would like to calculate the coefficients $C_{L_1,k}$ in \eqref{zon12},
\begin{equation}\label{exp2}
C_{L_1,k}=\frac{8\pi^2}{2k+1}\int_{-1}^1F_k(\alpha)L_1(\alpha)\,d\alpha=\frac{8\pi^2}{2k+1}\int_{-1}^1F_k(\alpha)(2-2\alpha)^{-1/2}\,d\alpha.
\end{equation}
To use the formula \eqref{zon7}, we calculate the full series (using the change of variables $\alpha=1-2\beta^2$)
\begin{equation*}
\begin{split}
I(r)&:=\int_{-1}^1(1-r^2)(1+r^2-2r \alpha)^{-3/2}(2-2\alpha)^{-1/2}\,d\alpha\\
&=(1-r^2)\int_0^12[(1-r)^2+4r\beta^2]^{-3/2}\,d\beta\\
&=2(1-r^2)\frac{\beta[(1-r)^2+4r\beta^2]^{-1/2}}{(1-r)^2}\Big|_0^1\\
&=2(1-r)^{-1}.
\end{split}
\end{equation*}
Therefore $C_{L_1,k}=4\pi/(2k+1)$ and the desired identity \eqref{exp21} follows.

{\bf{Step 2.}} Assume now that 
\begin{equation}\label{exp4}
\widetilde{L_{2,\epsilon}^{ij}}(x,y)=(x^i-y^i)(x^j-y^j)(\epsilon+|x-y|^2)^{-3/2},\qquad x,y\in\SS^2,
\end{equation}
where $i,j\in\{1,2,3\}$ and $\epsilon\geq 0$.\footnote{We are interested mainly in the case $\epsilon=0$, but some of the algebraic manipulations (for example the convergence of the integrals in \eqref{exp9}) require $\epsilon>0$.} For any $k\in\Z_+$ and $Y_k\in\mathcal{H}_k$ we would like to calculate
\begin{equation}\label{exp5}
T^{ij}_\epsilon(Y_k)(x):=\int_{\SS^2}\widetilde{L_{2,\epsilon}^{ij}}(x,y)Y_k(y)\,d\mu(y).
\end{equation}
The formula \eqref{zon11} does not apply directly in this case, because the kernels $\widetilde{L_{2,\epsilon}^{ij}}$ are not of the assumed form, and we need to use Lemma \ref{zon20} as well. Let $\widetilde{Y}_k\in\mathcal{A}_k$ denote the harmonic extension of $Y_k$, and let 
\begin{equation}\label{exp6}
\begin{split}
&\widetilde{Y}_{k+1}^i(y):=y^i\widetilde{Y}_k(y)-|y|^2\partial_i\widetilde{Y}_k(y)/(2k+1),\qquad\,\widetilde{Y}_{k-1}^i(y):=\partial_i\widetilde{Y}_k(y)/(2k+1),\\
&\widetilde{Y}_{k+1}^j(y):=y^j\widetilde{Y}_k(y)-|y|^2\partial_j\widetilde{Y}_k(y)/(2k+1),\qquad \widetilde{Y}_{k-1}^j(y):=\partial_j\widetilde{Y}_k(y)/(2k+1),
\end{split}
\end{equation}
\begin{equation}\label{exp7}
\begin{split}
&\widetilde{Y}_{k+2}^{ij}(y):=y^iy^j\widetilde{Y}_k(y)-\frac{\delta_{ij}|y|^2\widetilde{Y}_k(y)}{2k+3}-\frac{|y|^2(y^i\partial_j\widetilde{Y}_k(y)+y^j\partial_i\widetilde{Y}_k(y))}{2k+3}+\frac{|y|^4\partial_i\partial_j\widetilde{Y}_k(y)}{(2k+1)(2k+3)},\\
&\widetilde{Y}_{k}^{ij}(y):=\frac{\delta_{ij}\widetilde{Y}_k(y)}{2k+3}+\frac{y^i\partial_j\widetilde{Y}_k(y)+y^j\partial_i\widetilde{Y}_k(y)}{2k+3}-\frac{2|y|^2\partial_i\partial_j\widetilde{Y}_k(y)}{(2k-1)(2k+3)},\\
&\widetilde{Y}_{k-2}^{ij}(y):=\frac{\partial_i\partial_j\widetilde{Y}_k(y)}{(2k+1)(2k-1)}.
\end{split}
\end{equation}
Notice that $\widetilde{Y}_{k+1}^i, \widetilde{Y}_{k+1}^j\in\mathcal{A}_{k+1}$, $\widetilde{Y}_{k-1}^i, \widetilde{Y}_{k-1}^j\in\mathcal{A}_{k-1}$, $\widetilde{Y}_{k+2}^{ij}\in\mathcal{A}_{k+2}$, $\widetilde{Y}_k^{ij}\in\mathcal{A}_{k}$, $\widetilde{Y}_{k-2}^{ij}\in\mathcal{A}_{k-2}$. Moreover, if $y\in\SS^2$ then
\begin{equation}\label{exp8}
\begin{split}
y^i\widetilde{Y}_k(y)&=\widetilde{Y}_{k+1}^i(y)+\widetilde{Y}_{k-1}^i(y),\\
y^j\widetilde{Y}_k(y)&=\widetilde{Y}_{k+1}^j(y)+\widetilde{Y}_{k-1}^j(y),\\
y^iy^j\widetilde{Y}_k(y)&=\widetilde{Y}_{k+2}^{ij}(y)+\widetilde{Y}_k^{ij}(y)+\widetilde{Y}_{k-2}^{ij}(y).
\end{split}
\end{equation}

Let $Y_{k+1}^i,Y_{k-1}^i,Y_{k+1}^j,Y_{k-1}^j,Y_{k+2}^{ij},Y_k^{ij},Y_{k-2}^{ij}$ denote the restrictions of the harmonic functions $\widetilde{Y}_{k+1}^i,\widetilde{Y}_{k-1}^i,\widetilde{Y}_{k+1}^j,\widetilde{Y}_{k-1}^j,\widetilde{Y}_{k+2}^{ij},\widetilde{Y}_k^{ij},\widetilde{Y}_{k-2}^{ij}$ to the sphere $\SS^2$. We use now the definitions \eqref{exp4} and \eqref{exp5} together with \eqref{exp8} and Lemma \ref{zon10} to calculate, for any $\epsilon>0$,
\begin{equation*}
\begin{split}
&T^{ij}_\epsilon(Y_k)(x)=x^ix^j\int_{\SS^2}(\epsilon+2-2x\cdot y)^{-3/2}Y_k(y)\,d\mu(y)\\
&-x^i\int_{\SS^2}(\epsilon+2-2x\cdot y)^{-3/2}y^jY_k(y)\,d\mu(y)\\
&-x^j\int_{\SS^2}(\epsilon+2-2x\cdot y)^{-3/2}y^iY_k(y)\,d\mu(y)+\int_{\SS^2}(\epsilon+2-2x\cdot y)^{-3/2}y^iy^jY_k(y)\,d\mu(y)\\
&=x^ix^jY_k(x)B_{k,\epsilon}-\big[x^iY_{k+1}^j(x)+x^jY_{k+1}^i(x)\big]B_{k+1,\epsilon}-\big[x^iY_{k-1}^j(x)+x^jY_{k-1}^i(x)\big]B_{k-1,\epsilon}\\
&+Y_{k+2}^{ij}(x)B_{k+2,\epsilon}+Y_{k}^{ij}(x)B_{k,\epsilon}+Y_{k-2}^{ij}(x)B_{k-2,\epsilon}.
\end{split}
\end{equation*}
where
\begin{equation}\label{exp9}
B_{l,\epsilon}:=\frac{8\pi^2}{2l+1}\int_{-1}^1F_l(\alpha)(2+\epsilon-2\alpha)^{-3/2}\,d\alpha,\qquad l\in\Z_+.
\end{equation}
We can use the formulas \eqref{exp6}--\eqref{exp7} to express $T_\epsilon^{ij}(Y_k)$ in terms of the functions $\widetilde{Y}_k$, $\partial_i\widetilde{Y}_k$, $\partial_j\widetilde{Y}_k$, $\partial_i\partial_j\widetilde{Y}_k$, in the form
\begin{equation}\label{exp10}
\begin{split}
T^{ij}_\epsilon(Y_k)(x)&=x^ix^j\widetilde{Y}_k(x)\big[B_{k,\epsilon}-2B_{k+1,\epsilon}+B_{k+2,\epsilon}\big]-\delta_{ij}\widetilde{Y}_k(x)\frac{B_{k+2,\epsilon}-B_{k,\epsilon}}{2k+3}\\
&+\big[x^i\partial_j\widetilde{Y}_k(x)+x^j\partial_i\widetilde{Y}_k(x)\big]\Big[\frac{B_{k+1,\epsilon}-B_{k-1,\epsilon}}{2k+1}-\frac{B_{k+2,\epsilon}-B_{k,\epsilon}}{2k+3}\Big]\\
&+\partial_i\partial_j\widetilde{Y}_k(x)\frac{(2k-1)B_{k+2,\epsilon}-2(2k+1)B_{k,\epsilon}+(2k+3)B_{k-2,\epsilon}}{(2k-1)(2k+1)(2k+3)}.
\end{split}
\end{equation}

To calculate the integrals in \eqref{exp9} we use the simple formula
\begin{equation}\label{exp11}
\frac{d}{dx}\big[(ax+b)(Ax^2+2Bx+C)^{-1/2}\big]=\frac{(aB-bA)x+(aC-bB)}{(Ax^2+2Bx+C)^{3/2}},
\end{equation}
for any constants $a,b,A,B,C\in\R$. Using \eqref{exp9} and \eqref{zon7.6} we have 
\begin{equation}\label{exp12}
B_{l,\epsilon}=2\pi\int_{-1}^1P_l(\alpha)(2+\epsilon-2\alpha)^{-3/2}\,d\alpha,\qquad l\in\Z_+.
\end{equation}
We examine the defining formula \eqref{zon7.5} and calculate, using \eqref{exp11},
\begin{equation}\label{exp13}
\begin{split}
&\int_{-1}^1(1+r^2-2r\alpha)^{-1/2}(2+\epsilon-2\alpha)^{-3/2}\,d\alpha\\
&=\int_{-1}^1\frac{-2r\alpha+1+r^2}{[4r\alpha^2-2(1+(2+\epsilon)r+r^2)\alpha+(2+\epsilon)(1+r^2)]^{3/2}}\,d\alpha\\
&=\frac{a\alpha+b}{[4r\alpha^2-2(1+(2+\epsilon)r+r^2)\alpha+(2+\epsilon)(1+r^2)]^{1/2}}\Big|_{-1}^1\\
&=\frac{a+b}{\sqrt\epsilon(1-r)}-\frac{-a+b}{\sqrt{4+\epsilon}(1+r)},
\end{split}
\end{equation}
where
\begin{equation}\label{exp14}
a:=\frac{-2r}{1+r^2-(2+\epsilon)r},\qquad b:=\frac{1+r^2}{1+r^2-(2+\epsilon)r}.
\end{equation}
Using also \eqref{exp12} we derive the identity
\begin{equation}\label{exp15}
\frac{1}{2\pi}\big[B_{0,\epsilon}+B_{1,\epsilon}r+B_{2,\epsilon}r^2+\ldots\big]=\frac{1-r}{\sqrt{\epsilon} [1+r^2-(2+\epsilon)r]}-\frac{1+r}{\sqrt{4+\epsilon} [1+r^2-(2+\epsilon)r]}.
\end{equation}
We multiply this identity by $1-r$, 
\begin{equation}\label{exp16}
\begin{split}
\frac{1}{2\pi}\big[B_{0,\epsilon}&+(B_{1,\epsilon}-B_{0,\epsilon})r+(B_{2,\epsilon}-B_{1,\epsilon})r^2+\ldots\big]\\
&=\frac{1}{\sqrt{\epsilon}}+\frac{\sqrt{\epsilon}r}{1+r^2-(2+\epsilon)r}-\frac{1-r^2}{\sqrt{4+\epsilon} [1+r^2-(2+\epsilon)r]},
\end{split}
\end{equation}
and let $\epsilon\to 0$ to derive the identities
\begin{equation}\label{exp17}
\lim_{\epsilon\to 0}B_{l+1,\epsilon}-B_{l,\epsilon}=-2\pi\qquad \text{ for any }l\in\Z_+.
\end{equation}
Using the formula \eqref{exp10} it follows that
\begin{equation}\label{exp18}
\frac{\lim_{\epsilon\to 0}T^{ij}_\epsilon(Y_k)x}{2\pi}=\frac{2\delta_{ij}\widetilde{Y}_k(x)}{2k+3}-\frac{4\big[x^i\partial_j\widetilde{Y}_k(x)+x^j\partial_i\widetilde{Y}_k(x)\big]}{(2k+1)(2k+3)}+\frac{8\partial_i\partial_j\widetilde{Y}_k(x)}{(2k-1)(2k+1)(2k+3)}.
\end{equation}

{\bf{Step 3.}} The identities \eqref{exp22} follow from \eqref{exp18}. To prove \eqref{exp22.5} we define
\begin{equation}\label{exp30}
R^i_\epsilon(Y)(x):=\int_{\SS^2}\frac{x^i-p^i}{(\epsilon+2-2x\cdot p)^{3/2}}(Y(x)-Y(p))\,d\mu(p).
\end{equation}
Then, using Lemmas \ref{zon10} and \ref{zon20},
\begin{equation*}
\begin{split}
&R^i_\epsilon(Y_k)(x)=x^iY_k(x)\int_{\SS^2}(\epsilon+2-2x\cdot y)^{-3/2}\,d\mu(y)-Y_k(x)\int_{\SS^2}(\epsilon+2-2x\cdot y)^{-3/2}y^i\,d\mu(y)\\
&-x^i\int_{\SS^2}(\epsilon+2-2x\cdot y)^{-3/2}Y_k(y)\,d\mu(y)+\int_{\SS^2}(\epsilon+2-2x\cdot y)^{-3/2}y^iY_k(y)\,d\mu(y)\\
&=x^i\widetilde{Y}_k(x)B_{0,\epsilon}-x^i\widetilde{Y}_k(x)B_{1,\epsilon}-x^i\widetilde{Y}_k(x)B_{k,\epsilon}+\widetilde{Y}_{k+1}^i(x)B_{k+1,\epsilon}+\widetilde{Y}_{k-1}^i(x)B_{k-1,\epsilon},
\end{split}
\end{equation*}
where $\widetilde{Y}_{k+1}^i,\widetilde{Y}_{k-1}^i$ are defined as in \eqref{exp6} and $B_{l,\epsilon}$ are defined as in \eqref{exp9}. The desired identities \eqref{exp22.5} follow using \eqref{exp17} and letting $\epsilon\to 0$.

\end{proof}

\subsection{Vector spherical harmonics and diagonalization of the linear operator}\label{vec1}

We decompose $Y=\sum_{k\geq 0}Y_k$, $Y^j=\sum_{k\geq 0}Y^j_k$, $j\in\{1,2,3\}$, where $Y^j_k\in\mathcal{H}_k$ are scalar spherical harmonics. Then we calculate, using \eqref{Hes7} and \eqref{exp21}-\eqref{exp22},
\begin{equation}\label{Pes7}
\begin{split}
\mathcal{N}^i_{1,1}&(Y_k)(x)=\frac{-k(k+1)}{8\pi}\Big\{\frac{4\pi\widetilde{Y}^i_k(x)}{2k+1}+\frac{4\pi\widetilde{Y}^i_k(x)}{2k+3}\\
&\qquad\qquad\qquad-\frac{16\pi}{(2k+1)(2k+3)}\Big[\frac{x^i\partial_j\widetilde{Y}_k^j(x)+x^j\partial_i\widetilde{Y}_k^j(x)}{2}-\frac{|x|^2\partial_i\partial_j\widetilde{Y}_k^j(x)}{2k-1}\Big]\Big\}\\
&=\frac{-2k(k+1)}{(2k+1)(2k+3)}\Big\{(k+1)\widetilde{Y}^i_k(x)-\Big[\frac{x^i\partial_j\widetilde{Y}_k^j(x)+x^j\partial_i\widetilde{Y}_k^j(x)}{2}-\frac{|x|^2\partial_i\partial_j\widetilde{Y}_k^j(x)}{2k-1}\Big]\Big\}.
\end{split}
\end{equation}
Moreover, using the formula \eqref{Hes8} and the definitions \eqref{exp20},
\begin{equation*}
\begin{split}
\mathcal{N}^i_{1,2}(Y_k)(x)&=\frac{1}{4\pi}\big[x^iR^l(Y_k^l)(x)-x^lR^i(Y_k^l)(x)+(1/2)Y^i_k(x)S(1)(x)\\
&-(1/2)S(Y_k^i)(x)-(3/2)Y_k^l(x)T^{il}(1)(x)+(3/2)T^{il}(Y_k^l)(x)\big].
\end{split}
\end{equation*}
Using the formulas \eqref{exp21}--\eqref{exp22.5} we obtain
\begin{equation}\label{Pes8}
\begin{split}
\mathcal{N}^i_{1,2}(Y_k)(x)&=\frac{x^i\partial_l\widetilde{Y}_k^l(x)}{2k+1}-\frac{x^l\partial_i\widetilde{Y}_k^l(x)}{2k+1}+\frac{\widetilde{Y}_k^i(x)}{2}-\frac{\widetilde{Y}_k^i(x)}{2(2k+1)}-\frac{\widetilde{Y}_k^i(x)}{2}\\
&+\frac{3\widetilde{Y}_k^i(x)}{2(2k+3)}-\frac{6}{(2k+1)(2k+3)}\Big\{\frac{x^i\partial_l\widetilde{Y}^l_k(x)+x^l\partial_i\widetilde{Y}^l_k(x)}{2}-\frac{|x|^2\partial_i\partial_l\widetilde{Y}^l_k(x)}{2k-1}\Big\}\\
&=\frac{x^i\partial_l\widetilde{Y}_k^l(x)-x^l\partial_i\widetilde{Y}_k^l(x)}{2k+1}+\frac{2k\widetilde{Y}_k^i(x)}{(2k+1)(2k+3)}\\
&-\frac{6}{(2k+1)(2k+3)}\Big\{\frac{x^i\partial_l\widetilde{Y}^l_k(x)+x^l\partial_i\widetilde{Y}^l_k(x)}{2}-\frac{|x|^2\partial_i\partial_l\widetilde{Y}^l_k(x)}{2k-1}\Big\}.
\end{split}
\end{equation}
We combine now \eqref{Pes7} and \eqref{Pes8} to derive the formula
\begin{equation}\label{Pes9}
\begin{split}
\mathcal{N}^i_{1}(Y_k)(x)&=\frac{-2k^2(k+2)\widetilde{Y}_k^i(x)}{(2k+1)(2k+3)}+\frac{x^i\partial_l\widetilde{Y}_k^l(x)-x^l\partial_i\widetilde{Y}_k^l(x)}{2k+1}\\
&+\frac{2k^2+2k-6}{(2k+1)(2k+3)}\Big\{\frac{x^i\partial_l\widetilde{Y}^l_k(x)+x^l\partial_i\widetilde{Y}^l_k(x)}{2}-\frac{|x|^2\partial_i\partial_l\widetilde{Y}^l_k(x)}{2k-1}\Big\}.
\end{split}
\end{equation}
Notice that $\mathcal{N}_1^i(Y_0)\equiv 0$ if $k=0$.

We would like to create an orthogonal  3-dimensional basis of spherical harmonics in such a way that the linear operator $\mathcal{N}_{1}$ is diagonalized in this basis. For this we use a suitable choice of vector spherical harmonics. 
Given a scalar spherical harmonic $h_k\in\mathcal{H}_k$, $k\in\Z_+$, let $\widetilde{h}_k\in\mathcal{A}_k$ denote its homogeneous harmonic extension. We define the functions $\H_{k+1}^1,\H_{k-1}^2,\H_{k}^3:\R^3\to\R^3$ by the formulas
\begin{equation}\label{vec2}
\begin{split}
&\H_{k+1}^{1,i}(x):=(2k+1)x^i\widetilde{h}_k(x)-|x|^2\partial_i\widetilde{h}_k(x),\\
&\H_{k-1}^{2,i}(x):=\partial_i\widetilde{h}_k(x),\\
&\H_{k}^{3,i}(x):=\in_{imn}x^m\partial_n\widetilde{h}_k(x),
\end{split}
\end{equation}
where $\in_{imn}$ denotes the Levi-Civita symbol.
It is easy to see that these functions are homogeneous spherical harmonics, $\H_{k+1}^1\in\mathcal{A}_{k+1}^3$, $\H_{k-1}^2\in\mathcal{A}_{k-1}^3$, $\H_{k}^3\in\mathcal{A}_{k}^3$. 

We would like to calculate $\mathcal{N}_1(\H_{k+1}^1)$, $\mathcal{N}_1(\H_{k-1}^2)$, and $\mathcal{N}_1(\H_{k}^3)$, where $\mathcal{N}_1$ is the linear part of the Peskin nonlinearity, as defined in \eqref{Pes6}. For this we would like to use the formula \eqref{Pes9}, applied to the vector-valued homogeneous solid harmonics  $\H_{k+1}^1,\H_{k-1}^2,\H_{k}^3$. Notice that we can rewrite \eqref{Pes9} as
\begin{equation}\label{vec20}
\begin{split}
\mathcal{N}^i_{1}(Y_k)(x)&=\frac{-2k^2(k+2)\widetilde{Y}_k^i(x)}{(2k+1)(2k+3)}+\frac{x^i\partial_l\widetilde{Y}_k^l(x)-\partial_i(x^l\widetilde{Y}_k^l(x))+\widetilde{Y}_k^i(x)}{2k+1}\\
&+\frac{2k^2+2k-6}{(2k+1)(2k+3)}\Big\{\frac{x^i\partial_l\widetilde{Y}^l_k(x)+\partial_i(x^l\widetilde{Y}^l_k(x))-\widetilde{Y}_k^i(x)}{2}-\frac{|x|^2\partial_i\partial_l\widetilde{Y}^l_k(x)}{2k-1}\Big\},
\end{split}
\end{equation}
 for any spherical harmonic $Y_k\in\mathcal{H}_k^3$, where $\widetilde{Y}_k\in\mathcal{A}_k^3$ denotes the homogeneous harmonic extension of degree $k$. In our case, we calculate, using the definitions \eqref{vec2},
\begin{equation}\label{vec21}
x^l\H_{k+1}^{1,l}(x)=(k+1)|x|^2\widetilde{h}_k(x),\qquad x^l\H_{k-1}^{2,l}(x)=k\widetilde{h}_k(x), \qquad x^l\H_{k}^{3,l}(x)=0,
\end{equation}
and
\begin{equation}\label{vec22}
\partial_l\H_{k+1}^{1,l}(x)=(k+1)(2k+3)\widetilde{h}_k(x),\qquad \partial_l\H_{k-1}^{2,l}(x)=0, \qquad \partial_l\H_{k}^{3,l}(x)=0.
\end{equation}

Therefore we calculate
\begin{equation}\label{vec27}
\mathcal{N}^i_{1}(\H_{k-1}^{2})(x)=-\frac{(k-1)(k+1)(k+2)}{(2k-1)(2k+1)}\H_{k-1}^{2,i}(x),
\end{equation}
\begin{equation}\label{vec28}
\mathcal{N}^i_{1}(\H_{k}^{3})(x)=-\frac{(k-1)(k+2)}{2k+1}\H_{k}^{3,i}(x).
\end{equation}
Since $\partial_i\widetilde{h}_k(x)=\H_{k-1}^{2,i}$ and $(2k+1)x^i\widetilde{h}_k(x)=\H_{k+1}^{1,i}+|x|^2\H_{k-1}^{2,i}$, we calculate
\begin{equation}\label{vec29}
\begin{split}
&x^i\partial_l{\bf H}_{k+1}^{1,l}(x)-\partial_i(x^l{\bf H}_{k+1}^{1,l}(x))+{\bf H}_{k+1}^{1,i}(x)=(k+2){\bf H}_{k+1}^{1,i}(x),\\
&x^i\partial_l{\bf H}_{k+1}^{1,l}(x)+\partial_i(x^l{\bf H}_{k+1}^{1,l}(x))-{\bf H}_{k+1}^{1,i}(x)\\
&\qquad\qquad\qquad\qquad=\frac{2k^2+5k+4}{2k+1}{\bf H}_{k+1}^{1,i}(x)+\frac{(k+1)(4k+6)}{2k+1}|x|^2{\bf H}_{k-1}^{2,i}(x).
\end{split}
\end{equation}
\begin{equation}\label{vec30}
\mathcal{N}^i_{1}(\H_{k+1}^{1})(x)=-\frac{k(k-1)(k+2)}{(2k+1)(2k+3)}\H_{k+1}^{1,i}(x).
\end{equation}
The main identities \eqref{vec27}, \eqref{vec28}, and \eqref{vec30} provide a complete diagonalization of the linear operator $\mathcal{N}_1$. More precisely, we have the following proposition:

\begin{proposition}\label{linearSummary}
For any $k\in\Z_+$ we fix an orthogonal basis $h_{k,m}$, $m\in\{-k,\ldots,k\}$, of the space of spherical harmonics $\mathcal{H}_k$ and define the associated functions ${\bf H}_{k+1,m}^1,{\bf H}_{k-1,m}^2,{\bf H}_{k,m}^3$ as in \eqref{vec2} (if $k=0$ then ${\bf H}_{k-1,m}^2={\bf H}_{k,m}^3=0$, so only the function ${\bf H}_{k+1,m}^1={\bf H}_{1,0}^1$ is relevant). Then 

(i) $\H_{k+1,m}^1\in\mathcal{A}_{k+1}^3$, $\H_{k-1,m}^2\in\mathcal{A}_{k-1}^3$, $\H_{k,m}^3\in\mathcal{A}_{k}^3$ for any $k\in\Z_+$ and $m\in\{-k,\ldots,k\}$;

(ii) The vector-valued functions $\big\{{\bf H}_{1,0}^1, {\bf H}_{k+1,m}^1,{\bf H}_{k-1,m}^2,{\bf H}_{k,m}^3:\,k\in\{1,\ldots\}, m\in\{-k,\ldots,k\}\big\}$ when restricted to the sphere $\SS^2$ form an orthogonal basis of the space of functions $L^2(\SS^2:\R^3)$;

(iii) The linear operator $\mathcal{N}_1$ is diagonal in this basis,
\begin{equation}\label{vec31}
\begin{split}
\mathcal{N}_{1}(\H_{k+1,m}^{1})&=-a_k\H_{k+1,m}^{1},\\
\mathcal{N}_{1}(\H_{k-1,m}^{2})&=-b_k\H_{k-1,m}^{2},\\
\mathcal{N}_{1}(\H_{k,m}^{3})&=-c_k\H_{k,m}^{3},
\end{split}
\end{equation}
where the eigenvalues are
\begin{equation}\label{tete4}
a_k:=\frac{k(k-1)(k+2)}{(2k+1)(2k+3)},\qquad b_k:=\frac{(k-1)(k+1)(k+2)}{(2k-1)(2k+1)},\qquad c_k:=\frac{(k-1)(k+2)}{(2k+1)}.
\end{equation}
\end{proposition}

\begin{proof} Part (i) follows easily from the definition \eqref{vec2} and part (iii) follows from our calculations above, see the formulas \eqref{vec27}, \eqref{vec28}, \eqref{vec30}. 

To prove part (ii) it suffices to show that for any $k\geq 1$ the functions $\{\H^1_{k,m_1}:\, |m_1|\leq k-1\}$, $\{\H^2_{k,m_2}:\, |m_2|\leq k+1\}$, and $\{\H^3_{k,m_3}:\, |m_3|\leq k\}$ form a basis of the space of solid spherical harmonics $\mathcal{A}^3_k$. The total number of these functions is $(2k-1)+(2k+1)+(2k+3)=6k+3$, the same as the dimension of the space $\mathcal{A}^3_k$. 
Therefore it suffice to prove that all these functions when restricted to $\mathbb{S}^2$ are pairwise orthogonal to each other.
We first notice that, on $\SS^2$,
\begin{equation*}
\begin{aligned}
        \nabla \tilde{h}_l(x)=\nabla_{\SS^2}h_l(x)+lxh_l(x).
    \end{aligned}
\end{equation*}
Therefore, on $\SS^2$ we can write
\begin{equation*}
\begin{aligned}
        \H_{k,m}^1(x)&=kxh_{k-1,m}(x)-\nabla_{\SS^2}h_{k-1,m}(x),\\
        \H_{k,m}^2(x)&=(k+1)xh_{k+1,m}(x)+\nabla_{\SS^2}h_{k+1,m}(x),\\
        \H_{k,m}^3(x)&=x\wedge \nabla_{\SS^2}h_{k,m}(x).
    \end{aligned}
\end{equation*}
We will repeatedly use that
\begin{equation*}
    \begin{aligned}
        \int_{\SS^2}\nabla_{\SS^2}h_{l,m}\cdot\bar{\nabla_{\SS^2}h_{l',m'}}d\mu=l(l+1)\int_{\SS^2}h_{l,m}\cdot\bar{h_{l',m'}}d\mu,
    \end{aligned}
\end{equation*}
which follows by integration by parts and \eqref{coo7}. In particular, by the orthogonality assumption on the $\{h_{k,m}\}$ basis, the integral above vanishes if either $l\neq l'$ or if $l=l'$ but $m\neq m'$.

{\bf{Step 1.}} We first show the orthogonality inside each family. When $k\neq k'$, the orthogonality simply follows from \eqref{inr5}.
For $\H_{k,m}^1$, we have
\begin{equation*}
    \begin{aligned}
        \int_{\SS^2}\H_{k,m}^1\cdot\bar{\H_{k,m'}^1}d\mu&=k^2\int_{\SS^2}h_{k-1,m}\bar{h_{k-1,m'}}d\mu+\int_{\SS^2}\nabla_{\SS^2}h_{k-1,m}\cdot\bar{\nabla_{\SS^2}h_{k-1,m'}}d\mu,
    \end{aligned}
\end{equation*}
where we use that $x\cdot\nabla_{\SS^2}h_l=0$.
Then, integration by parts followed by \eqref{coo7} gives 
\begin{equation*}
    \begin{aligned}
        \int_{\SS^2}\H_{k,m}^1\cdot\bar{\H_{k,m'}^1}d\mu&=\big(k^2+k(k-1)\big)\int_{\SS^2}h_{k-1,m}\bar{h_{k-1,m'}}d\mu,
    \end{aligned}
\end{equation*}
thus the functions $\H_{k,m}^1$ are mutually orthogonal.
Similarly, 
\begin{equation*}
    \begin{aligned}
        \int_{\SS^2}\H_{k,m}^2\cdot\bar{\H_{k,m'}^2}d\mu&=\big((k+1)^2+(k+1)(k+2)\big)\int_{\SS^2}h_{k+1,m}\bar{h_{k+1,m'}}d\mu.
    \end{aligned}
\end{equation*}
Finally, we use that $\nabla_{\SS^2}h_{k,m}$ is tangent to the sphere and that for any $a,b$ orthogonal to $x\in\SS^2$,
\begin{equation*}
    (x\wedge a)\cdot\bar{x\wedge b}=(x\cdot x)(a\cdot\bar{b})-(x\cdot\bar{b})(x\cdot a)=a\cdot \bar{b}.
\end{equation*}
We get
\begin{equation*}
    \begin{aligned}
        \int_{\SS^2}\H_{k,m}^3\cdot\bar{\H_{k,m'}^3}d\mu&=\int_{\SS^2}(x\wedge \nabla_{\SS^2}h_{k,m})\cdot\bar{(x\wedge \nabla_{\SS^2}h_{k,m'})}d\mu=\int_{\SS^2}\nabla_{\SS^2}h_{k,m}\cdot\bar{\nabla_{\SS^2}h_{k,m'}}d\mu\\
        &=k(k+1)\int_{\SS^2}h_{k,m}\cdot\bar{h_{k,m'}}d\mu,
    \end{aligned}
\end{equation*}
thus the third family is also mutually orthogonal.

{\bf{Step 2.}} We next prove the orthogonality between the different families.
First, 
\begin{equation*}
    \begin{aligned}
        \int_{\SS^2}\H_{k+1,m}^1\cdot\bar{\H_{k'-1,m'}^2}d\mu&=\int_{\SS^2}
\left(
(k+1)h_{k,m}x-\nabla_{\SS^2}h_{k,m}
\right)
\cdot
\bar{
\left(
k' h_{k',m'}x+\nabla_{\SS^2}h_{k',m'}
\right)}
\,d\mu .
    \end{aligned}
\end{equation*}
The mixed normal-tangential terms vanish pointwise because
$x\cdot \nabla_{\SS^2}h_{k,m}=0$, $x\cdot \nabla_{\SS^2}h_{k',m'}=0$.
Hence
\begin{equation*}
    \begin{aligned}
    \int_{\SS^2}\H_{k+1,m}^1\cdot\bar{\H_{k'-1,m'}^2}d\mu&=
(k+1)k'
\int_{\SS^2}
h_{k,m}\bar{h_{k',m'}}\,d\mu-
\int_{\SS^2}
\nabla_{\SS^2}h_{k,m}
\cdot
\bar{\nabla_{\SS^2}h_{k',m'}}\,d\mu.    
    \end{aligned}
\end{equation*}
By integration by parts and the eigenvalue equation \eqref{coo7},
we have
\begin{equation*}
    \begin{aligned}
    \int_{\SS^2}\H_{k+1,m}^1\cdot\bar{\H_{k'-1,m'}^2}d\mu
&=
\big[(k+1)k'-k'(k'+1)\big]
\int_{\SS^2}
h_{k,m}\bar{h_{k',m'}}\,d\mu
\\
&=
k'(k-k')
\int_{\SS^2}
h_{k,m}\bar{h_{k',m'}}\,d\mu .
    \end{aligned}
\end{equation*}
If $k\neq k'$, the scalar product vanishes by orthogonality of scalar
spherical harmonics. If $k=k'$, the coefficient $k'(k-k')$ vanishes, thus we conclude the orthogonality result for the first and second family.

Next, consider the first and third families. We compute
\begin{equation*}
\begin{aligned}
\int_{\SS^2}\H^1_{k+1,m}\cdot\bar{
\mathbf H^3_{k',m'}}d\mu
&=
\int_{\SS^2}
\left(
(k+1)h_{k,m}x-\nabla_{\SS^2}h_{k,m}
\right)
\cdot
\bar{
\left(
x\times \nabla_{\SS^2}h_{k',m'}
\right)}
\,d\mu\\
&=-
\int_{\SS^2}
\nabla_{\SS^2}h_{k,m}
\cdot
\left(
x\times \bar{\nabla_{\SS^2}h_{k',m'}}
\right)
\,d\mu.
\end{aligned}
\end{equation*}
Since the tangent vector field
$x\times \nabla_{\SS^2}\overline{h_{k',m'}}$
is divergence-free on $\SS^2$, integration by parts gives that
\begin{equation*}
    \begin{aligned}
\int_{\SS^2}
\nabla_{\SS^2}h_{k,m}
\cdot
\left(
x\times \nabla_{\SS^2}\bar{h_{k',m'}}
\right)
\,d\mu
&=
-\int_{\SS^2}
h_{k,m}
\mathrm{div}_{\SS^2}
\left(
x\times \nabla_{\SS^2}\bar{h_{k',m'}}
\right)
\,d\mu=0.
\end{aligned}
\end{equation*}

Finally, consider the second and third families. Using the same arguments, we obtain the result
\begin{equation*}
\begin{aligned}
\int_{\SS^2}
\mathbf H^2_{k-1,m}\cdot
\bar{\mathbf H^3_{k',m'}}d\mu
&=
\int_{\SS^2}
\left(
k h_{k,m}x+\nabla_{\SS^2}h_{k,m}
\right)
\cdot
\overline{
\left(
x\times \nabla_{\SS^2}h_{k',m'}
\right)}
\,d\mu\\
&=
\int_{\SS^2}
\nabla_{\SS^2}h_{k,m}
\cdot
\left(
x\times \nabla_{\SS^2}\overline{h_{k',m'}}
\right)
d\mu\\
&=0.
\end{aligned}
\end{equation*}

\end{proof}

It follows from Proposition \ref{linearSummary} that the operator $\mathcal{N}_1$ has only non-positive eigenvalues and a non-trivial kernel, spanned by the functions $\H_{2,m}^{1},\H_{0,m}^{2},\H_{1,m}^{3}$, $m\in\{-1,0,1\}$, (coming from the index $k=1$) and $\H^1_{1,0}$ (coming from the index $k=0$). We can calculate these functions explicitly, for example if we start with the basis of solid harmonics
\begin{equation}\label{vec32}
\widetilde{h}_{0,0}(x)=1,\qquad \widetilde{h}_{1,-1}(x)=x^1,\qquad \widetilde{h}_{1,0}(x)=x^2,\qquad \widetilde{h}_{1,1}(x)=x^3.
\end{equation}
Then
\begin{equation}\label{vec33}
\begin{split}
\H_{1,0}^1(x)&={}^t(x^1,x^2,x^3),\\
\H_{2,-1}^1(x)&={}^t(2(x^1)^2-(x^2)^2-(x^3)^2,3x^1x^2,3x^1x^3),\\
\H_{2,0}^1(x)&={}^t(3x^1x^2,2(x^2)^2-(x^1)^2-(x^3)^2,3x^2x^3),\\
\H_{2,1}^1(x)&={}^t(3x^1x^3,3x^2x^3,2(x^3)^2-(x^1)^2-(x^2)^2),
\end{split}
\end{equation}
\begin{equation}\label{vec34}
\begin{split}
\H_{0,-1}^2(x)={}^t(1,0,0),\qquad\H_{0,0}^2(x)={}^t(0,1,0),\qquad\H_{0,1}^2(x)={}^t(0,0,1),
\end{split}
\end{equation}
\begin{equation}\label{vec35}
\begin{split}
\H_{1,-1}^3(x)={}^t(0,x^3,-x^2),\qquad \H_{1,0}^3(x)={}^t(-x^3,0,x^1),\qquad \H_{1,1}^3(x)={}^t(x^2,-x^1,0).
\end{split}
\end{equation}

\begin{lemma}[Kernel and Steady States] \label{LemSteadyKernel}
	The vector spherical harmonics \eqref{vec33}--\eqref{vec35} span the $10$-dimensional kernel of the linearized operator $\mathcal{N}_1$. These basis fields generate the exact manifold of spherical steady states:  dilations, translations, and the $\mathbb{SO}^+(3,1)$ conformal symmetries.
\end{lemma}

\begin{proof}
	By Proposition \ref{linearSummary}, the linear operator $\mathcal{N}_1$ is completely diagonalized by the orthogonal basis of vector spherical harmonics $\mathbf{H}^1_{k+1,m}$, $\mathbf{H}^2_{k-1,m}$, and $\mathbf{H}^3_{k,m}$. To find the kernel of $\mathcal{N}_1$, we seek all modes $(k,m)$ for which the corresponding eigenvalues vanish. 
	
	From the eigenvalue formulas established in \eqref{vec31}, we examine the roots of the respective multipliers
	\begin{itemize}
		\item For $\mathbf{H}_{k+1,m}^1$, the eigenvalue $-\frac{k(k-1)(k+2)}{(2k+1)(2k+3)}$ vanishes exactly when $k=0$ and $k=1$.
		\item For $\mathbf{H}_{k-1,m}^2$, the eigenvalue $-\frac{(k-1)(k+1)(k+2)}{(2k-1)(2k+1)}$ vanishes when $k=1$ (as $k \ge 1$ is required for this mode to be well-defined).
		\item For $\mathbf{H}_{k,m}^3$, the eigenvalue $-\frac{(k-1)(k+2)}{2k+1}$ vanishes when $k=1$.
	\end{itemize}
	
	Evaluating the dimensions, the $k=0$ mode yields $1$ dimension (restricted to $m=0$), while the $k=1$ modes yield $3$ dimensions each ($m \in \{-1, 0, 1\}$) across the three harmonic families. This generates exactly $1 + 9 = 10$ dimensions. 
	
	To prove these modes exactly represent the physical steady states, we compute the infinitesimal generators (the Lie algebra) of the steady-state transformations evaluated at the identity map $X _0(\mathbf{p}) = \mathbf{p}$. We map these generators to our vector spherical harmonic basis, which is defined for a solid harmonic $\widetilde{h}_k(X )$ at the surface $|\mathbf{p}|=1$ as
	\begin{equation}\label{VSH_Definitions}
	\mathbf{H}^1_{k+1}[\widetilde{h}_k](\mathbf{p}) = (2k+1)\mathbf{p}\widetilde{h}_k(\mathbf{p}) - \nabla \widetilde{h}_k(\mathbf{p}), \quad
	\mathbf{H}^2_{k-1}[\widetilde{h}_k](\mathbf{p}) = \nabla \widetilde{h}_k(\mathbf{p}), \quad
	\mathbf{H}^3_{k}[\widetilde{h}_k](\mathbf{p}) = \mathbf{p} \times \nabla \widetilde{h}_k(\mathbf{p}).
	\end{equation}
	
	\textbf{Case 1. Dilations ($1$ Dimension).}
	An infinitesimal uniform scaling of the sphere is given by $X _\epsilon(\mathbf{p}) = (1+\epsilon)\mathbf{p}$. Taking the derivative with respect to $\epsilon$ at the identity gives the purely radial generator
	\begin{equation}
	\mathbf{V}_{dil}(\mathbf{p}) = \left. \frac{d}{d\epsilon} \right|_{\epsilon=0} (1+\epsilon)\mathbf{p} = \mathbf{p}.
	\end{equation}
	Evaluating the $k=0$ mode with the constant solid harmonic $\widetilde{h}_0(X ) = 1$ (where $\nabla \widetilde{h}_0 = 0$), we obtain
	\begin{equation}
	\mathbf{H}^1_1[\widetilde{h}_0](\mathbf{p}) = (2(0)+1)\mathbf{p}(1) - 0 = \mathbf{p}.
	\end{equation}
	Thus, dilations are perfectly captured by the $\mathbf{H}^1_{1,0}$ mode.
	
	\textbf{Case 2. Translations ($3$ Dimensions).}
	An infinitesimal translation by a constant vector $\mathbf{a} \in \mathbb{R}^3$ is $X _\epsilon(\mathbf{p}) = \mathbf{p} + \epsilon \mathbf{a}$. The generator is the constant vector field
	\begin{equation}
	\mathbf{V}_{trans}(\mathbf{p}) = \left. \frac{d}{d\epsilon} \right|_{\epsilon=0} (\mathbf{p} + \epsilon \mathbf{a}) = \mathbf{a}.
	\end{equation}
	Using the $k=1$ solid harmonic $\widetilde{h}_1(X ) = \mathbf{a} \cdot X $, we find $\nabla \widetilde{h}_1(X ) = \mathbf{a}$. Applying the $\mathbf{H}^2$ formula yields
	\begin{equation}
	\mathbf{H}^2_0[\widetilde{h}_1](\mathbf{p}) = \nabla (\mathbf{a} \cdot \mathbf{p}) = \mathbf{a}.
	\end{equation}
	Hence, translations are completely encapsulated by the $k=1$ modes of $\mathbf{H}^2$.
	
	\textbf{Case 3. Conformal Symmetries of $\mathbb{SO}^+(3,1)$ ($6$ Dimensions).}
	The restricted Lorentz group acting projectively on the sphere generates $3$ rigid rotations and $3$ Lorentzian boosts (see Figure \ref{fig:boost} below).
	
	\textit{(i) Rotations (Toroidal Deformations).} An infinitesimal rotation around an axis $\boldsymbol{\omega}$ generates the vector field $\mathbf{V}_{rot}(\mathbf{p}) = \boldsymbol{\omega} \times \mathbf{p}$. 
	Selecting the $k=1$ solid harmonic $\widetilde{h}_1(X ) = -\boldsymbol{\omega} \cdot X $ such that $\nabla \widetilde{h}_1 = -\boldsymbol{\omega}$, the $\mathbf{H}^3$ formula yields
	\begin{equation}
	\mathbf{H}^3_1[\widetilde{h}_1](\mathbf{p}) = \mathbf{p} \times (-\boldsymbol{\omega}) = \boldsymbol{\omega} \times \mathbf{p}.
	\end{equation}
	This confirms that the tangential, divergence-free rotations map exactly to the $k=1$ modes of $\mathbf{H}^3$.
	
	\textit{(ii) Lorentzian Boosts (Poloidal Deformations).} A conformal boost in the direction of a vector $\mathbf{v}$ yields the generator $\mathbf{V}_{boost}(\mathbf{p}) = \mathbf{v} - (\mathbf{v} \cdot \mathbf{p})\mathbf{p}$.
	Using the solid harmonic $\widetilde{h}_1(X ) = \mathbf{v} \cdot X $ (where $\nabla \widetilde{h}_1 = \mathbf{v}$), we evaluate both the $\mathbf{H}^1$ and $\mathbf{H}^2$ modes for $k=1$
	\begin{equation}
	\mathbf{H}^1_2[\widetilde{h}_1](\mathbf{p}) = 3\mathbf{p}(\mathbf{v} \cdot \mathbf{p}) - \mathbf{v}, \quad \text{and} \quad \mathbf{H}^2_0[\widetilde{h}_1](\mathbf{p}) = \mathbf{v}.
	\end{equation}
	Constructing a precise linear combination of these two modes exactly recovers the boost generator
	\begin{equation}
	\frac{2}{3}\mathbf{H}^2_0[\widetilde{h}_1] - \frac{1}{3}\mathbf{H}^1_2[\widetilde{h}_1] = \frac{2}{3}\mathbf{v} - \left( \mathbf{p}(\mathbf{v} \cdot \mathbf{p}) - \frac{1}{3}\mathbf{v} \right) = \mathbf{v} - (\mathbf{v} \cdot \mathbf{p})\mathbf{p}.
	\end{equation}

	By explicitly reconstructing the infinitesimal generators of the entire steady-state family from the $k=0$ and $k=1$ vector spherical harmonics, we conclude that the $10$-dimensional kernel of $\mathcal{N}_1$ strictly coincides with the tangent space of the steady-state manifold evaluated at the identity.
    
\end{proof}

\vspace{-0.5cm}

\begin{figure}[h]
    \centering
\includegraphics[width=0.85\linewidth]{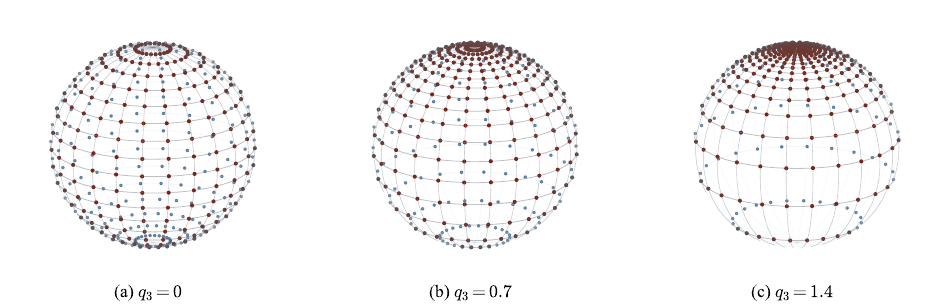}
    \caption{Lorentzian boost deformation}
    \label{fig:boost}
\end{figure}

\section{Nonlinear analysis and proof of the main theorem}\label{sec:nonlinear}

Letting $X(x)=x+Y(x)$, the nonlinearity \eqref{Pes1} can be written in the form
\begin{equation}\label{Pes20}
\mathcal{N}^i(x):=\int_{\SS^2}G_{ij}((x-p)+Y(x)-Y(p)))\Delta_{\SS^2}(p^j+Y^j(p))\,d\mu(p),
\end{equation}
where $i\in\{1,2,3\}$, $x\in\SS^2$. To understand the term $G_{ij}((x-p)+Y(x)-Y(p)))$ we start by writing
\begin{equation}\label{Pes21}
\begin{split}
&|(x-p)+(Y(x)-Y(p))|\\
&=\big[|x-p|^2+2(x^j-p^j)(Y^j(x)-Y^j(p))+(Y^j(x)-Y^j(p))(Y^j(x)-Y^j(p))\big]^{1/2}\\
&=|x-p|\Big[1+2\frac{(x^j-p^j)(Y^j(x)-Y^j(p))}{|x-p|^2}+\frac{(Y^j(x)-Y^j(p))(Y^j(x)-Y^j(p))}{|x-p|^2}\Big]^{1/2}\\
&=|x-p|\big[1+2\widetilde{Y}_{j,j}(x,p)+\widetilde{Y}_{k,j}(x,p)\widetilde{Y}_{k,j}(x,p)\big]^{1/2},
\end{split}
\end{equation}
where
\begin{equation}\label{Pes22}
\widetilde{Y}_{k,j}(x,p):=\frac{(x^k-p^k)(Y^j(x)-Y^j(p))}{|x-p|^2}.
\end{equation}
Therefore, recalling the definition \eqref{Intro2},
\begin{equation*}
\begin{split}
&G_{ij}((x-p)+(Y(x)-Y(p)))=\frac{1}{8\pi}\Big\{\frac{\delta_{ij}}{|x-p|\big[1+2\widetilde{Y}_{k,k}(x,p)+\widetilde{Y}_{k,l}(x,p)\widetilde{Y}_{k,l}(x,p)\big]^{1/2}}\\
&+\frac{[(x^i-p^i)+(Y^i(x)-Y^i(p)][(x^j-p^j)+(Y^j(x)-Y^j(p)]}{|x-p|^3\big[1+2\widetilde{Y}_{k,k}(x,p)+\widetilde{Y}_{k,l}(x,p)\widetilde{Y}_{k,l}(x,p)\big]^{3/2}}\Big\}\\
&=\frac{1}{8\pi|x-p|\big[1+2\widetilde{Y}_{k,k}(x,p)+\widetilde{Y}_{k,l}(x,p)\widetilde{Y}_{k,l}(x,p)\big]^{3/2}}\Big\{\delta_{ij}\big[1+2\widetilde{Y}_{k,k}(x,p)+\widetilde{Y}_{k,l}(x,p)\widetilde{Y}_{k,l}(x,p)\big]\\
&+\frac{[(x^i-p^i)+(Y^i(x)-Y^i(p)][(x^j-p^j)+(Y^j(x)-Y^j(p)]}{|x-p|^2}\Big\}\\
&=\frac{1}{8\pi|x-p|\big[1+2\widetilde{Y}_{k,k}(x,p)+\widetilde{Y}_{k,l}(x,p)\widetilde{Y}_{k,l}(x,p)\big]^{3/2}}\Big\{\delta_{ij}\big[1+2\widetilde{Y}_{k,k}(x,p)+\widetilde{Y}_{k,l}(x,p)\widetilde{Y}_{k,l}(x,p)\big]\\
&+\frac{(x^i-p^i)(x^j-p^j)}{|x-p|^2}+\widetilde{Y}_{i,j}(x,p)+\widetilde{Y}_{j,i}(x,p)+\widetilde{Y}_{k,i}(x,p)\widetilde{Y}_{k,j}(x,p)\Big\}.
\end{split}
\end{equation*}
Since $\Delta_{\SS^2}(p^j)=-2p^j$, our final formula of the Peskin nonlinearity is
\begin{equation}\label{Pes25}
\begin{split}
\mathcal{N}^i(x)&:=\frac{1}{8\pi}\int_{\SS^2}\frac{\big[1+2\widetilde{Y}_{k,k}(x,p)+\widetilde{Y}_{k,l}(x,p)\widetilde{Y}_{k,l}(x,p)\big]^{-3/2}}{|x-p|}[-2p^j+\Delta_{\SS^2}Y^j(p)]\\
&\times\Big\{\delta_{ij}+\frac{(x^i-p^i)(x^j-p^j)}{|x-p|^2}+\delta_{ij}\big[2\widetilde{Y}_{k,k}(x,p)+\widetilde{Y}_{k,l}(x,p)\widetilde{Y}_{k,l}(x,p)\big]+\\
&\quad\,\,\,+\widetilde{Y}_{i,j}(x,p)+\widetilde{Y}_{j,i}(x,p)+\widetilde{Y}_{k,i}(x,p)\widetilde{Y}_{k,j}(x,p)\Big\}\,d\mu(p).
\end{split}
\end{equation}
Therefore the Peskin nonlinearity can be thought of as an infinite polynomial expansion\footnote{The expansion converges as long as $|\widetilde{Y}_{k,j}(x,p)|\leq 1/4$ for any $k,j\in\{1,2,3\}$ and $x,p\in\mathbb{S}^2$.} of the functions $\Delta Y^j$ and $\widetilde{Y}_{k,j}$, with singular coefficients. 

To estimate the Peskin nonlinearity we recall the definition \eqref{Pes22} and consider multilinear operators of the form
\begin{equation}\label{Pes27}
\mathcal{T}_m(f;\underline{g})(x):=\int_{\SS^2}\Delta_{\SS^2}f(y)\cdot(g_1(x)-g_1(y))\cdot \ldots\cdot (g_m(x)-g_m(y)) K_m(x,y)\,d\mu(y),
\end{equation}
where $m\geq 1$, $f,g_1,\ldots,g_m:\mathbb{S}^2\to\mathbb{C}$ are suitable functions and the kernels $K_m=K_{m,m',\underline{k}}:\SS^2\times\SS^2\to\R$ are defined by
\begin{equation}\label{har2}
K_m(x,y):=\frac{(x^{k_1}-y^{k_1})\cdot\ldots\cdot(x^{k_{m'}}-y^{k_{m'}})}{|x-y|^{m+m'+1}},
\end{equation}
where $m'\leq m+2$ and $k_1,\ldots,k_{m'}\in\{1,2,3\}$. Our main nonlinear estimate is the following:

\begin{lemma}\label{LemmaNonlinMain}
With the notation in Theorem \ref{MainTHM}, there is a constant $C\geq 1$ such that if $m\geq 1$, $k_1,\ldots,k_m\in\{1,2,3\}$, and $f,g_1,\ldots,g_m\in Z[0,1]$ then
\begin{equation}\label{sif1}
\sup_{t\in[0,1]}\sup_{n\in\Z_+}(2^nt)^{\delta_0}\|\mathcal{P}'_n\mathcal{T}_{{\underline{k}}}(f;\underline{g})\|_{L^\infty}\leq C^m\|f\|_{Z[0,1]}\|g_1\|_{Z[0,1]}\cdot\ldots\cdot\|g_m\|_{Z[0,1]}.
\end{equation}
\end{lemma}

We will prove this lemma in subsection \ref{sif00} below. We show first how to use this lemma to prove part (i) of the main Theorem \ref{MainTHM}.

\begin{proof}[Proof of Theorem \ref{MainTHM}] (i) In terms of the variable $Y$, our main evolution equation \eqref{Intro1} can be written in the form
\begin{equation}\label{sif2}
\partial_tY^i=\mathcal{N}^i(Y)=\mathcal{N}^i_1(Y)+\mathcal{N}^i_{\geq 2}(Y),
\end{equation}
where $\mathcal{N}_1$ is the linear component (as defined in \eqref{Pes6}) and $\mathcal{N}_{\geq 2}:=\mathcal{N}-\mathcal{N}_1$. Letting $\mathcal{A}:=\mathcal{N}_1$ and recalling that $\mathcal{A}$ can be diagonalized and has only non-positive eigenvalues (see Proposition \ref{linearSummary}), the equation \eqref{sif2} can be written in integral (Duhamel) form as
\begin{equation}\label{sif3}
Y(t)=e^{t\mathcal{A}}Y_0+\int_0^te^{(t-s)\mathcal{A}}\mathcal{N}_{\geq 2}(Y)(s)\,ds.
\end{equation}
We construct the solution of the equation \eqref{sif3} using a fixed-point argument in the space $Z[0,1]$. 

{\bf{Step 1. }} 
We define the solution operator $\Phi: Z[0,1] \to Z[0,1]$ based on the right-hand side of the integral equation \eqref{sif3}
\begin{equation}
\Phi(Y)(t) := e^{t\mathcal{A}} Y_0 + \int_0^t e^{(t-s)\mathcal{A}} \mathcal{N}_{\geq 2}(Y)(s) \, ds \text{.}
\end{equation}
Let $B_\varepsilon = \{ Y \in Z[0,1] : \|Y\|_{Z[0,1]} \leq \varepsilon \}$ be the closed ball of radius $\varepsilon > 0$ centered at the origin. Our goal is to demonstrate that for sufficiently small initial data $\|Y_0\|_{W^{1,\infty}} \le \varepsilon_0$, there exists an $\varepsilon > 0$ such that $\Phi$ is a strict contraction from $B_\varepsilon$ into itself.

{\bf{Step 2.}} 
By Lemma \ref{tAop}, we obtain
\begin{equation} \label{eq:LocalLinBound}
\|e^{t\mathcal{A}} Y_0\|_{Z[0,1]} \leq C_L \|Y_0\|_{W^{1,\infty}},
\end{equation}
for the linear evolution. To bound the nonlinear remainder, we evaluate 
\begin{equation}
I_n(t) := 2^n (1+2^nt)^{\delta_0} \left\| \mathcal{P}'_n \int_0^t e^{(t-s)\mathcal{A}} \mathcal{N}_{\geq 2}(Y)(s) \, ds \right\|_{L^\infty}.
\end{equation}

The nonlinear remainder admits the formal multilinear expansion over degrees $m \ge 1$
\begin{equation}
\mathcal{N}_{\geq 2}(Y) = \sum_{m=1}^{\infty} \sum_{\underline{k}} c_{m,\underline{k}} \mathcal{T}_{\underline{k}}(Y; \underbrace{Y, \dots, Y}_{m \text{ times}}).
\end{equation}

Applying Lemma \ref{LemmaNonlinMain} with arguments $f = g_1 = \dots = g_m = Y$, we get
\begin{equation}
\|\mathcal{P}'_n \mathcal{T}_{\underline{k}}(Y; Y, \dots, Y)(s)\|_{L^\infty} \leq (2^n s)^{-\delta_0} C^m \|Y\|_{Z[0,1]}^{m+1}.
\end{equation}

Applying the triangle inequality and factoring yields
\begin{align}
\|\mathcal{P}'_n \mathcal{N}_{\geq 2}(Y)(s)\|_{L^\infty} &\leq \sum_{m=1}^{\infty} \widetilde{c}_m (2^n s)^{-\delta_0} C^m \|Y\|_{Z[0,1]}^{m+1} \nonumber \\
&= (2^n s)^{-\delta_0} \|Y\|_{Z[0,1]}^2 \left( \sum_{m=1}^{\infty} \widetilde{c}_m C^m \|Y\|_{Z[0,1]}^{m-1} \right),
\end{align}
where $\widetilde{c}_m = \sum_{\underline{k}} |c_{m,\underline{k}}|$. Choosing the initial data such that $\|Y\|_{Z[0,1]} \leq \varepsilon < C^{-1}$, we obtain
\begin{equation}
\widetilde{C}_N := \sum_{m=1}^{\infty} \widetilde{c}_m C^m \varepsilon^{m-1} < \infty.
\end{equation}
This implies 
\begin{equation}
\|\mathcal{P}'_n \mathcal{N}_{\ge 2}(Y)(s)\|_{L^\infty} \le \widetilde{C}_N (2^n s)^{-\delta_0} \|Y\|_{Z[0,1]}^2.
\end{equation}
Furthermore, by Lemma \ref{tAop}, the linear semigroup provides high-frequency decay bounded by a continuous envelope $\mathcal{E}(\tau) = \min(1, \tau^{-4})$, yielding $\|\mathcal{P}'_n e^{\tau \mathcal{A}} X\|_{L^\infty} \lesssim \mathcal{E}(\tau 2^n) \|X\|_{L^\infty}$. This yields
\begin{equation}
I_n(t) \lesssim 2^n (1+2^nt)^{\delta_0} \int_0^t \mathcal{E}((t-s)2^n) (2^n s)^{-\delta_0} \|Y\|_Z^2 \, ds.
\end{equation}
To evaluate this convolution, we set $x = 2^n t$ and $\sigma = 2^n s$. Notice that $ds = 2^{-n} d\sigma$ perfectly cancels the $2^n$ derivative weight outside the integral. We hence have
\begin{equation}
I_n(t) \lesssim (1+x)^{\delta_0} \left( \int_0^x \mathcal{E}(x-\sigma) \sigma^{-\delta_0} \, d\sigma \right) \|Y\|_Z^2 \equiv H(x) \|Y\|_Z^2.
\end{equation}

It remains to show that
\begin{equation}
    H(x) := (1+x)^{\delta_0} \int_0^x \mathcal{E}(x-\sigma) \sigma^{-\delta_0} \, d\sigma
\end{equation}
is uniformly bounded for $x \ge 0$. We first consider $0 < x \le 1$. Since $\mathcal{E} \le 1$, we have
\begin{equation*}
\begin{split}
    H(x) &\le (1+x)^{\delta_0} \int_0^x \sigma^{-\delta_0} \, d\sigma= \frac{(1+x)^{\delta_0} x^{1-\delta_0}}{1-\delta_0}\lesssim 1,
\end{split}
\end{equation*}
because $0 < \delta_0 < 1$.

It remains to consider \(x\ge1\). We split the integral at \(x/2\).
\begin{align}
H(x) &\le (1+x)^{\delta_0} \left( \int_0^{x/2} \mathcal{E}(x-\sigma) \sigma^{-\delta_0} \, d\sigma + \int_{x/2}^x \mathcal{E}(x-\sigma) \sigma^{-\delta_0} \, d\sigma \right) \nonumber \\
&\le (1+x)^{\delta_0} \left( \sup_{\sigma \in [0, x/2]} \mathcal{E}(x-\sigma) \int_0^{x/2} \sigma^{-\delta_0} \, d\sigma + \sup_{\sigma \in [x/2, x]} \sigma^{-\delta_0} \int_{x/2}^x \mathcal{E}(x-\sigma) \, d\sigma \right) \nonumber \\
&\le (1+x)^{\delta_0} \left( \min(1, 16x^{-4}) \left[ \frac{\sigma^{1-\delta_0}}{1-\delta_0} \right]_0^{x/2} + (x/2)^{-\delta_0} \int_0^\infty \mathcal{E}(\tau) \, d\tau \right) \nonumber \\
&\le \frac{2^{\delta_0-1}}{1-\delta_0} \min \Big( x^{1-\delta_0}(1+x)^{\delta_0}, 16x^{-3-\delta_0}(1+x)^{\delta_0} \Big) + 2^{\delta_0} \left(\frac{1+x}{x}\right)^{\delta_0} \|\mathcal{E}\|_{L^1(\mathbb{R}_+)}.
\end{align}
Since $x \ge 1$, we have $((1+x)/x)^{\delta_0} \lesssim 1$. Moreover,
\begin{equation*}
    \min \big( x^{1-\delta_0}(1+x)^{\delta_0}, 16x^{-3-\delta_0}(1+x)^{\delta_0} \big) \lesssim 1.
\end{equation*}
Therefore $H(x) \lesssim 1$ also for $x \ge 1$. Combining the two cases gives
\begin{equation*}
    \sup_{x \ge 0} H(x) < \infty.
\end{equation*}

Because $\delta_0 = 1/8 < 1$ and $\|\mathcal{E}\|_{L^1(\mathbb{R}_+)} < \infty$, the right-hand side is  continuous on $(0, \infty)$. Furthermore, taking the limit as $x \to \infty$ yields
\begin{equation}
\limsup_{x \to \infty} H(x) \le 0 + 2^{\delta_0} (1) \|\mathcal{E}\|_{L^1(\mathbb{R}_+)} < \infty.
\end{equation}
Thus, the integral is globally bounded by a uniform constant: $\sup_{x \ge 0} H(x) \equiv C_H < \infty$. Taking the supremum over $n \in \mathbb{Z}_+$ and $t \in [0,1]$ yields
\begin{equation} \label{eq:LocalNonlinBound}
\left\| \int_0^t e^{(t-s)\mathcal{A}} \mathcal{N}_{\geq 2}(Y)(s) \, ds \right\|_{Z[0,1]} \leq C_H \|Y\|_{Z[0,1]}^2 \equiv C_N \|Y\|_{Z[0,1]}^2.
\end{equation}
Finally, taking the $Z[0,1]$ norm of $\Phi(Y)$ and applying \eqref{eq:LocalLinBound} and \eqref{eq:LocalNonlinBound}, we obtain:
\begin{equation}
\|\Phi(Y)\|_{Z[0,1]} \leq C_L \varepsilon_0 + C_N \varepsilon^2.
\end{equation}
We choose the radius of our ball to be $\varepsilon = 2C_L \varepsilon_0$. To ensure $\Phi$ maps $B_\varepsilon$ into itself, we require $C_L \varepsilon_0 + C_N (2C_L \varepsilon_0)^2 \leq 2C_L \varepsilon_0$, which restricts the initial data to $\varepsilon_0 \leq (4C_L C_N)^{-1}$. Under this condition, $\|\Phi(Y)\|_{Z[0,1]} \leq \varepsilon$.

{\bf{Step 3.}} 
Let $Y_1, Y_2 \in B_\varepsilon$. We evaluate
\begin{equation}
\Phi(Y_1) - \Phi(Y_2) = \int_0^t e^{(t-s)\mathcal{A}} \big( \mathcal{N}_{\geq 2}(Y_1) - \mathcal{N}_{\geq 2}(Y_2) \big)(s) \, ds.
\end{equation}

Using the formal expansion $\mathcal{N}_{\geq 2}(Y) = \sum_{m=2}^{\infty} \sum_{\underline{k}} c_{m,\underline{k}} \mathcal{T}_{\underline{k}}(\underbrace{Y; \dots, Y}_{m \text{ times}})$, the multilinear difference expands via the standard algebraic identity:
\begin{equation}
\mathcal{T}_{\underline{k}}(Y_1^{(m)}) - \mathcal{T}_{\underline{k}}(Y_2^{(m)}) = \sum_{j=1}^m \mathcal{T}_{\underline{k}}(Y_1; \dots, Y_1, \underbrace{Y_1 - Y_2}_{j\text{-th slot}}, Y_2, \dots, Y_2).
\end{equation}

Applying Lemma \ref{LemmaNonlinMain} to each of the $m$ terms in the expansion yields
\begin{equation}
\|\mathcal{P}'_n \big( \mathcal{T}_{\underline{k}}(Y_1^{(m)}) - \mathcal{T}_{\underline{k}}(Y_2^{(m)}) \big)(s)\|_{L^\infty} \leq m (2^n s)^{-\delta_0} C^m \max(\|Y_1\|_{Z}, \|Y_2\|_{Z})^{m-1} \|Y_1 - Y_2\|_{Z}.
\end{equation}

Summing over $m \ge 2$ and $\underline{k}$, and applying $\|Y_1\|_{Z}, \|Y_2\|_{Z} \le \varepsilon$, we obtain the pointwise Lipschitz estimate on the remainder
\begin{align}
\|\mathcal{P}'_n \big( \mathcal{N}_{\geq 2}(Y_1) - \mathcal{N}_{\geq 2}(Y_2) \big)(s)\|_{L^\infty} &\leq (2^n s)^{-\delta_0} \left( \sum_{m=2}^{\infty} \widetilde{c}_m m C^m \varepsilon^{m-1} \right) \|Y_1 - Y_2\|_{Z[0,1]} \nonumber \\
&\equiv \widetilde{C}_{Lip} (2^n s)^{-\delta_0} \|Y_1 - Y_2\|_{Z[0,1]},
\end{align}
where $\widetilde{c}_m = \sum_{\underline{k}} |c_{m,\underline{k}}|$. For $\varepsilon > 0$ strictly small ($\varepsilon < C^{-1}$), the power series converges to the finite constant $\widetilde{C}_{Lip} > 0$. 

Using this bound and the bound $C_H$ derived in \eqref{eq:LocalNonlinBound} yields
\begin{equation}
\|\Phi(Y_1) - \Phi(Y_2)\|_{Z[0,1]} \leq C_H \widetilde{C}_{Lip} \varepsilon \|Y_1 - Y_2\|_{Z[0,1]}.
\end{equation}
Restricting the initial data such that $\varepsilon \leq (2 C_H \widetilde{C}_{Lip})^{-1}$ yields
\begin{equation}
\|\Phi(Y_1) - \Phi(Y_2)\|_{Z[0,1]} \leq \frac{1}{2} \|Y_1 - Y_2\|_{Z[0,1]}.
\end{equation}
This confirms that $\Phi$ is a strict contraction mapping on $B_\varepsilon$.

{\bf{Step 4. Conclusion.}} 
Because $B_\varepsilon$ is a closed subset of the complete Banach space $Z[0,1]$, and $\Phi: B_\varepsilon \to B_\varepsilon$ is a strict contraction, the Banach Fixed-Point Theorem guarantees the existence of a unique fixed point $Y \in B_\varepsilon$ satisfying $Y = \Phi(Y)$. This unique $Y \in Z[0,1]$ constitutes the solution to the evolution equation \eqref{sif3} on the interval $[0,1]$, completing the proof.

\end{proof}

\subsection{Proof of Lemma \ref{LemmaNonlinMain}}\label{sif00} By multilinearity, we may assume that
\begin{equation}\label{har1}
\|f\|_{Z[0,1]}=\|g_1\|_{Z[0,1]}=\ldots=\|g_m\|_{Z[0,1]}=1.
\end{equation}
We use the identity \eqref{vf14}. Let
\begin{equation}\label{Pes28}
\mathcal{T}'_{m,V}(f;\underline{g})(x):=\int_{\SS^2}V(V(f))(y)\cdot (g_1(x)-g_1(y))\cdot\ldots\cdot(g_m(x)-g_m(y))K_m(x,y)\,d\mu(y),
\end{equation}
where $V\in\{V_{12}, V_{23}, V_{31}\}$. It suffices to prove that if $m\geq 1$, $t\in[0,1]$, and $n\in\Z_+$ then
\begin{equation}\label{har3}
(2^nt)^{\delta_0}\sum_{n_0\in\Z_+}\big\|\mathcal{P}'_n\mathcal{T}'_{m,V}(\mathcal{P}'_{n_0}f;\underline{g})\big\|_{L^\infty}\leq C^m
\end{equation}
for some constant $C\geq 1$, provided that the functions $f,g_1,\ldots,g_m\in Z[0,1]$ satisfy \eqref{har1}.

\textbf{Step 1.} We estimate first the sum over $n_0\geq n-10$. Notice first that
\begin{equation}\label{K1bound}
|V^{j_1}_yV^{j_2}_x K_m(x,y)|\lesssim_m|x-y|^{-(m+1+j_1+j_2)}
\end{equation}
for any $j_1,j_2\in\{0,1, 2\}$ and $x,y\in\SS^2$, where $V^j$ denotes any combination of $j$ of the vector-fields $V_{12}, V_{23}, V_{31}$ (see definition \eqref{vf11}) and the notation $A\lesssim_mB$ means that there is a constant $C\geq 1$ such that $A\leq C^mB$. We  decompose
\begin{equation}\label{har4}
\mathcal{T}'_{m,V}(\mathcal{P}'_{n_0}f;\underline{g})=H_1+H_2+H_3,
\end{equation}
where
\begin{equation}\label{case1_H1H2}
\begin{split}
H_1(x)&:=\int_{\SS^2}V(V(\mathcal{P}'_{n_0}f))(y)\cdot \prod_{l=1}^m(g_l(x)-g_l(y))\cdot K_m(x,y)\varphi_{\leq -n_0-20}(x-y)\,d\mu(y),\\
H_2(x)&:=\int_{\SS^2}V(V(\mathcal{P}'_{n_0}f))(y)\cdot \prod_{l=1}^m(g_l(x)-g_l(y))\cdot K_m(x,y)\varphi_{(-n_0-20,-n+4]}(x-y)\,d\mu(y),\\
H_3(x)&:=\int_{\SS^2}V(V(\mathcal{P}'_{n_0}f))(y)\cdot \prod_{l=1}^m(g_l(x)-g_l(y))\cdot K_m(x,y)\varphi_{>-n+4}(x-y)\,d\mu(y).
\end{split}
\end{equation}

It follows from the assumption \eqref{har1}, the definition \eqref{Intro7} and Corollary \ref{LiPaProj} (ii) that
\begin{equation}\label{Pes29}
\begin{split}
&\|V(\mathcal{P}'_{n_0}f)\|_{L^\infty}+2^{-n_0}\|V(V(\mathcal{P}'_{n_0}f))\|_{L^\infty}\lesssim (1+2^{n_0}t)^{-\delta_0},\\
&\|V(g_l)\|_{L^\infty}+\sup_{x,y\in\SS^2}\frac{|g_l(x)-g_l(y)|}{|x-y|}\lesssim 1\qquad\text{ for any }l\in\{1,\ldots,m\}.
\end{split}
\end{equation}

We can use \eqref{Pes29} and \eqref{K1bound} to estimate directly
\begin{equation}\label{Pes30}
\|H_1\|_{L^\infty}\lesssim_m2^{n_0}(1+2^{n_0}t)^{-\delta_0}\sup_{x\in\SS^2}\int_{\SS^2}|x-y|^{-1}\varphi_{\leq -n_0-20}(x-y)\,d\mu(y)\lesssim_m(1+2^{n_0}t)^{-\delta_0}.
\end{equation}
To estimate $\|H_2\|_{L^\infty}$ we would like to integrate by parts using the identity \eqref{vf13}. Notice that
\begin{equation*}
\Big|V_y\Big\{\prod_{l=1}^m(g_l(x)-g_l(y))\cdot K_m(x,y)\varphi_{(-n_0-20,-n+4]}(x-y)\Big\}\Big|\lesssim _m|x-y|^{-2}\mathbf{1}_{[2^{-n_0-22}, 2^{-n+6}]}(|x-y|)
\end{equation*}
for any $x,y\in\SS^2$, as a consequence of \eqref{Pes29} and \eqref{K1bound}. Therefore
\begin{equation}\label{Pes31}
\begin{split}
\|H_2\|_{L^\infty}&\lesssim_m\sup_{x\in\SS^2}\int_{\SS^2}|V(\mathcal{P}'_{n_0}f)(y)||x-y|^{-2}\mathbf{1}_{[2^{-n_0-22}, 2^{-n+6}]}(|x-y|)\,d\mu(y)\\
&\lesssim_m(1+2^{n_0}t)^{-\delta_0}(1+|n_0-n|).
\end{split}
\end{equation}

Finally, we estimate $\|\mathcal{P}'_nH_3\|_{L^\infty}$. We may assume $n\geq 1$ (since $H_3\equiv 0$ if $n=0$) and use Corollary \ref{LiPaProj} (ii) to estimate
\begin{equation*}
\|\mathcal{P}'_nH_3\|_{L^\infty}\lesssim 2^{-n}\sum_{V'\in\{V_{12}, V_{23}, V_{31}\}}\|V'(\mathcal{P}'_nH_3)\|_{L^\infty}.
\end{equation*}
Moreover, we notice that for any $V,V'\in\{V_{12}, V_{23}, V_{31}\}$
\begin{equation*}
\Big|V'_x\Big\{V_y\Big\{\prod_{l=1}^m(g_l(x)-g_l(y))\cdot K_m(x,y)\varphi_{>-n+4}(x-y)\Big\}\Big\}\Big|\lesssim _m|x-y|^{-3}\mathbf{1}_{[2^{-n},2]}(|x-y|)
\end{equation*}
for any $x,y\in\SS^2$, as a consequence of \eqref{Pes29} and \eqref{K1bound}. Therefore
\begin{equation}\label{Pes32}
\begin{split}
\|\mathcal{P}'_nH_3\|_{L^\infty}&\lesssim_m2^{-n}\sup_{x\in\SS^2}\int_{\SS^2}|V(\mathcal{P}'_{n_0}f)(y)||x-y|^{-3}\mathbf{1}_{[2^{-n},2]}(|x-y|)\,d\mu(y)\\
&\lesssim_m(1+2^{n_0}t)^{-\delta_0}.
\end{split}
\end{equation}

Using \eqref{Pes30}--\eqref{Pes32} it follows that
\begin{equation}\label{Pes33}
(2^nt)^{\delta_0}\sum_{n_0\geq n-10}\big\|\mathcal{P}'_n\mathcal{T}'_{m,V}(\mathcal{P}'_{n_0}f;\underline{g})\big\|_{L^\infty}\lesssim_m\sum_{n_0\geq n-10}\frac{(2^nt)^{\delta_0}(1+|n_0-n|)}{(1+2^{n_0}t)^{\delta_0}}\lesssim_m1.
\end{equation}

\textbf{Step 2.} We proceed with the second sum over $n_0< n-10$. We first decompose
\begin{equation}\label{case2_split}
\begin{split}
&\mathcal{T}'_{m,V}(\mathcal{P}'_{n_0}f;\underline{g})=H'_1+H'_2\\
&H'_1(x):=\int_{\SS^2}V(V(\mathcal{P}'_{n_0}f))(y)\cdot \prod_{l=1}^m(g_l(x)-g_l(y))\cdot K_m(x,y)\varphi_{\leq-(n+n_0)/2}(x-y)\,d\mu(y),\\
&H'_2(x):=\int_{\SS^2}V(V(\mathcal{P}'_{n_0}f))(y)\cdot \prod_{l=1}^m(g_l(x)-g_l(y))\cdot K_m(x,y)\varphi_{>-(n+n_0)/2}(x-y)\,d\mu(y),
\end{split}
\end{equation}

As before, we can use \eqref{Pes29} and \eqref{K1bound} to estimate directly
\begin{equation}\label{Pes40}
\begin{split}
\|H'_1\|_{L^\infty}&\lesssim_m2^{n_0}(1+2^{n_0}t)^{-\delta_0}\sup_{x\in\SS^2}\int_{\SS^2}\frac{\varphi_{\leq-(n+n_0)/2}(x-y)}{|x-y|}\,d\mu(y)\\
&\lesssim_m2^{(n_0-n)/2}(1+2^{n_0}t)^{-\delta_0}.
\end{split}
\end{equation}

Next, we take derivatives of $H_2'$ to obtain, for any $V'\in\{V_{12}, V_{23}, V_{31}\}$,
\begin{equation}\label{case2H2_split}
\begin{split}
&V'(H_{2}')=H'_{21}+H'_{22},\\
&H'_{21}(x):=\sum_{r=1}^m(V'g_r)(x)\int_{\SS^2}V(V(\mathcal{P}'_{n_0}f))(y)\cdot \prod_{l=1,\,l\neq r}^m(g_l(x)-g_l(y))\cdot \widetilde{K}_{m,>-(n+n_0)/2}(x,y)\,d\mu(y),\\
 &H'_{22}(x):=\int_{\SS^2}V(V(\mathcal{P}'_{n_0}f))(y)\cdot \prod_{l=1}^m(g_l(x)-g_l(y))\cdot V'_x(\widetilde{K}_{m,>-(n+n_0)/2})(x,y)\,d\mu(y),
\end{split}
\end{equation}
where, for any $a\in\R$,
\begin{equation}\label{Pes41}
\widetilde{K}_{m,>a}(x,y):=K_m(x,y)\varphi_{>a}(x-y),\qquad x,y\in\SS^2.
\end{equation}

To estimate $H'_{21}$ we notice that
\begin{equation*}
\Big|V_y\Big\{\prod_{l=1,\,l\neq r}^m(g_l(x)-g_l(y))\cdot \widetilde{K}_{m,>-(n+n_0)/2}(x,y)\Big\}\Big|\lesssim _m|x-y|^{-3}\mathbf{1}_{[2^{-(n_0+n)/2-2}, 2]}(|x-y|)
\end{equation*}
for any $x,y\in\SS^2$, as a consequence of \eqref{Pes29} and \eqref{K1bound}. Therefore, we can integrate by parts in $y$ (using \eqref{vf13}) and use again the bounds \eqref{Pes29} to conclude that
\begin{equation}\label{Pes42}
\begin{split}
\|H'_{21}\|_{L^\infty}&\lesssim_m\sup_{x\in\SS^2}\int_{\SS^2}|V(\mathcal{P}'_{n_0}f)(y)||x-y|^{-3}\mathbf{1}_{[2^{-(n_0+n)/2-2}, 2]}(|x-y|)\,d\mu(y)\\
&\lesssim_m(1+2^{n_0}t)^{-\delta_0}2^{(n+n_0)/2}.
\end{split}
\end{equation}

Similarly, to estimate $H'_{22}$ we notice that
\begin{equation*}
\Big|V_y\Big\{\prod_{l=1}^m(g_l(x)-g_l(y))\cdot V'_x(\widetilde{K}_{m,>-(n+n_0)/2})(x,y)\Big\}\Big|\lesssim _m|x-y|^{-3}\mathbf{1}_{[2^{-(n_0+n)/2-2}, 2]}(|x-y|)
\end{equation*}
for any $x,y\in\SS^2$, as a consequence of \eqref{Pes29} and \eqref{K1bound}. The same estimate as in \eqref{Pes42} above then shows that $\|H'_{22}\|_{L^\infty}\lesssim_m(1+2^{n_0}t)^{-\delta_0}2^{(n+n_0)/2}$. Therefore $\|V'H'_2\|_{L^\infty}\lesssim_m(1+2^{n_0}t)^{-\delta_0}2^{(n+n_0)/2}$ for any $V'\in\{V_{12}, V_{23}, V_{31}\}$. Using Corollary \ref{LiPaProj} (ii) we have $\|\mathcal{P}'_n(H'_2)\|_{L^\infty}\lesssim_m(1+2^{n_0}t)^{-\delta_0}2^{(n_0-n)/2}$. Finally, we combine this with the bounds \eqref{Pes40} and the definitions \eqref{case2_split} to conclude that
\begin{equation}\label{Pes45}
\|\mathcal{P}'_n\mathcal{T}'_{m,V}(\mathcal{P}'_{n_0}f;\underline{g})\|_{L^\infty}\lesssim_m(1+2^{n_0}t)^{-\delta_0}2^{(n_0-n)/2}.
\end{equation}
Therefore
\begin{equation}\label{Pes46}
(2^nt)^{\delta_0}\sum_{n_0\leq n-10}\big\|\mathcal{P}'_n\mathcal{T}'_{m,V}(\mathcal{P}'_{n_0}f;\underline{g})\big\|_{L^\infty}
\lesssim_m\sum_{n_0\leq n-10}\frac{(2^nt)^{\delta_0}2^{(n_0-n)/2}}{(1+2^{n_0}t)^{\delta_0}}\lesssim_m1.
\end{equation}
The main bounds \eqref{har3} follow from \eqref{Pes33} and \eqref{Pes46}, as desired.

\section{Global extension and asymptotic stability}\label{sec:global}

\subsection{Parabolic Smoothing and Structural Decoupling}

By our local existence theory, the parabolic smoothing of the linear operator $\mathcal{A}$ instantly desingularizes the rough initial data. Consequently, at time $t=1$, the physical interface $X $  has regularized into the classical Hölder space, satisfying $\|X (1) - X_0 \|_{C^{1,\alpha}} \le \varepsilon$ for a fixed $\alpha \in (0,1)$. 

To study the global dynamics and asymptotic stability of the interface, we must decompose the flow into movements along the $10$-dimensional manifold of steady states and a strictly dissipative perturbation. Following the structural decomposition framework, we separate the deformations into volume fluctuations $\rho$, spatial translations $\mathbf{a}$, conformal parameters $\mathbf{q}$, and an orthogonal perturbation $Y$.

\begin{lemma}[Structural Parameterization and Decoupling] \label{Lem:ParameterDecoupling}
	Let $X (t)$ be a solution to \eqref{Intro1} on a time interval $I \subseteq [1, \infty)$, satisfying $\sup_{t \in I} \|X (t) - X_0 \|_{C^{1,\alpha}} < \varepsilon$ for a sufficiently small $\varepsilon > 0$. Then, $X $ admits a unique structural decomposition
	\begin{equation}
	X (\boldsymbol{\alpha}, t) = Z_{\rho}(\boldsymbol{\mathcal{S}}_{\mathbf{q}(t)}(\boldsymbol{\alpha}), t) - \frac{2}{3}\mathbf{v}(t) + \mathbf{a}(t), \quad \text{where} \quad Z_{\rho}(\boldsymbol{\beta}, t) = (1+\rho(t))\boldsymbol{\beta} + Y(\boldsymbol{\beta}, t) .
	\end{equation}
	Here, $(\rho, \mathbf{a}, \mathbf{q}) \in \mathbb{R} \times \mathbb{R}^3 \times \mathbb{R}^6$ represent volume dilations, spatial translations, and conformal symmetries respectively, with $\mathbf{v} = (q_1, q_2, q_3)^\top$. 
	
	This decomposition is uniquely determined by requiring that the infinite-dimensional perturbation $Y$ remains  orthogonal to the $10$-dimensional kernel of the linearized operator $\mathcal{N}_1$
	\begin{equation}\label{eq:KernelOrthogonality}
	\langle Y(\cdot, t), \mathbf{K}_n \rangle_{L^2(\mathbb{S}^2)} = 0, \quad \text{for all } n \in \{1, \dots, 10\}.
	\end{equation}
	
	Under this condition, the parameters $(\rho, \mathbf{a}, \mathbf{q})$ are governed by a locally invertible matrix ODE, while the decoupled state $Y$ evolves under the dissipative flow of $\mathcal{N}_1$. 
\end{lemma}

\begin{proof}
	
To study the asymptotic stability of the system, we must decouple the $10$-dimensional manifold of steady states from the dissipative flow. We parameterize the orientation-preserving conformal steady states $\boldsymbol{\mathcal{S}}_{\mathbf{q}}$ by mapping the restricted Lorentz group $\mathbb{SO}^+(3,1)$ to the sphere. Fixing a reference state $\Lambda_\star \in \mathbb{SO}^+(3,1)$, we represent perturbations using the exponential map
\begin{equation}\label{eq:ExpMap}
\Lambda_{\mathbf{q}} = \exp(\mathcal{G}_{\mathbf{q}}), \quad \mathcal{G}_{\mathbf{q}} = \sum_{i=1}^6 q_i Q_i = \begin{pmatrix} 0 & q_1 & q_2 & q_3 \\ q_1 & 0 & -q_6 & q_5 \\ q_2 & q_6 & 0 & -q_4 \\ q_3 & -q_5 & q_4 & 0 \end{pmatrix} ,
\end{equation}
where $\mathbf{q} \in \mathbb{R}^6$ resides in the vicinity of the origin. The parameters $q_1, q_2, q_3$ correspond to Lorentzian boosts, while $q_4, q_5, q_6$ generate rigid rotations. We define the steady state mapping as $\boldsymbol{\mathcal{S}}_{\mathbf{q}} \equiv \boldsymbol{\mathcal{S}}_{\Lambda_{\mathbf{q}}\Lambda_\star}$.

Defining the translational velocity component as $\mathbf{v}(t) = (q_1, q_2, q_3)^\top$, we parameterize an arbitrary interfacial deformation $X $ as
\begin{equation}\label{eq:DecompX}
X (\boldsymbol{\alpha}, t) = Z_{\rho}(\boldsymbol{\mathcal{S}}_{\mathbf{q}(t)}(\boldsymbol{\alpha}), t) - \frac{2}{3}\mathbf{v}(t) + \mathbf{a}(t), \quad Z_{\rho}(\boldsymbol{\beta}, t) = (1+\rho(t))\boldsymbol{\beta} + Y(\boldsymbol{\beta}, t) ,
\end{equation}
where $\mathbf{a}(t) \in \mathbb{R}^3$ represents spatial shifts and $\rho(t) \in \mathbb{R}$ captures volume dilations.

To guarantee the uniqueness of this decomposition, the perturbation $Y$ must be strictly confined to the orthogonal complement of the $10$-dimensional kernel of the linearized operator $\mathcal{N}_1$. We enforce the exact spatial $L^2(\mathbb{S}^2; \mathbb{C}^3)$ orthogonality condition against the explicit geometric basis spanning the kernel

\begin{equation}\label{eq:SpatialOrtho}
\langle Y(\cdot, t), \mathbf{K} \rangle_{L^2(\mathbb{S}^2; \mathbb{C}^3)} = 0, \quad \text{for all } \mathbf{K} \in \big\{ \mathbf{H}^1_{1,0}, \mathbf{H}^1_{2,m}, \mathbf{H}^2_{0,m}, \mathbf{H}^3_{1,m} \big\}_{m \in \{-1,0,1\}}.
\end{equation}
Let $\boldsymbol{\beta} = \boldsymbol{\mathcal{S}}_{\mathbf{q}(t)}(\boldsymbol{\alpha})$. To compute the kinematic time derivative of \eqref{eq:DecompX} we evaluate the total time derivative via the multivariable chain rule. Differentiating each constituent parameter component yields the expansion
\begin{equation}\label{eq:ChainRuleBase}
\partial_t X (\boldsymbol{\alpha}, t) = \dot{\rho}\boldsymbol{\beta} + (1+\rho)\partial_t \boldsymbol{\beta} + \partial_t Y(\boldsymbol{\beta}, t) + \big(\nabla_{\mathbb{S}^2}Y(\boldsymbol{\beta}, t)\cdot \partial_t \boldsymbol{\beta}\big) - \frac{2}{3}\dot{\mathbf{v}} + \dot{\mathbf{a}} .
\end{equation}
To evaluate $\partial_t \boldsymbol{\beta} = \partial_t \boldsymbol{\mathcal{S}}_{\mathbf{q}(t)}(\boldsymbol{\alpha})$, we apply the chain rule with respect to the conformal parameters $\mathbf{q}(t) \in \mathbb{R}^6$, yielding $\partial_t \boldsymbol{\beta} = \left(\partial_{q_i}\boldsymbol{\mathcal{S}}_{\mathbf{q}}(\boldsymbol{\alpha})\right)\dot{q}_i$. Recalling the projective action definition of the restricted Lorentz group $\mathbb{SO}^+(3,1)$ mapping onto the sphere
\begin{equation}\label{ConformalAction}
\boldsymbol{\mathcal{S}}_{\mathbf{q}}(\boldsymbol{\alpha}) = \frac{\mathbb{P}_3 \Lambda_{\mathbf{q}} \widetilde{\boldsymbol{\alpha}}}{\mathbf{e}_0^\top \Lambda_{\mathbf{q}} \widetilde{\boldsymbol{\alpha}}},
\end{equation}
where $\widetilde{\boldsymbol{\alpha}} = (1, \boldsymbol{\alpha}^\top)^\top \in \mathbb{R}^4$, we compute the partial derivative $\partial_{q_i}\boldsymbol{\mathcal{S}}_{\mathbf{q}}$ via the projective quotient rule on $\mathbb{R}^4$ to get
\begin{equation}\label{eq:ProjectiveQuotient}
\partial_{q_i}\boldsymbol{\mathcal{S}}_{\mathbf{q}}(\boldsymbol{\alpha}) = \frac{\mathbb{P}_3 \left(\partial_{q_i}\Lambda_{\mathbf{q}}\right)\widetilde{\boldsymbol{\alpha}}}{\mathbf{e}_0^\top \Lambda_{\mathbf{q}}\widetilde{\boldsymbol{\alpha}}} - \frac{\mathbb{P}_3 \Lambda_{\mathbf{q}}\widetilde{\boldsymbol{\alpha}}}{\left(\mathbf{e}_0^\top \Lambda_{\mathbf{q}}\widetilde{\boldsymbol{\alpha}}\right)^2} \left(\mathbf{e}_0^\top \left(\partial_{q_i}\Lambda_{\mathbf{q}}\right)\widetilde{\boldsymbol{\alpha}}\right) .
\end{equation}
Using the integral formulation of the non-commutative matrix exponential derivative, the partial derivative of the Lorentz matrix perturbation $\Lambda_{\mathbf{q}} = \exp(\mathcal{G}_{\mathbf{q}})$ is
\begin{equation}
\partial_{q_i}\Lambda_{\mathbf{q}} = \mathcal{A}_{\mathbf{q}}(Q_i)\Lambda_{\mathbf{q}}, \quad \text{where} \quad \mathcal{A}_{\mathbf{q}}(Q_i) = \int_0^1 \Lambda_{\sigma \mathbf{q}} Q_i \Lambda_{-\sigma \mathbf{q}} \, d\sigma.
\end{equation}
Substituting this back into \eqref{eq:ProjectiveQuotient} yields
\begin{equation}\label{eq:ProjectiveQuotientSubbed}
\partial_{q_i}\boldsymbol{\mathcal{S}}_{\mathbf{q}}(\boldsymbol{\alpha}) = \frac{\mathbb{P}_3 \mathcal{A}_{\mathbf{q}}(Q_i)\Lambda_{\mathbf{q}}\widetilde{\boldsymbol{\alpha}}}{\mathbf{e}_0^\top \Lambda_{\mathbf{q}}\widetilde{\boldsymbol{\alpha}}} - \left(\frac{\mathbf{e}_0^\top \mathcal{A}_{\mathbf{q}}(Q_i)\Lambda_{\mathbf{q}}\widetilde{\boldsymbol{\alpha}}}{\mathbf{e}_0^\top \Lambda_{\mathbf{q}}\widetilde{\boldsymbol{\alpha}}}\right) \frac{\mathbb{P}_3 \Lambda_{\mathbf{q}}\widetilde{\boldsymbol{\alpha}}}{\mathbf{e}_0^\top \Lambda_{\mathbf{q}}\widetilde{\boldsymbol{\alpha}}}.
\end{equation}
To express \eqref{eq:ProjectiveQuotientSubbed} directly in terms of $\boldsymbol{\beta}$ rather than the reference coordinates, we do a geometric pullback using the relation $\widetilde{\boldsymbol{\alpha}} = \Lambda_{-\mathbf{q}}\widetilde{\boldsymbol{\beta}}$. We notice that
\begin{equation}
\boldsymbol{\mathcal{M}}_{i,\mathbf{q}}(\boldsymbol{\beta}) \equiv \frac{\mathcal{A}_{\mathbf{q}}(Q_i)\Lambda_{\mathbf{q}}\widetilde{\boldsymbol{\alpha}}}{\mathbf{e}_0^\top \Lambda_{\mathbf{q}}\widetilde{\boldsymbol{\alpha}}} = \frac{\mathcal{A}_{\mathbf{q}}(Q_i)\Lambda_{\mathbf{q}}\Lambda_{-\mathbf{q}}\widetilde{\boldsymbol{\beta}}}{\mathbf{e}_0^\top \Lambda_{\mathbf{q}}\Lambda_{-\mathbf{q}}\widetilde{\boldsymbol{\beta}}} = \frac{\mathcal{A}_{\mathbf{q}}(Q_i)\widetilde{\boldsymbol{\beta}}}{\mathbf{e}_0^\top \widetilde{\boldsymbol{\beta}}} = \mathcal{A}_{\mathbf{q}}(Q_i)\widetilde{\boldsymbol{\beta}},
\end{equation}
where here we used that $\mathbf{e}_0^\top \widetilde{\boldsymbol{\beta}} = 1$, since $\beta$ is on the unit sphere. Reinserting $\boldsymbol{\mathcal{M}}_{i,\mathbf{q}}(\boldsymbol{\beta})$ into \eqref{eq:ProjectiveQuotientSubbed} yields
\begin{equation}\label{eq:ConformalGenerator}
\partial_{q_i}\boldsymbol{\mathcal{S}}_{\mathbf{q}}(\boldsymbol{\alpha}) = \mathbb{P}_3 \boldsymbol{\mathcal{M}}_{i,\mathbf{q}}(\boldsymbol{\beta}) - \left(\mathbf{e}_0^\top \boldsymbol{\mathcal{M}}_{i,\mathbf{q}}(\boldsymbol{\beta})\right)\boldsymbol{\beta} \equiv \boldsymbol{\mathcal{T}}_{i,\mathbf{q}}(\boldsymbol{\beta}),
\end{equation}
and hence, $\partial_t \boldsymbol{\beta} = \boldsymbol{\mathcal{T}}_{i,\mathbf{q}}(\boldsymbol{\beta})\dot{q}_i$.

Using \eqref{eq:ConformalGenerator}, we now also obtain
\begin{equation}
\nabla_{\mathbb{S}^2}Y(\boldsymbol{\beta}, t)\cdot \partial_t \boldsymbol{\beta} = \sum_{j=1}^3 \left(\partial_{\beta_j}Y(\boldsymbol{\beta}, t)\right) \boldsymbol{\mathcal{T}}^j_{i,\mathbf{q}}(\boldsymbol{\beta})\dot{q}_i = \langle \boldsymbol{\mathcal{T}}_{i,\mathbf{q}}(\boldsymbol{\beta}), \nabla_{\mathbb{S}^2}Y(\boldsymbol{\beta}, t) \rangle_{\mathbb{S}^2}\dot{q}_i.
\end{equation}
Finally, we use the fact that $X_0 (\boldsymbol{\beta}) = \boldsymbol{\beta}$ in $\dot{\rho}\boldsymbol{\beta}$, to obtain the total kinematic time derivative
\begin{equation}\label{eq:TimeDerivX}
\partial_t X (\boldsymbol{\alpha}, t) = (1+\rho)\boldsymbol{\mathcal{T}}_{i,\mathbf{q}}(\boldsymbol{\beta})\dot{q}_i + X_0 (\boldsymbol{\beta})\dot{\rho} + \partial_t Y(\boldsymbol{\beta}, t) + \langle \boldsymbol{\mathcal{T}}_{i,\mathbf{q}}(\boldsymbol{\beta}), \nabla_{\mathbb{S}^2}Y(\boldsymbol{\beta}, t) \rangle_{\mathbb{S}^2}\dot{q}_i - \frac{2}{3}\dot{\mathbf{v}} + \dot{\mathbf{a}}.
\end{equation}

Using the translation invariance and conformal covariance of the Peskin operator, the equation $\partial_t X = \mathcal{N}(X)$, pulled back by
\begin{equation*}
    \beta = \mathcal{S}_{\mathbf{q}(t)}(\alpha),
\end{equation*}
becomes the exact identity
\begin{equation}\label{5.15}
    \partial_t Y(\beta, t) + X_0 (\beta)\dot{\rho} + Q_i(\rho, \mathbf{q}, Y)(\beta)\mathbf{\dot{q}}_i - \frac{2}{3}\dot{\mathbf{v}} + \mathbf{\dot{a}} = \mathcal{F}_\rho(Y)(\beta), 
\end{equation}
where
\begin{equation*}
    \mathcal{F}_\rho(Y) := \mathcal{N}\big((1 + \rho)X_0  + Y\big), 
\end{equation*}
and
\begin{equation*}
    Q_i(\rho, \mathbf{q}, Y)(\beta) := (1 + \rho)\mathcal{T}_{i, \mathbf{q}}(\beta) + \big\langle \mathcal{T}_{i, \mathbf{q}}(\beta), \nabla_{\mathbb{S}^2}Y(\beta) \big\rangle_{\mathbb{S}^2}. 
\end{equation*}
Here $X_0 (\beta) = \beta$. Notice that no Taylor expansion has been made in \eqref{5.15}; this is only the exact equation in the modulated variables.

We now project \eqref{5.15} onto the null space by taking the $L^2(\mathbb{S}^2; \mathbb{C}^3)$ scalar product against the vector spherical harmonic basis spanning $\ker \mathcal{N}_1$:
\begin{equation*}
    K \in \mathbf{B}_{\mathrm{ker}} \equiv \{\mathbf{H}^1_{1, 0}\} \cup \{\mathbf{H}^1_{2, m}, \mathbf{H}^2_{0, m}, \mathbf{H}^3_{1, m}\}_{m \in \{-1, 0, 1\}}. 
\end{equation*}
Because the basis modes are static on the reference sphere, differentiating the orthogonality condition gives
\begin{equation*}
    \langle \partial_t Y, K \rangle_{L^2} = 0.
\end{equation*}
Therefore, projection of \eqref{5.15} gives the following finite-dimensional system:
\begin{equation*}
    \big\langle \mathbf{H}^1_{1, 0}, Q_i \big\rangle_{L^2}\mathbf{\dot{q}}_i + \big\langle \mathbf{H}^1_{1, 0}, X_0  \big\rangle_{L^2}\dot{\rho} = \big\langle \mathbf{H}^1_{1, 0}, \mathcal{F}_\rho(Y) \big\rangle_{L^2}, 
\end{equation*}
\begin{equation*}
    \big\langle \mathbf{H}^2_{0, m}, Q_i \big\rangle_{L^2}\mathbf{\dot{q}}_i - \frac{2}{3}\big\langle \mathbf{H}^2_{0, m}, \dot{\mathbf{v}} \big\rangle_{L^2} + \big\langle \mathbf{H}^2_{0, m}, \mathbf{\dot{a}} \big\rangle_{L^2} = \big\langle \mathbf{H}^2_{0, m}, \mathcal{F}_\rho(Y) \big\rangle_{L^2}, 
\end{equation*}
\begin{equation*}
    \big\langle \mathbf{H}^3_{1, m}, Q_i \big\rangle_{L^2}\mathbf{\dot{q}}_i = \big\langle \mathbf{H}^3_{1, m}, \mathcal{F}_\rho(Y) \big\rangle_{L^2}, 
\end{equation*}
\begin{equation*}
    \big\langle \mathbf{H}^1_{2, m}, Q_i \big\rangle_{L^2}\mathbf{\dot{q}}_i - \frac{2}{3}\big\langle \mathbf{H}^1_{2, m}, \dot{\mathbf{v}} \big\rangle_{L^2} + \big\langle \mathbf{H}^1_{2, m}, \mathbf{\dot{a}} \big\rangle_{L^2} = \big\langle \mathbf{H}^1_{2, m}, \mathcal{F}_\rho(Y) \big\rangle_{L^2}, 
\end{equation*}
for $m \in \{-1, 0, 1\}$. These equations can be written compactly as
\begin{equation*}
    \mathbf{B}(\rho, \mathbf{q}, Y)(\dot{\rho}, \mathbf{\dot{a}}, \mathbf{\dot{q}})^\top = \mathbf{R}(\rho, Y), 
\end{equation*}
where the components of $\mathbf{R}$ are the corresponding projections of $\mathcal{F}_\rho(Y)$ onto the basis elements of $\mathbf{B}_{\mathrm{ker}}$.

It remains to show that the matrix $\mathbf{B}(\rho, \mathbf{q}, Y)$ is invertible for small $(\rho, \mathbf{q}, Y)$. We evaluate the matrix at the identity configuration $(\rho, \mathbf{q}, Y) = (0, 0, 0)$. Let
\begin{equation*}
    \mathbf{B}_{\mathrm{ker}} = \{K_a\}_{a=1}^{10} = \{\mathbf{H}^1_{1, 0}\} \cup \{\mathbf{H}^2_{0, m}, \mathbf{H}^3_{1, m}, \mathbf{H}^1_{2, m}\}_{m \in \{-1, 0, 1\}} 
\end{equation*}
be the ordered orthonormal basis of $\ker \mathcal{N}_1$ established in Lemma~\ref{LemSteadyKernel}. The infinitesimal generators of the steady-state parameter manifold at the identity form the geometric set
\begin{equation*}
    \mathbf{B}_{\mathrm{geom}} = \{\mathbf{v}_b\}_{b=1}^{10} = \{\widehat{r}\} \cup \{e_k, \omega_k, g_k\}_{k=1}^3. 
\end{equation*}
The matrix
\begin{equation*}
    (\mathbf{B}_0)_{ab} = \langle K_a, \mathbf{v}_b \rangle_{L^2(\mathbb{S}^2; \mathbb{C}^3)} 
\end{equation*}
is the change-of-basis matrix between $\mathbf{B}_{\mathrm{geom}}$ and $\mathbf{B}_{\mathrm{ker}}$. By Lemma~\ref{LemSteadyKernel},
\begin{equation*}
    \operatorname{span}(\mathbf{B}_{\mathrm{geom}}) = \operatorname{span}(\mathbf{B}_{\mathrm{ker}}) = \ker \mathcal{N}_1,
\end{equation*}
and both spaces have dimension $10$. Hence,
\begin{equation*}
    \det \mathbf{B}_0 \neq 0. 
\end{equation*}
Since $\mathbf{B}(\rho, \mathbf{q}, Y)$ depends continuously on $(\rho, \mathbf{q}, Y)$, it remains invertible for
\begin{equation*}
    |\rho| + |\mathbf{q}| + \|Y\|_{C^{1, \alpha}}
\end{equation*}
sufficiently small.

Finally, define the nonlinear map
\begin{equation*}
    \Phi(\rho, \mathbf{a}, \mathbf{q}, Y) = Z_\rho(\mathbf{\mathcal{S}}_\mathbf{q}(\cdot)) - \frac{2}{3}\mathbf{v} + \mathbf{a} - X, \qquad Z_\rho(\beta) = (1 + \rho)\beta + Y(\beta). 
\end{equation*}
At the identity configuration, the Fr\'echet derivative of $\Phi$ with respect to $(\rho, \mathbf{a}, \mathbf{q}, Y)$ maps
\begin{equation*}
    (\delta \mathbf{p}, \delta Y) \longmapsto \mathbf{B}_0 \delta \mathbf{p} + \delta Y.
\end{equation*}
Since $\mathbf{B}_0 \delta \mathbf{p}$ spans $\ker \mathcal{N}_1$ and $\delta Y$ spans $(\ker \mathcal{N}_1)^\perp$, this derivative is an isomorphism from
\begin{equation*}
    \mathbb{R}^{10} \times (\ker \mathcal{N}_1)^\perp
\end{equation*}
onto $C^{1, \alpha}(\mathbb{S}^2; \mathbb{R}^3)$. The Implicit Function Theorem therefore gives, for every $X$ sufficiently close to $X_0 $ in $C^{1, \alpha}$, unique parameters $(\rho, \mathbf{a}, \mathbf{q})$ and a unique perturbation $Y \in (\ker \mathcal{N}_1)^\perp$. Applying this pointwise in time gives the claimed structural decomposition and the locally invertible matrix ODE for the parameters.
\end{proof}

\begin{remark}
    We will denote $r_*$ the exact radius of a sphere having the same volume $V_0$ as the initial data. By incompressibility, we will see that the volume is preserved, hence the final value of the parameter $r_\infty$ will be equal to $r_*$.
\end{remark}

\begin{lemma}[Volume constraint and dilation scaling]
\label{Lem:VolumeScaling}
Let $V_0 := V(X(0))$ and define the reference radius $r_* := (3V_0 / 4\pi)^{1/3}$. 
Assume the interface $X$ admits the decomposition
\begin{equation*}
    X = (1 + \rho)X_0  + Y, \qquad Y \in (\ker \mathcal{N}_1)^\perp,
\end{equation*}
subject to the volume constraint $V(X) = V_0$. If $\|X - X_0 \|_{C^{1, \alpha}}$ is sufficiently small, then
\begin{equation*}
    |\rho - (r_* - 1)| \le C \|Y\|_{C^{1, \alpha}}^2.
\end{equation*}
In particular, the distance to the rescaled steady state satisfies 
\begin{equation*}
    \|X - r_* X_0 \|_{C^{1, \alpha}} \le C\|Y\|_{C^{1, \alpha}}.
\end{equation*}
Furthermore, if $V_0 = V(X_0 )$, then the dilation parameter is quadratically controlled by the perturbation:
\begin{equation*}
    |\rho| \le C \|Y\|_{C^{1, \alpha}}^2.
\end{equation*}
\end{lemma}

\begin{proof}
By translation invariance of volume and invariance under reparametrization, the full decomposition
\begin{equation*}
    X (\alpha, t) = Z_\rho(S_{q(t)}(\alpha), t) - \frac{2}{3}\mathbf{v}(q(t)) + \mathbf{a}(t)
\end{equation*}
has the same enclosed volume as
\begin{equation*}
    Z_\rho(\beta, t) = (1 + \rho(t))\beta + Y(\beta, t).
\end{equation*}
Thus it suffices to prove the estimate for the reduced form
\begin{equation*}
    X  = (1 + \rho)X_0  + Y.
\end{equation*}

Since the Peskin velocity is induced by an incompressible Stokes flow, the enclosed volume is conserved:
\begin{equation*}
    V(X(t)) = V(X(0)) = V_0.
\end{equation*}
For a sphere of radius $r_*$, we have $V(r_* X_0 ) = \frac{4\pi}{3} r_*^3 = V_0$. Expanding the interface as
\begin{equation*}
    X = (1 + \rho)X_0  + Y,
\end{equation*}
we perform a Taylor expansion of the volume functional around $X_0 $ to obtain
\begin{equation*}
    V((1 + \rho)X_0  + Y) = \frac{4\pi}{3}(1 + \rho)^3 + \langle Y, (1 + \rho)^2 X_0  \rangle_{L^2} + O(\|Y\|_{C^{1, \alpha}}^2).
\end{equation*}
Because $Y \in (\ker \mathcal{N}_1)^\perp$ and $X_0 $ spans the kernel corresponding to radial dilations, the linear term in $Y$ vanishes. Consequently,
\begin{equation*}
    V_0 = \frac{4\pi}{3}(1 + \rho)^3 + O(\|Y\|_{C^{1, \alpha}}^2).
\end{equation*}
Substituting $V_0 = \frac{4\pi}{3}r_*^3$ yields
\begin{equation*}
    |(1 + \rho)^3 - r_*^3| \le C \|Y\|_{C^{1, \alpha}}^2.
\end{equation*}
Since $(1 + \rho)$ and $r_*$ are close to $1$, the map $f(z) = z^3$ is locally bi-Lipschitz, which implies
\begin{equation*}
    |\rho - (r_* - 1)| \le C \|Y\|_{C^{1, \alpha}}^2.
\end{equation*}
\end{proof}

\subsection{Global Function Spaces and Asymptotic Decay}

Applying Lemma \ref{Lem:ParameterDecoupling} at $t=1$, the decoupled state $Y$ inherits the regularity of $X $, yielding the initial bound $\|Y(1)\|_{C^{1,\alpha}} \le C\varepsilon \equiv \varepsilon_1$. 
For the long-time asymptotic stability, we introduce the time-weighted norm 
\begin{equation} \label{eq:GlobalXNorm}
\|Y\|_{\mathcal{X}} = \sup_{t \ge 1} \Big( e^{\lambda t} \|Y(\cdot, t)\|_{C^{1,\alpha}(\mathbb{S}^2)} \Big),
\end{equation}
where $0<\lambda<a_2/r_\infty$, with $a_2=8/35$ \eqref{tete4}.

By construction, proving that a solution remains uniformly bounded in this space, $\|Y\|_{\mathcal{X}} \le C$, is equivalent to proving that the interface converges to the steady-state spherical manifold at a rate of $\|Y(t)\|_{C^{1,\alpha}} = \mathcal{O}(e^{-\lambda t})$ as $t \to \infty$. 
Moreover, we will show at the end of the section that the exponential rate is actually $a_2/r_\infty$; see Proposition \ref{prop:decay_exact} and the numerical verification results in Section \ref{sec:numerical}.

\begin{lemma}[Schauder bound for the Peskin multilinear forms]
\label{Lem:PeskinSchauder}
Fix $\alpha \in (0, 1)$. Let $\mathcal{T}_m$ be one of the multilinear operators defined by
\begin{equation*}
    \mathcal{T}_m(f; g_1, \dots, g_m)(x) := \int_{\mathbb{S}^2} \Delta_{\mathbb{S}^2} f(y) \prod_{\ell=1}^m (g_\ell(x) - g_\ell(y)) K_m(x, y) \, d\mu(y),
\end{equation*}
with kernels $K_m$ as in \eqref{har2}. Then
\begin{equation*}
    \|\mathcal{T}_m(f; g_1, \dots, g_m)\|_{C^\alpha} \le C^m \|f\|_{C^{1, \alpha}} \prod_{\ell=1}^m \|g_\ell\|_{C^{1, \alpha}}.
\end{equation*}
Moreover, if all inputs are in $C^{2, \alpha}$, then the tame estimate
\begin{equation*}
    \|\mathcal{T}_m(f; g_1, \dots, g_m)\|_{C^{1, \alpha}} \le C^m \sum_{r=0}^m \|u_r\|_{C^{2, \alpha}} \prod_{s \ne r} \|u_s\|_{C^{1, \alpha}}
\end{equation*}
holds, where $u_0 = f$ and $u_s = g_s$ for $1 \le s \le m$.
\end{lemma}

\begin{proof}
Using the identity
\begin{equation*}
    \Delta_{\mathbb{S}^2} = V_{12}^2 + V_{23}^2 + V_{31}^2
\end{equation*}
and the divergence-free properties \eqref{vf13}, we integrate by parts with respect to $y$ to shift the two derivatives off $f$. Following this integration by parts, the expression becomes a finite sum of singular integral operators on $\mathbb{S}^2$ whose kernels satisfy standard Calderón-Zygmund bounds in local coordinates. The coefficients are products of divided differences of the form
\begin{equation*}
    \frac{g_\ell(x) - g_\ell(y)}{|x - y|},
\end{equation*}
which are bounded by $\|g_\ell\|_{C^{1, \alpha}}$.

The $C^\alpha$ estimate follows from standard Schauder estimates for Calderón-Zygmund operators on compact smooth manifolds. The $C^{1, \alpha}$ estimate is obtained by differentiating once with respect to one of the vector fields $V_{12}, V_{23}, V_{31}$. When the derivative acts on a single input, that input is estimated in $C^{2, \alpha}$, while all remaining inputs are estimated in $C^{1, \alpha}$, yielding the stated tame estimate.
\end{proof}

\begin{lemma}[Nonlinear remainder estimates]
\label{Lem:NonlinearBounds}
Fix $\alpha \in (0, 1)$ and define the remainder term
\begin{equation*}
    \mathcal{R}(Z) := \mathcal{N}(X_0  + Z) - \mathcal{N}(X_0 ) - \mathcal{N}_1(Z).
\end{equation*}
There exist constants $\varepsilon_0 > 0$ and $C > 0$ such that, if $\|Z\|_{C^{1, \alpha}} \le \varepsilon_0$, then
\begin{equation*}
    \|\mathcal{R}(Z)\|_{C^\alpha} \le C \|Z\|_{C^{1, \alpha}}^2.
\end{equation*}
Moreover, if $Z \in C^{2, \alpha}$, then the following tame estimate holds:
\begin{equation*}
    \|\mathcal{R}(Z)\|_{C^{1, \alpha}} \le C \|Z\|_{C^{1, \alpha}} \|Z\|_{C^{2, \alpha}}.
\end{equation*}
\end{lemma}

\begin{proof}
The expansion of the Peskin nonlinearity presented in Section \ref{sec:nonlinear} demonstrates that, after subtracting $\mathcal{N}(X_0 )$ and $\mathcal{N}_1(Z)$, every term in the remainder $\mathcal{R}(Z)$ is at least quadratic in $Z$ and can be represented as a finite linear combination of the multilinear forms addressed in Lemma \ref{Lem:PeskinSchauder}. The coefficients of this expansion are analytic functions of the divided differences
\begin{equation*}
    \widetilde{Z}_{k, j}(x, y) = \frac{(x^k - y^k)(Z^j(x) - Z^j(y))}{|x - y|^2}.
\end{equation*}
For sufficiently small $\|Z\|_{C^{1, \alpha}}$, these analytic expansions converge absolutely.

Applying Lemma \ref{Lem:PeskinSchauder} term by term, we obtain the estimate
\begin{equation*}
    \|\mathcal{R}(Z)\|_{C^\alpha} \le \sum_{m \ge 2} C^m \|Z\|_{C^{1, \alpha}}^m \le C \|Z\|_{C^{1, \alpha}}^2.
\end{equation*}
If $Z \in C^{2, \alpha}$, the tame estimate established in Lemma \ref{Lem:PeskinSchauder} yields
\begin{equation*}
    \|\mathcal{R}(Z)\|_{C^{1, \alpha}} \le \sum_{m \ge 2} C^m \|Z\|_{C^{2, \alpha}} \|Z\|_{C^{1, \alpha}}^{m-1} \le C \|Z\|_{C^{1, \alpha}} \|Z\|_{C^{2, \alpha}},
\end{equation*}
where the final inequality holds due to the smallness of $\|Z\|_{C^{1, \alpha}}$.
\end{proof}

For $t\ge1$ the solution has already smoothed out. Therefore the tame
$C^{1,\alpha}$ estimate in Lemma \ref{Lem:NonlinearBounds} may be applied on $[1,\infty)$, with the required higher norm controlled by the smoothing bounds obtained in the local theory.

\begin{lemma}[Modulated equation and quadratic parameter bound]
\label{Lem:ParameterVelocityBound}
Let $X$ be decomposed as in Lemma \ref{Lem:ParameterDecoupling}:
\begin{equation*}
    X(\alpha, t) = Z_\rho(\mathbf{\mathcal{S}}_{\mathbf{q}(t)}(\alpha), t) - \frac{2}{3}\mathbf{v}(\mathbf{q}(t)) + \mathbf{a}(t), \qquad Z_\rho(\beta, t) = (1 + \rho(t))\beta + Y(\beta, t),
\end{equation*}
where
\begin{equation*}
    Y(t) \in (\ker \mathcal{N}_1)^\perp.
\end{equation*}
Set $r(t) = 1 + \rho(t)$, $\mathbf{p}(t) = (\rho(t), \mathbf{a}(t), \mathbf{q}(t))$, and let $\Pi_\perp$ denote the $L^2$-orthogonal projection onto $(\ker \mathcal{N}_1)^\perp$. Then $Y$ satisfies
\begin{equation*}
    \partial_t Y = r(t)^{-1} \mathcal{N}_1 Y + \Pi_\perp \mathcal{E}(Y, \mathbf{p}),
\end{equation*}
where
\begin{equation*}
    \|\mathcal{E}(Y(t), \mathbf{p}(t))\|_{C^{1, \alpha}} \le C \|Y(t)\|_{C^{1, \alpha}}^2.
\end{equation*}
Moreover, the modulation parameters satisfy the quadratic bound
\begin{equation*}
    |\mathbf{\dot{p}}(t)| \le C \|Y(t)\|_{C^{1, \alpha}}^2.
\end{equation*}
\end{lemma}

\begin{proof}
Let 
\begin{equation*}
    \beta = \mathcal{S}_{\mathbf{q}(t)}(\alpha).
\end{equation*}
By the translation invariance and conformal covariance of the Peskin operator,
\begin{equation*}
    \mathcal{N}(X)(\alpha, t) = \mathcal{N}(Z_\rho)(\beta, t).
\end{equation*}
Since $Z_\rho(\beta, t) = r(t)\beta + Y(\beta, t)$, and since the Peskin operator is invariant under simultaneous dilations of the interface, we have
\begin{equation*}
    \mathcal{N}(rX_0  + Y) = \mathcal{N}(X_0  + r^{-1}Y).
\end{equation*}
Expanding around the unit sphere gives
\begin{equation*}
    \mathcal{N}(rX_0  + Y) = r^{-1}\mathcal{N}_1 Y + \mathcal{R}_r(Y),
\end{equation*}
where, uniformly for $r$ close to $1$,
\begin{equation*}
    \|\mathcal{R}_r(Y)\|_{C^{1, \alpha}} \le C\|Y\|_{C^{1, \alpha}}^2.
\end{equation*}
This follows from Lemma \ref{Lem:NonlinearBounds}, applied to $r^{-1}Y$.

We now differentiate the decomposition of $X$. Using
\begin{equation*}
    \partial_t \beta = \mathcal{T}_{i, q}(\beta)\mathbf{\dot{q}}_i,
\end{equation*}
we obtain
\begin{equation*}
    \partial_t X = \partial_t Y + \dot{\rho}\beta + r\mathcal{T}_{i, \mathbf{q}}(\beta)\mathbf{\dot{q}}_i + \langle \mathcal{T}_{i, \mathbf{q}}(\beta), \nabla_{\mathbb{S}^2}Y(\beta, t) \rangle \mathbf{\dot{q}}_i - \frac{2}{3}\dot{\mathbf{v}} + \mathbf{\dot{a}}.
\end{equation*}
We write the finite-dimensional terms schematically as
\begin{equation*}
    \mathcal{G}(p, Y)\mathbf{\dot{p}}.
\end{equation*}
Thus,
\begin{equation*}
    \partial_t X = \partial_t Y + \mathcal{G}(\mathbf{p}, Y)\mathbf{\dot{p}}.
\end{equation*}
Combining this identity with $\partial_t X = \mathcal{N}(X)$, we get
\begin{equation*}
    \partial_t Y + \mathcal{G}(p, Y)\mathbf{\dot{p}} = r^{-1}\mathcal{N}_1 Y + \mathcal{R}_r(Y).
\end{equation*}

We now project this equation onto the kernel of $\mathcal{N}_1$. Let \(\Pi_0:=I-\Pi_\perp\) denote the \(L^2\)-projection onto \(\ker\mathcal N_1\). Since
\begin{equation*}
    Y(t) \in (\ker \mathcal{N}_1)^\perp
\end{equation*}
for every $t$, we have
\begin{equation*}
    \Pi_0 \partial_t Y = 0.
\end{equation*}
Also, because $\mathcal{N}_1$ is self-adjoint and $Y \in (\ker \mathcal{N}_1)^\perp$, it follows that
\begin{equation*}
    \Pi_0 \mathcal{N}_1 Y = 0.
\end{equation*}
Therefore, applying the projection yields
\begin{equation*}
    \Pi_0 \mathcal{G}(p, Y)\dot{p} = \Pi_0 \mathcal{R}_r(Y).
\end{equation*}
At $(\mathbf{p}, Y) = (0, 0)$, the matrix $\Pi_0 \mathcal{G}(p, Y)$ is exactly the Jacobian matrix from Lemma \ref{Lem:ParameterDecoupling}, which is invertible. Hence, for $|p| + \|Y\|_{C^{1, \alpha}}$ sufficiently small, it remains uniformly invertible, and we obtain the bound
\begin{equation*}
    |\mathbf{\dot{p}}| \le C\|\mathcal{R}_r(Y)\|_{C^{1, \alpha}} \le C\|Y\|_{C^{1, \alpha}}^2.
\end{equation*}

Returning to the full equation, we have
\begin{equation*}
    \partial_t Y = r^{-1}\mathcal{N}_1 Y + \mathcal{R}_r(Y) - \mathcal{G}(\mathbf{p}, Y)\mathbf{\dot{p}}.
\end{equation*}
Define the remainder term
\begin{equation*}
    \mathcal{E}(Y, \mathbf{p}) := \mathcal{R}_r(Y) - \mathcal{G}(\mathbf{p}, Y)\mathbf{\dot{p}}.
\end{equation*}
Using the quadratic bound for $\dot{p}$, we obtain
\begin{equation*}
    \|\mathcal{E}(Y, \mathbf{p})\|_{C^{1, \alpha}} \le C\|Y\|_{C^{1, \alpha}}^2.
\end{equation*}
Finally, since both $\partial_t Y$ and $\mathcal{N}_1 Y$ lie in $(\ker \mathcal{N}_1)^\perp$, the right-hand side may be projected onto the stable subspace to yield
\begin{equation*}
    \partial_t Y = r^{-1}\mathcal{N}_1 Y + \Pi_\perp \mathcal{E}(Y, \mathbf{p}).
\end{equation*}
This completes the proof.
\end{proof}

To close the global bootstrap argument, we require time-decay estimates for the linear semigroup. The following corollary translates the high-frequency Littlewood-Paley bounds of Lemma \ref{tAop} into our time-weighted Hölder framework.

\begin{corollary}[Linear decay estimates]
\label{Cor:LinearDecay}
Fix $\alpha \in (0, 1)$ and $0 < \lambda  < a_2$. For any $T > 1$, define the time-weighted norm
\begin{equation*}
    \|F\|_{\mathcal{X}_T} := \sup_{t \in [1, T)} e^{\lambda  t} \|F(t)\|_{C^{1, \alpha}}.
\end{equation*}
If $Y \in (\ker \mathcal{N}_1)^\perp \cap C^{1, \alpha}(\mathbb{S}^2)$, then the linear semigroup exhibits the exponential decay bound
\begin{equation*}
    \|e^{\tau \mathcal{N}_1} Y\|_{C^{1, \alpha}} \le C e^{-\lambda  \tau} \|Y\|_{C^{1, \alpha}} \quad \text{for all } \tau \ge 0.
\end{equation*}
In particular, the evolution of the initial perturbation $Y(1)$ satisfies
\begin{equation*}
    \|e^{(t-1) \mathcal{N}_1} Y(1)\|_{\mathcal{X}_T} \le C_L \|Y(1)\|_{C^{1, \alpha}}.
\end{equation*}
\end{corollary}

\begin{proof}
Let $\tau \ge 0$. We use the Besov space characterization
\begin{equation*}
    \|F\|_{C^{1, \alpha}} \simeq \sup_{a \in \mathbb{Z}_+} 2^{a(1+\alpha)} \|\mathcal{P}'_a F\|_{L^\infty}.
\end{equation*}
Fix $a \in \mathbb{Z}_+$. We distinguish two cases based on the frequency $2^a$ and time $\tau$.

\textbf{Case 1: $\tau 2^a \lesssim 1$.} By Lemma \ref{tAop}(i), we have
\begin{equation*}
    \|\mathcal{P}'_a(e^{\tau \mathcal{N}_1} - I)Y\|_{L^\infty} \lesssim (\tau 2^a)^{1/2} \|Y_{a;3}\|_{L^\infty} \lesssim \|Y_{a;3}\|_{L^\infty}.
\end{equation*}
Therefore,
\begin{equation*}
    \|\mathcal{P}'_a e^{\tau \mathcal{N}_1} Y\|_{L^\infty} \le \|\mathcal{P}'_a Y\|_{L^\infty} + \|\mathcal{P}'_a(e^{\tau \mathcal{N}_1} - I)Y\|_{L^\infty} \lesssim \|Y_{a;3}\|_{L^\infty}.
\end{equation*}
Since $\tau 2^a \lesssim 1$ and $a \ge 0$, it follows that $\tau \lesssim 1$. Consequently, after adjusting the implicit constant, we obtain
\begin{equation*}
    \|\mathcal{P}'_a e^{\tau \mathcal{N}_1} Y\|_{L^\infty} \lesssim e^{-\lambda  \tau} \|Y_{a;3}\|_{L^\infty}.
\end{equation*}

\textbf{Case 2: $\tau 2^a \gtrsim 1$.} Since $Y \in (\ker \mathcal{N}_1)^\perp$, Lemma \ref{tAop}(ii) implies
\begin{equation*}
    \|\mathcal{P}'_a e^{\tau \mathcal{N}_1} Y\|_{L^\infty} \lesssim (\tau 2^a)^{-4} e^{-\lambda  \tau} \|Y_{a;3}\|_{L^\infty} \lesssim e^{-\lambda  \tau} \|Y_{a;3}\|_{L^\infty}.
\end{equation*}

Combining these cases and multiplying by $2^{a(1+\alpha)}$, we take the supremum over $a \in \mathbb{Z}_+$ to find
\begin{equation*}
    \|e^{\tau \mathcal{N}_1} Y\|_{C^{1, \alpha}} \lesssim e^{-\lambda  \tau} \sup_{a \in \mathbb{Z}_+} 2^{a(1+\alpha)} \|Y_{a;3}\|_{L^\infty}.
\end{equation*}
Recalling that $Y_{a;3} := \sum_{|a'-a| \le 3} \mathcal{P}'_{a'} Y$, the supremum is bounded by $\|Y\|_{C^{1, \alpha}}$. Thus,
\begin{equation*}
    \|e^{\tau \mathcal{N}_1} Y\|_{C^{1, \alpha}} \le C e^{-\lambda  \tau} \|Y\|_{C^{1, \alpha}}.
\end{equation*}

Setting $\tau = t - 1$, we arrive at
\begin{equation*}
    \|e^{(t-1) \mathcal{N}_1} Y(1)\|_{C^{1, \alpha}} \le C e^{-\lambda  (t-1)} \|Y(1)\|_{C^{1, \alpha}} = C e^{\lambda } e^{-\lambda  t} \|Y(1)\|_{C^{1, \alpha}}.
\end{equation*}
Multiplying by $e^{\lambda  t}$ and taking the supremum over $t \in [1, T)$ yields the $\mathcal{X}_T$ bound, with $C_L = C e^{\lambda }$.
\end{proof}

\begin{lemma}[Nonlinear Integral Bound] \label{Lem:DuhamelBound}
	Let $0<\lambda <a_2/r_\infty$, with $a_2=8/35$ \eqref{tete4}, and let $Y(1) \in (\ker \mathcal{N}_1)^\perp \cap C^{1,\alpha}(\mathbb{S}^2)$. Define the truncated time-weighted norm on any sub-interval $[1, T)$ as
	\begin{equation}
	\|Y\|_{\mathcal{X}_T} = \sup_{t \in [1, T)} \Big( e^{\lambda t} \|Y(\cdot, t)\|_{C^{1,\alpha}(\mathbb{S}^2)} \Big).
	\end{equation}
	 Then we have
	\begin{equation}
	\left\| \int_1^t U(t,s)\Pi_\perp\mathcal E(Y(s),\mathbf{p}(s))\,ds \right\|_{\mathcal{X}_T} \le C_I \|Y\|_{\mathcal{X}_T}^2 \text{,}
	\end{equation}
	for all $t \in [1, T)$ and a uniform constant $C_I > 0$, where here 
    \begin{equation*}
        U(t,s):=\exp\left(\left(\int_s^t (1+\rho(\tau))^{-1}\,d\tau\right)\mathcal N_1\right).
    \end{equation*}
\end{lemma}

\begin{proof}
On the interval under consideration, we assume $r(t)=1+\rho(t)$ remains uniformly close to $1$.

By Lemma \ref{Lem:ParameterVelocityBound}, the decoupled perturbation $Y$ satisfies the evolution equation
\begin{equation*}
    \partial_t Y = (1 + \rho(t))^{-1} \mathcal{N}_1 Y + \Pi_\perp \mathcal{E}(Y, \mathbf{p}),
\end{equation*}
with the quadratic nonlinearity bound
\begin{equation*}
    \|\mathcal{E}(Y(t), \mathbf{p}(t))\|_{C^{1, \alpha}} \le C \|Y(t)\|_{C^{1, \alpha}}^2.
\end{equation*}
Applying the Duhamel formula yields
\begin{equation*}
    Y(t) = U(t, 1)Y(1) + \int_1^t U(t, s) \Pi_\perp \mathcal{E}(Y(s), \mathbf{p}(s)) \, ds,
\end{equation*}
where the evolution operator is defined by
\begin{equation*}
    U(t, s) = \exp \left( \left( \int_s^t (1 + \rho(\tau))^{-1} \, d\tau \right) \mathcal{N}_1 \right).
\end{equation*}

Since $|\rho(t)|$ remains small, for any fixed $0 < \lambda < a_2/r_\infty$, we may shrink the smallness threshold such that
\begin{equation*}
    \int_s^t (1 + \rho(\tau))^{-1} \, d\tau \ge c(t - s)
\end{equation*}
with $c\lambda_* > \lambda$ for some $\lambda_* < a_2/r_\infty$. Consequently, Corollary \ref{Cor:LinearDecay} implies the decay estimate
\begin{equation*}
    \|U(t, s)W\|_{C^{1, \alpha}} \le C e^{-\lambda(t - s)} \|W\|_{C^{1, \alpha}}
\end{equation*}
for every $W \in (\ker \mathcal{N}_1)^\perp$.

Since $\Pi_\perp \mathcal{E}(Y(s), \mathbf{p}(s)) \in (\ker \mathcal{N}_1)^\perp$, we apply the decay estimate above to obtain
\begin{align*}
    \left\| \int_1^t U(t, s) \Pi_\perp \mathcal{E}(Y(s), \mathbf{p}(s)) \, ds \right\|_{C^{1, \alpha}} &\le C \int_1^t e^{-\lambda(t - s)} \|\mathcal{E}(Y(s), \mathbf{p}(s))\|_{C^{1, \alpha}} \, ds \\
    &\le C \int_1^t e^{-\lambda(t - s)} \|Y(s)\|_{C^{1, \alpha}}^2 \, ds \\
    &\le C \|Y\|_{\mathcal{X}_T}^2 \int_1^t e^{-\lambda(t - s)} e^{-2\lambda s} \, ds \\
    &\le C e^{-\lambda t} \|Y\|_{\mathcal{X}_T}^2.
\end{align*}
Multiplying by $e^{\lambda t}$ and taking the supremum over $t \in [1, T)$ gives
\begin{equation*}
    \left\| \int_1^t U(t,s) \Pi_\perp \mathcal{E}(Y(s), \mathbf{p}(s)) \, ds \right\|_{\mathcal{X}_T} \le C_I \|Y\|_{\mathcal{X}_T}^2.
\end{equation*}
\end{proof}

\begin{lemma}[Global Extension]
\label{Lem:BootstrapExtension}
Let $Y \in C([1, T^*); C^{1, \alpha}(\mathbb{S}^2))$ be a maximal solution of
\begin{equation*}
    \partial_t Y = (1 + \rho(t))^{-1} \mathcal{N}_1 Y + \Pi_\perp \mathcal{E}(Y, \mathbf{p}),
\end{equation*}
satisfying $Y(t) \in (\ker \mathcal{N}_1)^\perp$, with maximal time of existence $T^* > 1$. Let
\begin{equation*}
    U(t, s) := \exp \left( \left( \int_s^t (1 + \rho(\tau))^{-1} \, d\tau \right) \mathcal{N}_1 \right).
\end{equation*}
Assume that the linear evolution and the nonlinear Duhamel integral satisfy the time-weighted bounds
\begin{equation*}
    \|U(t, 1)Y(1)\|_{\mathcal{X}_T} \le C_L \|Y(1)\|_{C^{1, \alpha}}
\end{equation*}
and
\begin{equation*}
    \left\| \int_1^t U(t, s) \Pi_\perp \mathcal{E}(Y(s), \mathbf{p}(s)) \, ds \right\|_{\mathcal{X}_T} \le C_I \|Y\|_{\mathcal{X}_T}^2, 
\end{equation*}
for all $T \in (1, T^*)$, with uniform constants $C_I > 0$, $C_L \ge e^\lambda > 0$, and $0 < \lambda < a_2/r_\infty$, with $a_2 = 8/35$. 

For $\|Y(1)\|_{C^{1, \alpha}} \le \varepsilon_1$ with sufficiently small $\varepsilon_1$, the solution exists globally ($T^* = \infty$) and satisfies the uniform infinite-time bound
\begin{equation*}
    \|Y\|_{\mathcal{X}} \le 2C_L \varepsilon_1.
\end{equation*}
\end{lemma}

\begin{proof}
We employ a bootstrap argument. Define the set of times $S \subseteq [1, T^*)$ as
\begin{equation*}
    S = \left\{ T \in [1, T^*) : \|Y\|_{\mathcal{X}_T} \le 2C_L \varepsilon_1 \right\}.
\end{equation*}
First, we verify that $S$ is non-empty. At the initial time $T=1$, the time weight is $e^\lambda$, so the norm evaluates to
\begin{equation*}
    e^\lambda \|Y(1)\|_{C^{1, \alpha}} \le e^\lambda \varepsilon_1.
\end{equation*}
Since $C_L \ge e^\lambda$, we have $e^\lambda \varepsilon_1 \le 2C_L \varepsilon_1$, so $1 \in S$. By the continuity of the solution in $C^{1, \alpha}(\mathbb{S}^2)$, $S$ forms a closed interval.

Assume $T \in S$. Applying the truncated $\mathcal{X}_T$ norm to the Duhamel integral formulation
\begin{equation*}
    Y(t) = U(t, 1)Y(1) + \int_1^t U(t, s) \Pi_\perp \mathcal{E}(Y(s), \mathbf{p}(s)) \, ds,
\end{equation*}
and using Corollary \ref{Cor:LinearDecay} and the Duhamel estimate, we obtain
\begin{equation*}
    \|Y\|_{\mathcal{X}_T} \le C_L \|Y(1)\|_{C^{1, \alpha}} + C_I \|Y\|_{\mathcal{X}_T}^2 \le C_L \varepsilon_1 + C_I \|Y\|_{\mathcal{X}_T}^2. 
\end{equation*}
Because $T \in S$, we have
\begin{equation*}
    \|Y\|_{\mathcal{X}_T} \le C_L \varepsilon_1 + C_I(2C_L \varepsilon_1)^2 = C_L \varepsilon_1 + 4C_L^2 C_I \varepsilon_1^2. 
\end{equation*}
Defining $\varepsilon_0 = (8C_L C_I)^{-1}$, for any $\varepsilon_1 \le \varepsilon_0$, we obtain
\begin{equation*}
    \|Y\|_{\mathcal{X}_T} \le C_L \varepsilon_1 + \frac{1}{2}C_L \varepsilon_1 = \frac{3}{2}C_L \varepsilon_1 < 2C_L \varepsilon_1. 
\end{equation*}
This strict inequality shows that the norm never reaches the boundary $2C_L \varepsilon_1$. By continuity, the bound must hold on an open neighborhood extending beyond $T$. Thus, $S$ is both closed and open in $[1, T^*)$. By connectedness, $S = [1, T^*)$.

Because the time-weighted $\mathcal{X}_T$ norm remains uniformly bounded, the physical norm $\|Y(t)\|_{C^{1, \alpha}}$ cannot blow up in finite time. By standard continuation criteria, the maximal time of existence is infinite ($T^* = \infty$), completing the proof.
\end{proof}

\begin{lemma}[Asymptotic Parameter Convergence] \label{Lem:ParameterConvergence}
	Let $$\mathbf{p}(t) = (\rho(t), \mathbf{a}(t), \mathbf{q}(t))^\top \in C^1([1, \infty); \mathbb{R}^{10})$$ be a solution to the globally decoupled finite-dimensional system. Suppose $\|Y\|_{\mathcal{X}} \le C$, which guarantees $\|Y(t)\|_{C^{1,\alpha}(\mathbb{S}^2)} \le C e^{-\lambda t}$, where $\lambda$ is as in the previous lemmas.
	
	Then, as $t \to \infty$, the parameter vector $\mathbf{p}(t)$ converges to a unique constant steady-state configuration $\mathbf{p}_\infty \in \mathbb{R}^{10}$, satisfying the exponential convergence bound
	\begin{equation}
	|\mathbf{p}(t) - \mathbf{p}_\infty| \le \frac{C_p}{2\lambda} \|Y\|_{\mathcal{X}}^2 e^{-2\lambda t} \text{,}
	\end{equation}
	for all $t \ge 1$, where $C_p > 0$ is the uniform constant from Lemma \ref{Lem:ParameterVelocityBound}.
\end{lemma}

\begin{proof}
	By combining $\|Y(t)\|_{C^{1,\alpha}} \le \|Y\|_{\mathcal{X}} e^{-\lambda t}$ with the bounds of Lemma \ref{Lem:ParameterVelocityBound}, we obtain 
	\begin{equation} \label{eq:SpeedBound}
	|\dot{\mathbf{p}}(t)| \le C_p \|Y(t)\|_{C^{1,\alpha}}^2 \le C_p \|Y\|_{\mathcal{X}}^2 e^{-2\lambda t},
	\end{equation}
	implying $\dot{\mathbf{p}} \in L^1([1, \infty); \mathbb{R}^{10})$. 
	
	We now apply the Fundamental Theorem of Calculus. For any two times $t_2 > t_1 \ge 1$, we get
	\begin{equation} \label{eq:CauchyBound}
	|\mathbf{p}(t_2) - \mathbf{p}(t_1)| \le \int_{t_1}^{t_2} |\dot{\mathbf{p}}(s)| \, ds \le \int_{t_1}^\infty C_p \|Y\|_{\mathcal{X}}^2 e^{-2\lambda s} \, ds = \frac{C_p}{2\lambda }\|Y\|_{\mathcal{X}}^2 e^{-2\lambda t}.
	\end{equation}
	As $t_1 \to \infty$, the right-hand side of \eqref{eq:CauchyBound} vanishes. Therefore, $\mathbf{p}(t)$ forms a Cauchy sequence in the complete Banach space $\mathbb{R}^{10}$, guaranteeing convergence to a unique constant vector $\mathbf{p}_\infty$.
	
	Taking the limit as $t_2 \to \infty$ in \eqref{eq:CauchyBound} directly yields the explicit algebraic convergence rate
	\begin{equation}
	|\mathbf{p}_\infty - \mathbf{p}(t)| \le \frac{C_p}{2\lambda } \|Y\|_{\mathcal{X}}^2 e^{-2\lambda t},
	\end{equation}
	completing the proof.
    
\end{proof}

The previous argument proves exponential convergence at every rate below $a_2/r_\infty$. We now identify the leading asymptotic mode and show that, unless the projection onto the first stable eigenspace vanishes, the decay rate is exactly $a_2$ in rescaled time $\sigma(t)$, equivalently, $a_2/r_\infty$, in $L^2$.

\begin{proposition}[Sharp decay rate]\label{prop:decay_exact}
Let \(a_2=8/35\) be the first positive eigenvalue of \(-\mathcal N_1\) on \((\ker\mathcal N_1)^\perp\). Let
\begin{equation*}
    E_2:=\operatorname{span}\{H^1_{3,m}:-2\le m\le 2\}
\end{equation*}
be the corresponding eigenspace, and let \(\Pi_2\) denote the \(L^2\)-orthogonal projection onto \(E_2\). Define
\begin{equation*}
    \sigma(t):=\int_1^t r(s)^{-1}\,ds, \qquad r(t)=1+\rho(t).
\end{equation*}
Assume that the decoupled perturbation \(Y\) satisfies
\begin{equation*}
    \partial_tY=r(t)^{-1}\mathcal N_1Y+\mathcal E(t), \qquad t\ge1,
\end{equation*}
where the error term \(\mathcal E\) satisfies the quadratic estimate
\begin{equation*}
    \|\mathcal E(t)\|_{L^2(\mathbb S^2)} \le C\|Y(t)\|_{C^{1,\alpha}(\mathbb S^2)}^2.
\end{equation*}
Assume moreover that, for some \(\mu\in(a_2/2,a_2)\),
\begin{equation*}
    \|Y(t)\|_{C^{1,\alpha}(\mathbb S^2)} \le C_\mu\varepsilon e^{-\mu\sigma(t)}, \qquad t\ge1. 
\end{equation*}
Then there exists a vector \(Y_{2,\infty}\in E_2\) such that
\begin{equation}\label{sharp_decay01}
    \|\Pi_2Y(t)-e^{-a_2\sigma(t)}Y_{2,\infty}\|_{L^2} \lesssim e^{-2\mu\sigma(t)}, 
\end{equation}
and, for some \(\gamma>a_2\),
\begin{equation}\label{sharp_decay02}
    \|(I-\Pi_2)Y(t)\|_{L^2} \lesssim_\gamma e^{-\gamma\sigma(t)}. 
\end{equation}
Consequently, if \(Y_{2,\infty}\neq0\), then
\begin{equation*}
    \lim_{t\to\infty} e^{a_2\sigma(t)}\|Y(t)\|_{L^2} = \|Y_{2,\infty}\|_{L^2} > 0. 
\end{equation*}
In particular, if \(r(t)\to r_\infty\), then \(\sigma(t)=r_\infty^{-1}t+O(1)\), so the generic sharp physical-time decay rate is \(a_2/r_\infty\).
\end{proposition}

\begin{proof}
We decompose the solution into its first stable spectral component and its orthogonal complement:
\begin{equation*}
    Y(t)=\Pi_2Y(t)+(I-\Pi_2)Y(t).
\end{equation*}
Set \(Z(t)=\Pi_2Y(t)\). By definition of the projection \(\Pi_2\), we have
\begin{equation*}
    \Pi_2\mathcal N_1Y=-a_2\Pi_2Y.
\end{equation*}
Projecting the equation onto \(E_2\), we obtain
\begin{equation*}
    \partial_tZ(t)=-a_2r(t)^{-1}Z(t)+\Pi_2\mathcal E(t). 
\end{equation*}
Since \(\sigma'(t)=r(t)^{-1}\), this gives
\begin{equation}\label{z_eq}
    \frac{d}{dt}\left(e^{a_2\sigma(t)}Z(t)\right) = e^{a_2\sigma(t)}\Pi_2\mathcal E(t).
\end{equation}
By the quadratic estimate on \(\mathcal E\) from Lemmas \ref{Lem:NonlinearBounds} and \ref{Lem:ParameterVelocityBound} and the decay assumption in Lemma \ref{Lem:DuhamelBound},
\begin{equation*}
    \|\mathcal E(t)\|_{L^2} \le C\|Y(t)\|_{C^{1,\alpha}}^2 \le C_\mu\varepsilon^2e^{-2\mu\sigma(t)}.
\end{equation*}

Since \(\mu>a_2/2\), we have \(2\mu-a_2>0\). Since $r(t)$ remains uniformly close to $1$, $\sigma'(t)=r(t)^{-1}$ is bounded above and below by positive constants; hence, integrals in $t$ and in $\sigma$ are equivalent up to constants. Therefore
\begin{equation*}
    \int_1^\infty e^{a_2\sigma(s)}\|\Pi_2\mathcal E(s)\|_{L^2}\,ds \lesssim \varepsilon^2 \int_1^\infty e^{-(2\mu-a_2)\sigma(s)}\,ds < \infty.
\end{equation*}
Define
\begin{equation*}
    Y_{2,\infty} = \Pi_2Y(1) + \int_1^\infty e^{a_2\sigma(s)}\Pi_2\mathcal E(s)\,ds \in E_2. 
\end{equation*}
Integrating \eqref{z_eq} gives
\begin{equation*}
    e^{a_2\sigma(t)}Z(t) = Y_{2,\infty} - \int_t^\infty e^{a_2\sigma(s)}\Pi_2\mathcal E(s)\,ds.
\end{equation*}
Therefore,
\begin{align*}
    \|Z(t)-e^{-a_2\sigma(t)}Y_{2,\infty}\|_{L^2} &\le e^{-a_2\sigma(t)} \int_t^\infty e^{a_2\sigma(s)}\|\Pi_2\mathcal E(s)\|_{L^2}\,ds \\
    &\lesssim e^{-a_2\sigma(t)} \int_t^\infty e^{-(2\mu-a_2)\sigma(s)}\,ds \\
    &\lesssim e^{-2\mu\sigma(t)}.
\end{align*}
This proves \eqref{sharp_decay01}.

We now estimate the complementary component. Set
\begin{equation*}
    Y(t):=(I-\Pi_2)Y(t).
\end{equation*}
By Duhamel's formula,
\begin{equation*}
    Y(t) = e^{\sigma(t)\mathcal N_1}(I-\Pi_2)Y(1) + \int_1^t e^{(\sigma(t)-\sigma(s))\mathcal N_1} (I-\Pi_2)\mathcal E(s)\,ds.
\end{equation*}
On the subspace \((I-\Pi_2)(\ker\mathcal N_1)^\perp\), the spectrum of \(-\mathcal N_1\) is bounded from below by some \(a_2^+>a_2\). Hence, for every \(\gamma<a_2^+\),
\begin{equation*}
    \|e^{(\sigma(t)-\sigma(s))\mathcal N_1}(I-\Pi_2)\mathbf{F}\|_{L^2} \lesssim_\gamma e^{-\gamma(\sigma(t)-\sigma(s))}\|\mathbf{F}\|_{L^2}.
\end{equation*}
Using again the quadratic bound on \(\mathcal E\), we obtain
\begin{align*}
    \|Y(t)\|_{L^2} &\lesssim_\gamma e^{-\gamma\sigma(t)}\|Y(1)\|_{L^2} + \int_1^t e^{-\gamma(\sigma(t)-\sigma(s))} \|\mathcal E(s)\|_{L^2}\,ds \\
    &\lesssim_\gamma e^{-\gamma\sigma(t)}\|Y(1)\|_{L^2} + \varepsilon^2 \int_1^t e^{-\gamma(\sigma(t)-\sigma(s))} e^{-2\mu\sigma(s)}\,ds.
\end{align*}

If, in addition, \(\gamma<2\mu\), then
\begin{equation*}
    \int_1^t e^{-\gamma(\sigma(t)-\sigma(s))} e^{-2\mu\sigma(s)}\,ds \lesssim_\gamma e^{-\gamma\sigma(t)}.
\end{equation*}
Since \(a_2^+>a_2\) and \(2\mu>a_2\), we may choose \(a_2<\gamma<\min\{a_2^+,2\mu\}\). This proves \eqref{sharp_decay01}.

Combining the estimate for \(\Pi_2Y\) with the estimate for \((I-\Pi_2)Y\), we get
\begin{equation*}
    \lim_{t\to\infty} e^{a_2\sigma(t)}\|Y(t)\|_{L^2} = \|Y_{2,\infty}\|_{L^2},
\end{equation*}
provided \(Y_{2,\infty}\neq0\). This completes the proof.
\end{proof}

\begin{remark}
    In Section \ref{sec:numerical}, we numerically verify that the solution converges to the steady state manifold at exponential rate with exponent equal to $a_2/r_\infty$ for the deviation $Y$ and $2a_2/r_\infty$ for the parameter vector $\mathbf{p}$. For simplicity, in Section \ref{sec:numerical} the results are shown with $r_\infty=1$.
\end{remark}

We are now ready to finish the proof of our main theorem. 

\begin{proof}[Proof of Theorem \ref{MainTHM}(ii)]
By part (i), the solution exists on $[0,1]$ and satisfies
\begin{equation*}
    \|Y\|_{Z[0,1]} \lesssim \varepsilon.
\end{equation*}
The parabolic smoothing effect inherent in the local theory implies that for every fixed $\alpha \in (0, 1)$,
\begin{equation*}
    \|X(1) - X_0 \|_{C^{1, \alpha}(\mathbb{S}^2)} \le C\varepsilon.
\end{equation*}
By choosing $\varepsilon > 0$ sufficiently small, Lemma \ref{Lem:ParameterDecoupling} applies at time $t = 1$. Let $T^* > 1$ denote the maximal time for which the solution remains within the small $C^{1, \alpha}$ neighborhood where the structural decomposition is valid. For $t \in [1, T^*)$, we express the interface as
\begin{equation*}
    X(\mathbf{\alpha}, t) = Z_\rho(\mathbf{\mathcal{S}}_{\mathbf{q(t)}}(\mathbf{\alpha}), t) - \frac{2}{3}\mathbf{v}(t) + \textbf{a}(t), \qquad Z_\rho(\mathbf{\beta}, t) = (1 + \rho(t))\mathbf{\beta} + Y(\mathbf{\beta}, t),
\end{equation*}
subject to the orthogonality condition $Y(t) \in (\ker \mathcal{N}_1)^\perp$. The modulation equations derived in Lemma \ref{Lem:ParameterDecoupling}, in conjunction with Lemmas \ref{Lem:NonlinearBounds} and \ref{Lem:ParameterVelocityBound}, imply that the decoupled perturbation $Y$ satisfies a Duhamel integral equation of the form
\begin{equation*}
    Y(t) = U(t,1) Y(1) + \int_1^t U(t,s)\Pi_\perp \mathcal{E}(Y(s), p(s)) \, ds,
\end{equation*}
where $p = (\rho, a, q)$ and $\|\mathcal{E}(Y(t), p(t))\|_{C^{1, \alpha}} \lesssim \|Y(t)\|_{C^{1, \alpha}}^2$. Here, we use the tame $C^{1, \alpha}$ remainder estimate valid after the solution has smoothed at positive time.

Since $r(t)$ remains close to $1$, Corollary \ref{Cor:LinearDecay} applies to $U(t,s)$ with the same exponential bound, after decreasing the admissible rate $\lambda<a_2$ if necessary. Applying Corollary \ref{Cor:LinearDecay} and Lemma \ref{Lem:DuhamelBound}, the bootstrap argument in Lemma \ref{Lem:BootstrapExtension} yields $T^* = \infty$ and the uniform bound
\begin{equation*}
    \|Y\|_{\mathcal{X}} \le C \|Y(1)\|_{C^{1, \alpha}} \le C\varepsilon.
\end{equation*}
Consequently, we obtain the exponential decay estimate
\begin{equation*}
    \|Y(t)\|_{C^{1, \alpha}} \le C\varepsilon e^{-\lambda  t} \quad \text{for all } t \ge 1,
\end{equation*}
for any $0 < \lambda  < a_2/r_\infty$, with $a_2 = 8/35$. Furthermore, Lemma \ref{Lem:ParameterVelocityBound} implies
\begin{equation*}
    |\mathbf{\dot{p}}(t)| \lesssim \|Y(t)\|_{C^{1, \alpha}}^2 \lesssim \varepsilon^2 e^{-2\lambda  t}.
\end{equation*}
This shows $\mathbf{\dot{p}} \in L^1([1, \infty))$, and by Lemma \ref{Lem:ParameterConvergence}, there exists a limit $\mathbf{p}_\infty = (\rho_\infty, \mathbf{a}_\infty, \mathbf{q}_\infty)$ such that
\begin{equation*}
    |\mathbf{p}(t) - \mathbf{p}_\infty| \lesssim \varepsilon^2 e^{-2\lambda  t}.
\end{equation*}
Defining $\Lambda_\infty := \Lambda_{\mathbf{q}_\infty} \Lambda_\star$, $r_\infty := 1 + \rho_\infty$, and $b_\infty := \mathbf{a}_\infty - \frac{2}{3} \mathbf{v}(\mathbf{q}_\infty)$, the limiting steady state is given by $\mathcal{S}_{r_\infty, b_\infty, \Lambda_\infty}(\alpha) = b_\infty + r_\infty \mathcal{S}_{\Lambda_\infty}(\alpha)$. Since the mapping
\begin{equation*}
    (\rho, \mathbf{a}, \mathbf{q}) \mapsto (1 + \rho)\mathbf{\mathcal{S}}_\mathbf{q}(\alpha) - \frac{2}{3}\mathbf{v}(\mathbf{q}) + \mathbf{a}
\end{equation*}
is smooth in a neighborhood of the identity, we have
\begin{equation*}
    \|X(t) - \mathcal{S}_{r_\infty, b_\infty, \Lambda_\infty}\|_{C^1} \lesssim \|Y(t)\|_{C^1} + |\mathbf{p}(t) - \mathbf{p}_\infty| \lesssim \varepsilon e^{-\lambda  t}.
\end{equation*}
Thus, $X(t)$ converges in $C^{1,\alpha}$ (and hence in $C^1$) to the translated and dilated conformal steady state. Finally, uniqueness on $[0, \infty)$ follows from the local uniqueness on $[0, 1]$ combined with the standard continuation uniqueness for the smoothed equation on $[1, \infty)$. This concludes the proof of Theorem \ref{MainTHM}(ii).
\end{proof}

\section{Functional analysis on the sphere $\mathbb{S}^2$}\label{sec:appendix}

\subsection{Littlewood--Paley theory on $\mathbb{S}^2$} Recall the spectral Littlewood-Paley operators and the notation introduced in subsection \ref{LiPaley}. 

\begin{lemma}\label{LiPaLem}
Assume $a\in\Z_+$ and $m:\R\to\C$ is a $C^6$ function supported in the interval $[-2^a,2^a]$ and satisfying the bounds
\begin{equation}\label{lipa4.1}
2^{pa}|\partial^p_xm(x)|\leq 1 \qquad \text{ for any }x\in\R\text{ and }p\in\{0,\ldots,6\}.
\end{equation}
We define the operator $T_m$ and its kernel $H_m$ associated to the multiplier $m$ by
\begin{equation}\label{lipa4.2}
\begin{split}
&T_mY(x):=\sum_{k\in\Z_+}m(k)Y_k(x)=\int_{\SS^2}Y(y)H_m(x,y)\,d\mu(y),\\
&H_m(x,y):=\sum_{k\in\Z_+}m(k)Z_k(x,y).
\end{split}
\end{equation}
Then the kernel $H_m$ satisfy the identities and the bounds
\begin{equation}\label{lipa5}
H_m(x,y)=G_m(x\cdot y)\quad\text{ where }\quad G_m(\alpha):=\sum_{k\in\Z_+}m(k)F_k(\alpha),
\end{equation}
\begin{equation}\label{lipa5.5}
|H_m(x,y)|\lesssim 2^{2a}(1+2^a|x-y|)^{-3}\qquad\text{ for any }x,y\in\mathbb{S}^2.
\end{equation}
\end{lemma}

\begin{proof} Recall that $Z_k(x,y)=F_k(x\cdot y)$ (see \eqref{zon4}), so the identity \eqref{lipa5} follows from the definition of the kernel $H_m$.

To prove the estimates \eqref{lipa5.5} we use the identities \eqref{zon7}, \eqref{zon7.5}, and \eqref{zon7.6}. Notice that for any $\alpha\in[-1,1]$ the function $z\mapsto 1+z^2-2\alpha z=(z-\rho(\alpha))(z-\overline{\rho(\alpha)})$, where $\rho(\alpha):=\alpha+i\sqrt{1-\alpha^2}$, is analytic and does not vanish in the ball $B_1:=\{z\in\mathbb{C}:\,|z|<1\}$ (in fact $1+z^2-2\alpha z\in \C\setminus (-\infty,0]$ for any $z\in B_1$). In particular, the function $z\mapsto (1+z^2-2\alpha z)^{-1/2}$ is a well-defined analytic function in $B_1$, equal to $1$ where $z=0$ (corresponding to the principal branch of the function $\sqrt{}$). Letting $z=re^{i\theta}$ it follows from \eqref{zon7.5} that
\begin{equation}\label{lipa6}
(1+r^2e^{2i\theta}-2re^{i\theta} \alpha)^{-1/2}=\sum_{k\geq 0}e^{ik\theta}r^kP_k(\alpha),
\end{equation}
for any $\alpha\in[-1,1]$, $r\in[0,1)$ and $\theta\in[-\pi,\pi]$. In particular, using \eqref{zon7.6},
\begin{equation}\label{lipa7}
\begin{split}
r^kF_k(\alpha)&=\frac{2k+1}{8\pi^2}\int_{-\pi}^\pi e^{-ik\theta}(1+r^2e^{2i\theta}-2re^{i\theta} \alpha)^{-1/2}\,d\theta\qquad \text{ if }k\geq 0,\\
0&=\frac{2k+1}{8\pi^2}\int_{-\pi}^\pi e^{-ik\theta}(1+r^2e^{2i\theta}-2re^{i\theta} \alpha)^{-1/2}\,d\theta\qquad \text{ if }k\leq -1,
\end{split}
\end{equation}
for any $\alpha\in[-1,1]$, $r\in[0,1)$ and $k\in\Z$. We multiply these identities by $m(k)$ and sum over $k$, thus, for any $\alpha\in[-1,1]$ and $r\in[0,1)$,
\begin{equation}\label{lipa8}
\sum_{k\in\Z_+}m(k)r^kF_k(\alpha)=\frac{1}{8\pi^2}\int_{-\pi}^\pi \Big(\sum_{k\in\Z}e^{-ik\theta}(2k+1)m(k)\Big)(1+r^2e^{2i\theta}-2re^{i\theta} \alpha)^{-1/2}\,d\theta.
\end{equation}

We use the Poisson summation formula to write
\begin{equation}\label{lipa9}
s(\theta):=\sum_{k\in\Z}e^{-ik\theta}(2k+1)m(k)=\sum_{n\in\Z}\int_{\R}e^{-ix\theta}e^{-ix\cdot 2\pi n}(2x+1)m(x)\,dx=\sum_{n\in\Z}\chi(\theta+2\pi n)
\end{equation}
where
\begin{equation}\label{lipa10}
\begin{split}
&\chi(\rho):=\int_{\R}e^{-ix\rho}(2x+1)m(x)\,dx=\widehat{m}(\rho)+2i(\partial_\rho\widehat{m})(\rho),\\
&\widehat{m}(\xi):=\int_\R m(y)e^{-iy\xi}\,dy.
\end{split}
\end{equation}
The assumptions \eqref{lipa4.1} show that
\begin{equation}\label{lipa11}
|s(\theta)|+2^{-a}|\partial_\theta s(\theta)|\lesssim 2^{2a}(1+2^a|\theta|)^{-6}\text{ for any }\theta\in[-\pi,\pi].
\end{equation}

Let $\alpha=\cos\beta$ for some $\beta\in[0,\pi]$, so $1+r^2e^{2i\theta}-2re^{i\theta} \alpha=(1-re^{i(\theta-\beta)})(1-re^{i(\theta+\beta)})$. Since $|x-y|^2=2-2x\cdot y$, the bounds \eqref{lipa5.5} are equivalent to proving that $|G_{\leq a}(\alpha)|\lesssim 2^{2a}(1+2^{2a}(1-\alpha))^{-3/2}$ for any $\alpha\in[-1,1]$. In view of the identities \eqref{lipa9} and \eqref{lipa4}, it suffices to prove that
\begin{equation}\label{lipa15}
\Big|\int_{-\pi}^\pi s(\theta)\big[(1-re^{i(\theta-\beta)})(1-re^{i(\theta+\beta)})\big]^{-1/2}\,d\theta\Big|\lesssim 2^{2a}(1+2^a\beta)^{-3}
\end{equation}
for any $\beta\in[0,\pi]$ and $r\in[1-2^{-a-2},1)$. 

{\bf{Step 1.}} We bound first the contribution of $\theta$ away from $0$, and we prove the stronger estimates
\begin{equation}\label{lipa16}
\Big|\int_{-\pi}^\pi \varphi_{>-6}(\theta)s(\theta)\big[(1-re^{i(\theta-\beta)})(1-re^{i(\theta+\beta)})\big]^{-1/2}\,d\theta\Big|\lesssim 2^{-a}
\end{equation}
for any $\beta\in[0,\pi]$ and $r\in[1/2,1)$, where the cutoff function $\varphi_{>-6}$ is defined as in \eqref{lipa1}.

We record two elementary bounds: if $r\in[1/2,1)$ and $\gamma\in[-2\pi+1/100,2\pi-1/100]$ then
\begin{equation}\label{lipa17}
|1-re^{i\gamma}|^{-1/2}\lesssim (|1-r\cos\gamma|+r|\sin\gamma|)^{-1/2}\lesssim ((1-r)+|\gamma|)^{-1/2}.
\end{equation}
Moreover, if $x\geq y>0$ then we have the cancellation bounds
\begin{equation}\label{lipa18}
|(-x+iy)^{-1/2}+(-x-iy)^{-1/2}|\lesssim yx^{-3/2}.
\end{equation}
Indeed, to prove \eqref{lipa18}, we write $-x+iy=re^{i(\pi-\gamma)}$, $-x-iy=re^{-i(\pi-\gamma)}$, where $r=\sqrt{a^2+b^2}$ and $\tan\gamma=y/x$, $\gamma\in(0,\pi/4]$. By the definition of the principal branch of the function $\sqrt{}$ we thus have
\begin{equation*}
|(-x+iy)^{-1/2}+(-x-iy)^{-1/2}|=r^{-1/2}|e^{-i(\pi-\gamma)/2}+e^{i(\pi-\gamma)/2}|=r^{-1/2}|e^{i\gamma/2}-e^{-i\gamma/2}|,
\end{equation*}
and the desired bounds \eqref{lipa18} follow.

We can now prove the bounds \eqref{lipa16}. First, using only \eqref{lipa11} and \eqref{lipa17}, if $\beta\in[0,\pi-1/10]$ then we estimate the left-hand side of \eqref{lipa16} by
\begin{equation}\label{lipa19}
C\int_{-\pi}^\pi |\varphi_{>-6}(\theta)|2^{-4a}\frac{1}{|\theta-\beta|^{1/2}|\theta+\beta|^{1/2}}\,d\theta\lesssim 2^{-4a},
\end{equation}
as desired. On the other hand, if $\beta\in[\pi-1/10,\pi]$ then we need to use the cancellation estimates \eqref{lipa18}. For this we first make the change of variables $\theta\to-\theta$ when $\theta\in[-\pi,0]$, so we rewrite 
\begin{equation*}
\begin{split}
&\int_{-\pi}^\pi \varphi_{>-6}(\theta)s(\theta)\big[(1-re^{i(\theta-\beta)})(1-re^{i(\theta+\beta)})\big]^{-1/2}\,d\theta\\
&=\int_{0}^\pi \varphi_{>-6}(\theta)\Big\{\frac{s(\theta)}{\big[(1-re^{i(\theta-\beta)})(1-re^{i(\theta+\beta)})\big]^{1/2}}+\frac{s(-\theta)}{\big[(1-re^{i(-\theta-\beta)})(1-re^{i(-\theta+\beta)})\big]^{1/2}}\Big\}\,d\theta.
\end{split}
\end{equation*}
The contribution of the integral for $\theta\in[0,\pi-1/5]$ can be bounded as in \eqref{lipa19} by $C2^{-4a}$, using just \eqref{lipa11} and \eqref{lipa17}. Then we make the changes of variable $\theta=\pi-\rho$, $\beta=\pi-\beta'$; for \eqref{lipa16} it remains to prove that
\begin{equation}\label{lipa20}
\int_{0}^{1/5} \Big|\frac{s(\pi-\rho)}{\big[(1-re^{-i(\rho-\beta')})(1-re^{-i(\rho+\beta')})\big]^{1/2}}+\frac{s(-\pi+\rho)}{\big[(1-re^{i(\rho-\beta')})(1-re^{i(\rho+\beta')})\big]^{1/2}}\Big|\,d\rho\lesssim 2^{-a}
\end{equation}
for any $r\in[1/2,1)$ and $\beta'\in[0,1/10]$. 

The integral over $\rho\lesssim (1-r)+\beta'$ can be estimated as before, using just \eqref{lipa11} and \eqref{lipa17}. Also, using the derivative bounds in \eqref{lipa11} we may replace the factors $s(\pi-\rho)$ and $s(-\pi+\rho)=s(\pi+\rho)$ with $s(\pi)$, at the expense of $2^{-3a}$ errors. After these reductions, and using the bounds $|s(\pi)|\lesssim 2^{-4a}$ in \eqref{lipa11}, for \eqref{lipa20} it suffices to prove that
\begin{equation}\label{lipa21}
\int_{\rho_0}^{1/5} \Big|\big[(1-re^{-i(\rho-\beta')})(1-re^{-i(\rho+\beta')})\big]^{-1/2}+\big[(1-re^{i(\rho-\beta')})(1-re^{i(\rho+\beta')})\big]^{-1/2}\Big|\,d\rho\lesssim 1,
\end{equation}
for any $r\in[1/2,1)$ and $\beta'\in[0,1/10]$, where $\rho_0:=100[(1-r)+\beta']$. 

To prove \eqref{lipa21} we apply the cancellation bounds \eqref{lipa18}. Notice that
\begin{equation*}
\begin{split}
&(1-re^{-i(\rho-\beta')})(1-re^{-i(\rho+\beta')})=-x+iy,\\
&x:=-1-r^2\cos(2\rho)+2r\cos\rho\cos\beta',\qquad y:=-r^2\sin(2\rho)+2r\sin\rho\cos\beta'.
\end{split}
\end{equation*}
The integral in the left-hand side of \eqref{lipa21} is nontrivial only if $(1-r)+\beta'\leq 1/500$. In this case we can write, for $\rho\in[100[(1-r)+\beta'],1/5]$,
\begin{equation*}
\begin{split}
x&=-1-r^2(\cos\rho)^2+r^2(\sin\rho)^2+2r\cos\rho-2r\cos\rho(1-\cos\beta')\\
&=r^2(\sin\rho)^2-(1-r\cos\rho)^2-2r\cos\rho-4r\cos\rho[\sin(\beta'/2)]^2\\
&\in[r^2(\sin\rho)^2/2,r^2(\sin\rho)^2],
\end{split}
\end{equation*}
and
\begin{equation*}
\begin{split}
y=2r\sin\rho (\cos\beta'-r\cos\rho)\in [2r\sin\rho (1-r), 2r\sin\rho (1-r)+2r^3(\sin\rho)^3].
\end{split}
\end{equation*}
Therefore, using \eqref{lipa18}, we have
\begin{equation*}
\Big|\big[(1-re^{-i(\rho-\beta')})(1-re^{-i(\rho+\beta')})\big]^{-1/2}+\big[(1-re^{i(\rho-\beta')})(1-re^{i(\rho+\beta')})\big]^{-1/2}\Big|\lesssim 1+(1-r)/\rho^2,
\end{equation*}
and the desired bounds \eqref{lipa21} follow.
\medskip

{\bf{Step 2.}} We bound now the contribution of $\theta$ very close to $0$, so we show that
\begin{equation}\label{lipa25}
\Big|\int_{-\pi}^\pi \varphi_{\leq l}(\theta)s(\theta)\big[(1-re^{i(\theta-\beta)})(1-re^{i(\theta+\beta)})\big]^{-1/2}\,d\theta\Big|\lesssim 2^{2a}(1+2^a\beta)^{-3}
\end{equation}
for any $\beta\in[0,\pi]$ and $r\in[1-2^{-a-2},1)$, where $l\in\Z$ is such that $2^l\leq 2^{-10}\beta< 2^{l+1}$. 

Indeed, these bounds follow easily from \eqref{lipa17} and \eqref{lipa11} if $\beta\lesssim 2^{-a}$. Moreover, letting
\begin{equation*}
s'(\theta):=\sum_{n\in\Z\setminus\{0\}}\chi(\theta+2\pi n)=s(\theta)-\chi(\theta),
\end{equation*}
we have $|s'(\theta)|\lesssim 2^{-4a}$ for any $\theta\in[-\pi,\pi]$. Therefore the contribution of $s'$ leads to an acceptable error $\lesssim 2^{-4a}$, using \eqref{lipa17} as before; for \eqref{lipa25} it suffices to prove that
\begin{equation}\label{lipa26}
\Big|\int_\R \varphi_{\leq l}(\theta)\chi(\theta)\big[(1-re^{i(\theta-\beta)})(1-re^{i(\theta+\beta)})\big]^{-1/2}\,d\theta\Big|\lesssim 2^{2a}(1+2^a\beta)^{-3}
\end{equation}
for any $\beta\in[2^{-a+10},\pi]$ and $r\in[1-2^{-a-2},1)$.

To prove \eqref{lipa26} we need to integrate by parts in $\theta$. Since $\chi(\rho)=\widehat{m}(\rho)+2i\partial_\rho\widehat{m}(\rho)$, we have
\begin{equation}\label{lipa27}
\begin{split}
&\Big|\int_\R \varphi(2^{-l}\theta)\chi(\theta)\big[(1-re^{i(\theta-\beta)})(1-re^{i(\theta+\beta)})\big]^{-1/2}\,d\theta\Big|\lesssim I_1+I_2,\\
&I_1:=\Big|\int_\R 2^{-l}\varphi'(2^{-l}\theta)\widehat{m}(\theta)\big[(1-re^{i(\theta-\beta)})(1-re^{i(\theta+\beta)})\big]^{-1/2}\,d\theta\Big|\\
&I_2:=\Big|\int_\R \varphi(2^{-l}\theta)\widehat{m}(\theta)(1-r^2e^{2i\theta})\big[(1-re^{i(\theta-\beta)})(1-re^{i(\theta+\beta)})\big]^{-3/2}\,d\theta\Big|,
\end{split}
\end{equation}
where $\varphi'$ denote the derivative of $\varphi$ and the $I_2$ integral is obtained after integration by parts and careful algebraic simplifications. Since $|\widehat{m}(\theta)|\lesssim 2^{a}(1+2^a|\theta|)^{-6}$ and $\beta\approx 2^l\gtrsim 2^{-a}$ we can use \eqref{lipa17} to estimate
\begin{equation*}
I_1\lesssim \Big|\int_{|\theta|\in [2^{l-1},2^{l+1}]} 2^{-l}2^{a}(1+2^a|\theta|)^{-6}\beta^{-1}\,d\theta\Big|\lesssim 2^{a}(1+2^a\beta)^{-6}\beta^{-1}.
\end{equation*}
Using also the fact that $|1-r^2e^{2i\theta}|\leq 2|1-re^{i\theta}|\lesssim (1-r)+|\theta|\lesssim 2^{-a}+|\theta|$ we also estimate
\begin{equation*}
I_2\lesssim \int_{|\theta|\leq 2^{l+1}}2^{a}(1+2^a|\theta|)^{-6}(2^{-a}+|\theta|)\beta^{-3}\,d\theta\lesssim 2^{-a}\beta^{-3}.
\end{equation*}
The desired bounds \eqref{lipa26} follow from \eqref{lipa27} and the last two inequalities.

{\bf{Step 3.}} Finally we bound the contribution of the intermediate values of $\theta$, so we show that
\begin{equation}\label{lipa30}
\Big|\int_{-\pi}^\pi \varphi_{>l}(\theta)\varphi_{\leq -6}(\theta)s(\theta)\big[(1-re^{i(\theta-\beta)})(1-re^{i(\theta+\beta)})\big]^{-1/2}\,d\theta\Big|\lesssim 2^{2a}(1+2^a\beta)^{-3}
\end{equation}
for any $\beta\in[0,\pi]$ and $r\in[1/2,1)$, where $l\in\Z$ is such that $2^l\leq 2^{-10}\beta< 2^{l+1}$.

Using first just \eqref{lipa17} and \eqref{lipa11} we estimate
\begin{equation*}
\begin{split}
&\Big|\int_{|\theta|\in[2^{l-1},2^{l+20}]}\varphi_{>l}(\theta)\varphi_{\leq -6}(\theta)s(\theta)\big[(1-re^{i(\theta-\beta)})(1-re^{i(\theta+\beta)})\big]^{-1/2}\,d\theta\Big|\\
&\lesssim \int_{|\theta|\in[2^{l-1},2^{l+20}]}2^{2a}(1+2^a|\theta|)^{-6}|\theta-\beta|^{-1/2}|\theta+\beta|^{-1/2}\,d\theta\lesssim 2^{2a}(1+2^a\beta)^{-6}
\end{split}
\end{equation*}
and
\begin{equation*}
\begin{split}
&\Big|\int_{|\theta|\geq\max(2^{l+20},2^{-a-20})}\varphi_{>l}(\theta)\varphi_{\leq -6}(\theta)s(\theta)\big[(1-re^{i(\theta-\beta)})(1-re^{i(\theta+\beta)})\big]^{-1/2}\,d\theta\Big|\\
&\lesssim \int_{|\theta|\geq\max(2^{l+20},2^{-a-20})}2^{2a}(1+2^a|\theta|)^{-6}|\theta|^{-1}\,d\theta\lesssim 2^{2a}(1+2^a\beta)^{-5}.
\end{split}
\end{equation*}
For \eqref{lipa30} it remains to prove that if $l\leq -a-40$, $\beta\in[2^{l+10},2^{l+11}]$, and $r\in[1/2,1)$ then
\begin{equation}\label{lipa31}
\Big|\int_{|\theta|\in[2^{l+20},2^{-a-20}]}s(\theta)\big[(1-re^{i(\theta-\beta)})(1-re^{i(\theta+\beta)})\big]^{-1/2}\,d\theta\Big|\lesssim 2^{2a}.
\end{equation}

To prove \eqref{lipa31} (without logarithmic loss) we need to use again the cancellation bounds \eqref{lipa18}. First we replace $s(\theta)$ with $s(0)$ at the expense of an acceptable $O(2^{2a})$ error. Since $|s(0)|\lesssim 2^{2a}$ for \eqref{lipa31} it suffices to prove that
\begin{equation}\label{lipa32}
\int_{2^n}^{2^{-20}}\Big|\big[(1-re^{i(\theta-\beta)})(1-re^{i(\theta+\beta)})\big]^{-1/2}+\big[(1-re^{i(-\theta+\beta)})(1-re^{i(-\theta-\beta)})\big]^{-1/2}\Big|\,d\theta\lesssim 1,
\end{equation}
provided that $n\leq -20$, $\beta\in[0,2^{n-9}]$, and $r\in[1/2,1)$. The contribution of the integral over $\theta\lesssim (1-r)+\beta$ can be bounded using just \eqref{lipa17}, while the contribution of the integral over $\theta\geq 100((1-r)+\beta)$ is suitably bounded due to \eqref{lipa21}. This completes the proof of the main estimates \eqref{lipa15}. 
\end{proof}

We specialize now to the case of the spectral projections and prove the following:

\begin{corollary}\label{LiPaProj}
(i) For any $a\in\Z_+$ the operators $\mathcal{P}'_{\leq a}$ and their kernels $\mathcal{K}_{\leq a}$ satisfy the bounds
\begin{equation}\label{lipa5.6}
\|\mathcal{P}'_{\leq a}Y\|_{L^q(\mathbb{S}^2)}\lesssim 2^{2a(1/p-1/q)}\|Y\|_{L^p(\SS^2)}\qquad\text{ for any }p\leq q\in[1,\infty],
\end{equation}
and
\begin{equation}\label{lipa33}
|\mathcal{K}_{\leq a}(x,y)|\lesssim 2^{2a}(1+2^a|x-y|)^{-3}\qquad\text{ for any }x,y\in\mathbb{S}^2.
\end{equation}

(ii) Assume $V\in\{V_{12}, V_{23}, V_{31}\}$ (see definition \eqref{vf11}) and $a\in\Z_+$. Then $V\mathcal{P}'_a=\mathcal{P}'_aV$ and
\begin{equation}\label{lipa77}
\|V\mathcal{P}'_af\|_{L^p}\lesssim 2^a\|\mathcal{P}'_af\|_{L^p}
\end{equation}
for any $p\in[1,\infty]$ and $f\in L^p(\SS^2)$. Moreover, if $a\geq 1$ then
\begin{equation}\label{lipa78}
\|\mathcal{P}'_af\|_{L^p}\lesssim 2^{-a}\sum_{V\in\{V_{12}, V_{23}, V_{31}\}}\|V\mathcal{P}'_af\|_{L^p}.
\end{equation}
\end{corollary}

\begin{proof} (i) We apply Lemma \ref{LiPaLem} with $m(x)=\varphi(2^{-a}x)$, so $\mathcal{P}_{\leq a}=T_m$ and the kernel bounds \eqref{lipa33} follow from \eqref{lipa5.5}. In particular
\begin{equation*}
\sup_{x\in\mathbb{S}^2}\|\mathcal{K}_{\leq a}(x,.)\|_{L^r(\mathbb{S}^2)}+\sup_{y\in\mathbb{S}^2}\|\mathcal{K}_{\leq a}(.,y)\|_{L^r(\mathbb{S}^2)}\lesssim 2^{2a(1-1/r)}
\end{equation*}
for any $r\in[1,\infty]$, and the desired $L^p\to L^q$ bounds \eqref{lipa5.6} follow.

(ii) More generally, assume that
\begin{equation*}
T_mf(x)=\int_{\SS^2}f(y)H_m(x,y)\, d\mu(y)=\int_{\SS^2}f(y)G_m(x\cdot y)\, d\mu(y)
\end{equation*}
is a multiplier operator as in Lemma \ref{LiPaLem}. If $V\in\{V_{12}, V_{23}, V_{31}\}$ and $x\in \SS^2$ then
\begin{equation*}
\begin{split}
V_x(T_mf)(x)&=\int_{\SS^2}f(y)G'_m(x\cdot y)(x_iy_j-x_jy_i)\, d\mu(y)=-\int_{\SS^2}f(y)(V_yH_m)(x,y)\, d\mu(y)\\
&=\int_{\SS^2}(Vf)(y)H_m(x,y)\, d\mu(y)=T_m(Vf)(x),
\end{split}
\end{equation*}
using \eqref{vf13}. Thus $V$ commutes with all operators defined by Fourier multipliers.

We prove now the bounds \eqref{lipa77}--\eqref{lipa78}. Notice that the bounds \eqref{lipa78} would in fact follow from \eqref{lipa77}: if $a\geq 1$ then
\begin{equation*}
\begin{split}
\|\mathcal{P}'_af\|_{L^p}&\lesssim 2^{-2a}\|\Delta_{\SS^2}\mathcal{P}'_af\|_{L^p}\\
&\lesssim 2^{-2a}\sum_{V\in\{V_{12}, V_{23}, V_{31}\}}\|V\mathcal{P}'_aVf\|_{L^p}\\
&\lesssim 2^{-a}\sum_{V\in\{V_{12}, V_{23}, V_{31}\}}\|\mathcal{P}'_aVf\|_{L^p}
\end{split}
\end{equation*}
using Lemma \ref{LiPaLem}, the identity \eqref{coo7}, the commutation identity $V\mathcal{P}'_a=\mathcal{P}'_aV$, and the bounds \eqref{lipa77}. This gives the desired estimates \eqref{lipa78}.

It remains to prove the bounds \eqref{lipa77}. Let \(g:=\mathcal{P}'_a f\). Without loss of generality we may assume that $V=V_{12}$. Using the standard $(\theta,\phi)$ coordinates defined in section \ref{coord}, $V=\partial_\phi$, so it suffices to prove that 
\begin{equation}\label{lipa77.5}
\|\partial_\phi g(\theta,.)\|_{L^p[0,2\pi]}\lesssim 2^a\|g(\theta,.)\|_{L^p[0,2\pi]}
\end{equation}
for any $\theta\in[0,\pi]$ and $p\in[1,\infty]$. However, the function $g$ is the restriction to the sphere of a harmonic polynomial in the Euclidean variables of degree $\leq 2^{a+1}$. Thus, for any $\theta$ fixed, the function $\phi\to g(\theta,\phi)$ is a trigonometric polynomial in $\phi$ of degree $\leq 2^{a+1}$. The desired bounds \eqref{lipa77.5} follow from the Bernstein inequality.

\end{proof}

\subsection{Bounds on vector spherical harmonics} We prove first quantitative bounds on the radial-poloidal-toroidal decomposition of vector-fields on the sphere.

\begin{lemma}\label{rptDec0}
(i) Any vector-valued function $X\in H^1(\SS^2:\R^3)$ admits a unique radial-poloidal-toroidal decomposition 
\begin{equation}\label{rptDec1}
\begin{split}
&X=\mathbf{Y}(X)+\mathbf{\Psi}(X)+\mathbf{\Phi}(X),\\
&\mathbf{Y}(X):=(x^iX^i)\vec{\mathbf{r}},\\
&\mathbf{\Psi}(X):={}^t(W_1f_{\mathbf{\Psi}}(X), W_2f_{\mathbf{\Psi}}(X), W_3f_{\mathbf{\Psi}}(X)),\\
&\mathbf{\Phi}(X):={}^t(V_{23}f_{\mathbf{\Phi}}(X), V_{31}f_{\mathbf{\Phi}}(X), V_{12}f_{\mathbf{\Phi}}(X)),
\end{split}
\end{equation}
where $\vec{\mathbf{r}}(x):={}^t(x^1,x^2,x^3)$, $L:=x^1\partial _1+x^2\partial_2+x^3\partial_3$, $W_j:=\partial_j-x^jL$, $j\in\{1,2,3\}$, $V_{12}, V_{23}, V_{31}$ are defined as in \eqref{vf11}, and $f_{\mathbf{\Psi}}(X), f_{\mathbf{\Phi}}(X)\in H^2(\SS^2:\R)$.

(ii) For any $a\in\Z_+$ and $p\in[2,\infty]$ we have
\begin{equation}\label{rptDec2}
\begin{split}
\|\mathcal{P}'_a\mathbf{Y}(X)\|_{L^p}+\|\mathcal{P}'_a\mathbf{\Psi}(X)\|_{L^p}+\|\mathcal{P}'_a\mathbf{\Phi}(X)\|_{L^p}+2^a\|\mathcal{P}'_af_{\mathbf{\Psi}}(X)\|_{L^p}+2^a\|\mathcal{P}'_af_{\mathbf{\Phi}}(X)\|_{L^p}\\
\lesssim \big\|\sum_{a'\in\Z_+,\,|a'-a|\leq 3}\mathcal{P}'_{a'}X\big\|_{L^p}.
\end{split}
\end{equation}
\end{lemma}

\begin{proof} (i) Notice that the vector-fields $W_1, W_2, W_3$ are tangent to the sphere $\SS^2$ and 
\begin{equation}\label{rptDec3}
W_1=x^3V_{31}-x^2V_{12}, \qquad W_2=x^1V_{12}-x^3V_{23},\qquad W_3=x^2V_{23}-x^1V_{31}.
\end{equation}
Notice also that $[V_{12},V_{23}]=-V_{31}$, $[V_{23},V_{31}]=-V_{12}$, $[V_{31},V_{12}]=-V_{23}$, and $x^1V_{23}+x^2V_{31}+x^3V_{12}=0$. Using these formulas it is easy to see that 
\begin{equation*}
V_{23}W_1+V_{31}W_2+V_{12}W_3=W_1V_{23}+W_2V_{31}+W_3V_{12}=0.
\end{equation*}
Moreover, using the identities \eqref{vf12} we have
\begin{equation*}
W_1=\cos\phi\cos\theta\partial_\theta-\frac{\sin\phi}{\sin\theta}\partial_\phi,\qquad W_2=\sin\phi\cos\theta\partial_\theta+\frac{\cos\phi}{\sin\theta}\partial_\phi,\qquad W_3=-\sin\theta\partial_\theta,
\end{equation*}
and $\Delta_{\SS^2}=W_1^2+W_2^2+W_3^2$.

Using these identities we can find the poloidal and toroidal potentials  $f_{\mathbf{\Psi}}(X), f_{\mathbf{\Phi}}(X)$. We decompose first $X=\mathbf{Y}(X)+\mathbf{Y}_{tan}(X)$ where $\mathbf{Y}_{tan}(X):=X-(x^iX^i)\vec{\mathbf{r}}$ is the tangential component of $X$. Since $\mathbf{Y}_{tan}(X)=\mathbf{\Psi}(X)+\mathbf{\Phi}(X)$, we necessarily have
\begin{equation*}
\begin{split}
&W_1(\mathbf{Y}_{tan}^1(X))+W_2(\mathbf{Y}_{tan}^2(X))+W_3(\mathbf{Y}_{tan}^3(X))=\Delta_{\SS^2}f_{\mathbf{\Psi}}(X),\\
&V_{23}(\mathbf{Y}_{tan}^1(X))+V_{31}(\mathbf{Y}_{tan}^2(X))+V_{12}(\mathbf{Y}_{tan}^3(X))=\Delta_{\SS^2}f_{\mathbf{\Phi}}(X).
\end{split}
\end{equation*}
Moreover, using \eqref{vf13} and \eqref{rptDec4} and letting $h_j:=\mathbf{Y}_{tan}^j(X)$, $j\in\{1,2,3\}$, we have
\begin{equation*}
\int_{\SS^2}(V_{23}h_1+V_{31}h_2+V_{12}h_3)\,d\mu=0
\end{equation*}
and
\begin{equation*}
\begin{split}
&\int_{\SS^2}(W_1h_1+W_2h_2+W_3h_3)\,d\mu\\
&=\int_{\SS^2}[(x^3V_{31}-x^2V_{12})h_1+(x^1V_{12}-x^3V_{23})h_2+(x^2V_{23}-x^1V_{31})h_3]\,d\mu\\
&=\int_{\SS^2}(2x^1h_1+2x^2h_2+2x^3h_3)\,d\mu=0.
\end{split}
\end{equation*}
The last identity follows since $x^1h_1+x^2h_2+x^3h_3=0$, due to the fact that $\mathbf{Y}_{tan}(X)$ is tangent to $\SS^2$. Thus, the potentials $f_{\mathbf{\Psi}}(X), f_{\mathbf{\Phi}}(X)$ can be defined (uniquely up to constants) by
\begin{equation}\label{rptDec4}
\begin{split}
&f_{\mathbf{\Psi}}(X):=\Delta_{\SS^2}^{-1}\big[W_1(\mathbf{Y}_{tan}^1(X))+W_2(\mathbf{Y}_{tan}^2(X))+W_3(\mathbf{Y}_{tan}^3(X))\big],\\
&f_{\mathbf{\Phi}}(X):=\Delta_{\SS^2}^{-1}\big[V_{23}(\mathbf{Y}_{tan}^1(X))+V_{31}(\mathbf{Y}_{tan}^2(X))+V_{12}(\mathbf{Y}_{tan}^3(X))\big].
\end{split}
\end{equation}
Clearly, the definitions show that
\begin{equation}\label{rptDec5}
\|\mathbf{Y}_{tan}(X)\|_{H^1}+\|f_{\mathbf{\Psi}}(X)\|_{H^2}+\|f_{\mathbf{\Phi}}(X)\|_{H^2}\lesssim \|X\|_{H^1}.
\end{equation}

It remains to prove the identity in the first line of \eqref{rptDec1}, which is equivalent to proving that $\mathbf{Y}_{tan}(X)=\mathbf{\Psi}(X)+\mathbf{\Phi}(X)$. Let $\rho:=\mathbf{Y}_{tan}(X)-\mathbf{\Psi}(X)-\mathbf{\Phi}(X)$; the definitions show that
\begin{equation}\label{rptDec7}
\begin{split}
&x^1\rho_1+x^2\rho_2+x^3\rho_3=0,\\
&V_{23}\rho_1+V_{31}\rho_2+V_{12}\rho_3=0,\\
&W_1\rho_1+W_2\rho_2+W_3\rho_3=0,
\end{split}
\end{equation}
on the sphere $\SS^2$, and we have to prove that $\rho\equiv 0$. 

This claim can be proved intrinsically, using covariant div-curl systems on $\SS^2$. We prefer to prove it extrinsically. For this we extend the functions $\rho_j$ to the annulus $B_{1/2,2}:=\{x\in\R^3:\,|x|\in (1/2,2)\}$ as homogeneous functions of degree $-1$, $\widetilde{\rho}_j(x):=|x|^{-1}\rho_j(x/|x|)$. Then we use the first two equations in \eqref{rptDec7} to eliminate the variable $\widetilde{\rho}_3$, i.e.
\begin{equation*}
\widetilde{V}_{12}(x^1\widetilde{\rho}_1+x^2\widetilde{\rho}_2)-x^3(\widetilde{V}_{23}\widetilde{\rho}_1+\widetilde{V}_{31}\widetilde{\rho}_2)=0,
\end{equation*}
where $\widetilde{V}_{12}=x^1\partial_2-x^2\partial_1$, $\widetilde{V}_{23}=x^2\partial_3-x^3\partial_2$, $\widetilde{V}_{31}:=x^3\partial_1-x^1\partial_3$ are the natural extensions of the vector-fields $V_{12}, V_{23}, V_{31}$ to the domain $B_{1/2,2}$. We can also eliminate the $x^3\partial_3$ derivatives, since $(x^1\partial_1+x^2\partial_2+x^3\partial_3)\widetilde{\rho}_j=-\widetilde{\rho}_j$ by homogeneity, $j\in\{1,2,3\}$, so the identity above gives $\partial_2\widetilde{\rho}_1-\partial_1\widetilde{\rho}_2=0$ in $B_{1/2,2}$. Similar calculations show that $\partial_3\widetilde{\rho}_2-\partial_2\widetilde{\rho}_3=0$ and $\partial_1\widetilde{\rho}_3-\partial_3\widetilde{\rho}_1=0$ in $B_{1/2,2}$. Since the domain $B_{1/2,2}$ is simply connected it follows that 
\begin{equation}\label{rptDec8}
\widetilde{\rho}_j=\partial_jp\qquad \text{ for some function }p\in H^2_{loc}(B_{1/2,2}).
\end{equation}

The function $p$ is homogeneous of degree $0$. We substitute \eqref{rptDec8} into the last equation in \eqref{rptDec7}, therefore $\Delta_{\R^3}p=0$ in  $B_{1/2,2}$. Using \eqref{coo4} it follows that $\Delta_{S^2}p=0$, thus $p$ is a constant function, so $\widetilde{\rho}_j\equiv 0$ as a consequence of \eqref{rptDec8}.

(ii) To prove the frequency localized bounds in \eqref{rptDec2} we notice that if $X\in\mathcal{H}_k^3$ for some $k\in\Z_+$ then, as a consequence of Lemma \ref{zon20} and the formulas \eqref{rptDec1}, 
\begin{equation}\label{rptDec10}
\begin{split}
&\mathbf{Y}(X),f_{\mathbf{\Phi}}(X),\mathbf{\Phi}(X)\in\bigcup_{k'\in\Z_+,\,|k'-k|\leq 2}\mathcal{H}_{k'}^3,\\
&f_{\mathbf{\Psi}}(X)\in\bigcup_{k'\in\Z_+,\,|k'-k|\leq 3}\mathcal{H}_{k'}^3,\qquad \mathbf{\Psi}(X)\in\bigcup_{k'\in\Z_+,\,|k'-k|\leq 4}\mathcal{H}_{k'}^3.
\end{split}
\end{equation}
Therefore $\mathcal{P}'_a\mathbf{Y}(X)=\mathcal{P}'_a\mathbf{Y}(X_{a;3})$, where for $m\geq 0$, $X_{a;m}:=\sum_{a'\in\Z_+,\,|a-a'|\leq m}\mathcal{P}'_{a'}(X)$. Thus
\begin{equation*}
\|\mathcal{P}'_a\mathbf{Y}(X)\|_{L^p}\lesssim \|\mathbf{Y}(X_{a;3})\|_{L^p}\lesssim \|X_{a;3}\|_{L^p},
\end{equation*}
for any $a\in\Z_+$ and $p\in [2,\infty]$, as claimed in \eqref{rptDec2}. Similarly $\mathcal{P}'_a f_{\mathbf{\Phi}}(X)=\mathcal{P}'_af_{\mathbf{\Phi}}(X_{a;3})$ and $\mathcal{P}'_a\mathbf{\Phi}(X)=\mathcal{P}'_a\mathbf{\Phi}(X_{a;3})$, so, using \eqref{rptDec4}, \eqref{lipa5.6}, and \eqref{lipa77},
\begin{equation*}
\begin{split}
\|\mathcal{P}'_af_{\mathbf{\Phi}}(X)\|_{L^p}&\lesssim 2^{-2a}\|\mathcal{P}'_a\Delta_{\SS^2} (f_{\mathbf{\Phi}}(X_{a;3}))\|_{L^p}\lesssim 2^{-a}\|\mathbf{Y}_{tan}(X_{a;3})\|_{L^p}\lesssim 2^{-a}\|X_{a;3}\|_{L^p},
\end{split}
\end{equation*}
as desired. Using \eqref{rptDec1}, \eqref{lipa5.6}, and \eqref{lipa77} it follows that $\|\mathcal{P}'_a\mathbf{\Phi}(X)\|_{L^p}\lesssim \|X_{a;3}\|_{L^p}$. The bounds for the function $f_{\mathbf{\Psi}}(X)$ and the vector-field $\mathbf{\Psi}(X)$ follow in the same way, which completes the proof of the lemma.
\end{proof}

We prove now bounds on the linear semigroup of our Peskin system.

\begin{lemma}\label{tAop}
(i) Recall the operator $\mathcal{A}=\mathcal{N}_1$ defined in \eqref{vec31}. Then, for any $p\in[2,\infty]$, $t\in[0,\infty)$, $a\in\Z_+$, and $X\in L^p(\SS^2:\R^3)$
\begin{equation}\label{lipa80}
\begin{split}
\big\|\mathcal{P}'_a(e^{t\mathcal{A}}-I)X\big\|_{L^p}&\lesssim (t2^a)^{1/2}\|X_{a;3}\|_{L^p}\qquad \text{ if }\qquad t2^{a}\lesssim 1,\\
\big\|\mathcal{P}'_ae^{t\mathcal{A}}X\big\|_{L^p}&\lesssim (t2^a)^{-4}\|X_{a;3}\|_{L^p}\qquad \text{ if }\qquad t2^{a}\gtrsim 1,
\end{split}
\end{equation}
where, as before, $X_{a;m}:=\sum_{a'\in\Z_+,\,|a-a'|\leq m}\mathcal{P}'_{a'}(X)$ for any $m\geq 0$.

(ii) Moreover, if $Y\in (\ker \mathcal{N}_1)^\perp$ and $0<\lambda <a_2=8/35$, then
\begin{equation}\label{lipa80_exp}
\begin{split}
\big\|\mathcal{P}'_ae^{t\mathcal{A}}Y\big\|_{L^p}&\lesssim (t2^a)^{-4}e^{-t\lambda }\|Y_{a;3}\|_{L^p}\qquad \text{ if }\qquad t2^{a}\gtrsim 1,
\end{split}
\end{equation}

(iii) In particular, for any $t>0$, $p\in[2,\infty]$ and $X\in L^p(\SS^2:\R^3)$ we have
\begin{equation}\label{lipa85}
\|e^{t\mathcal{A}}X\big\|_{L^p}\lesssim \|X\|_{L^p}.
\end{equation}
\end{lemma}

\begin{proof} Notice first that the operators $\mathcal{A}$ and $e^{t\mathcal{A}}$ are not defined by spectral multipliers, since different parts of the space of spherical harmonics of degree $k$ are multiplied by different multipliers depending on $k$, as given by \eqref{vec31}, so Lemma \eqref{LiPaLem} cannot be applied directly. 

We proceed using the radial-poloidal-toroidal decomposition in Lemma \ref{rptDec0}. For any vector-valued function $X\in H^1(\SS^2:\R^3)$ we decompose it as in Lemma \ref{rptDec0},
\begin{equation}\label{tete1}
\begin{split}
&X=\mathbf{Y}(X)+\mathbf{\Psi}(X)+\mathbf{\Phi}(X),\\
&\mathbf{Y}(X):=f_{\mathbf{Y}}(X)\cdot\vec{\mathbf{r}},\\
&\mathbf{\Psi}(X):={}^t(W_1f_{\mathbf{\Psi}}(X), W_2f_{\mathbf{\Psi}}(X), W_3f_{\mathbf{\Psi}}(X)),\\
&\mathbf{\Phi}(X):={}^t(V_{23}f_{\mathbf{\Phi}}(X), V_{31}f_{\mathbf{\Phi}}(X), V_{12}f_{\mathbf{\Phi}}(X)),
\end{split}
\end{equation}
 where $f_{\mathbf{\Psi}}(X), f_{\mathbf{\Phi}}(X)\in H^2(\SS^2:\R)$ and $f_{\mathbf{Y}}(X):=x^iX^i\in H^1(\SS^2:\R)$. We would like to understand the action of the operators $\mathcal{A}$ and $e^{t\mathcal{A}}$ on the components $\mathbf{Y}(X)$, $\mathbf{\Psi}(X)$, $\mathbf{\Phi}(X)$. 
 \medskip

{\bf{Step 1.}} More generally, assume that $h\in H^1(\SS^2:\R)$, let $h=\sum_{k\in\Z}h_k$, $k\in\Z_+$, $h_k\in\mathcal{H}_k$, denote its decomposition in spherical harmonics, and let $\widetilde{h}_k\in\mathcal{A}_k$ denote the homogeneous extension of $h_k$. For $k\in\Z$ let
\begin{equation}\label{tete2}
\mathbf{Y}_{h_k}:=h_k\cdot\vec{\mathbf{r}}, \qquad \mathbf{\Psi}_{h_k}:={}^t(W_1h_k, W_2h_k, W_3h_k), \qquad \mathbf{\Phi}_{h_k}:={}^t(V_{23}h_k, V_{31}h_k, V_{12}h_k).
\end{equation}
Using the definitions \eqref{vec2} we have
\begin{equation}\label{tete3}
\begin{split}
&\mathbf{Y}_{h_k}=\frac{1}{2k+1}\mathbf{H}^1_{k+1}+\frac{1}{2k+1}\mathbf{H}^2_{k-1},\\
&\mathbf{\Psi}_{h_k}=\frac{-k}{2k+1}\mathbf{H}^1_{k+1}+\frac{k+1}{2k+1}\mathbf{H}^2_{k-1},\\
&\mathbf{\Phi}_{h_k}=\mathbf{H}^3_k.
\end{split}
\end{equation}
Therefore, using \eqref{vec31} and recalling \eqref{tete4},
we have
\begin{equation}\label{tete5}
\begin{split}
e^{t\mathcal{A}}(\mathbf{Y}_{h_k})&=\frac{1}{2k+1}e^{-ta_k}\H_{k+1}^{1}+\frac{1}{2k+1}e^{-tb_k}\mathbf{H}^2_{k-1}\\
&=\frac{e^{-ta_k}}{2k+1}[(k+1)\mathbf{Y}_{h_k}-\mathbf{\Psi}_{h_k}]+\frac{e^{-tb_k}}{2k+1}[k\mathbf{Y}_{h_k}+\mathbf{\Psi}_{h_k}]\\
&=\frac{(k+1)e^{-ta_k}+ke^{-tb_k}}{2k+1}\mathbf{Y}_{h_k}+\frac{-e^{-ta_k}+e^{-tb_k}}{2k+1}\mathbf{\Psi}_{h_k},
\end{split}
\end{equation}
\begin{equation}\label{tete6}
\begin{split}
e^{t\mathcal{A}}(\mathbf{\Psi}_{h_k})&=\frac{-k}{2k+1}e^{-ta_k}\H_{k+1}^{1}+\frac{k+1}{2k+1}e^{-tb_k}\mathbf{H}^2_{k-1}\\
&=\frac{-ke^{-ta_k}}{2k+1}[(k+1)\mathbf{Y}_{h_k}-\mathbf{\Psi}_{h_k}]+\frac{(k+1)e^{-tb_k}}{2k+1}[k\mathbf{Y}_{h_k}+\mathbf{\Psi}_{h_k}]\\
&=\frac{k(k+1)(-e^{-ta_k}+e^{-tb_k})}{2k+1}\mathbf{Y}_{h_k}+\frac{ke^{-ta_k}+(k+1)e^{-tb_k}}{2k+1}\mathbf{\Psi}_{h_k},
\end{split}
\end{equation}
\begin{equation}\label{tete7}
e^{t\mathcal{A}}(\mathbf{\Phi}_{h_k})=e^{-tc_k}\H_k^3=e^{-tc_k}\mathbf{\Phi}_{h_k}.
\end{equation}

{\bf{Step 2.}} We are now ready to prove the bounds \eqref{lipa80}. Assume $X\in H^1(\SS^2:\R^3)$ is given and decompose it as in \eqref{tete1}, 
\begin{equation*}
X=f_{\mathbf{Y}}\cdot\vec{\mathbf{r}}+{}^t(W_1f_{\mathbf{\Psi}}, W_2f_{\mathbf{\Psi}}, W_3f_{\mathbf{\Psi}})+{}^t(V_{23}f_{\mathbf{\Phi}}, V_{31}f_{\mathbf{\Phi}}, V_{12}f_{\mathbf{\Phi}})
\end{equation*}
where $f_{\mathbf{Y}}=f_{\mathbf{Y}}(X)\in H^1(\SS^2:\R)$, $f_{\mathbf{\Psi}}=f_{\mathbf{\Psi}}(X)\in H^2(\SS^2:\R)$, $f_{\mathbf{\Phi}}=f_{\mathbf{\Phi}}(X)\in H^2(\SS^2:\R)$. Then
\begin{equation}\label{tete10}
e^{t\mathcal{A}}(X)=g_{\mathbf{Y}}\cdot\vec{\mathbf{r}}+{}^t(W_1g_{\mathbf{\Psi}}, W_2g_{\mathbf{\Psi}}, W_3g_{\mathbf{\Psi}})+{}^t(V_{23}g_{\mathbf{\Phi}}, V_{31}g_{\mathbf{\Phi}}, V_{12}g_{\mathbf{\Phi}}),
\end{equation}
where the functions $g_{\mathbf{Y}}=g_{\mathbf{Y}}(X), g_{\mathbf{\Psi}}=g_{\mathbf{\Psi}}(X), g_{\mathbf{\Phi}}=g_{\mathbf{\Phi}}(X)$ are defined by linear operators,
\begin{equation}\label{tete11}
\begin{split}
&g_{\mathbf{Y}}=T_1(f_{\mathbf{Y}})+T_2(f_{\mathbf{\Psi}}),\qquad g_{\mathbf{\Psi}}=T_3(f_{\mathbf{Y}})+T_4(f_{\mathbf{\Psi}}),\qquad g_{\mathbf{\Phi}}=T_5(f_{\mathbf{\Phi}}).
\end{split}
\end{equation} 
In view of the identities \eqref{tete5}--\eqref{tete7}, the linear operators $T_1, T_2, T_3, T_4, T_5$ are defined by Fourier multipliers. More precisely,
\begin{equation}\label{tete12}
T_j\Big(\sum_{k\in\Z_+}h_k\Big)=\sum_{k\in\Z_+}m_j(k,t)h_k,\qquad j\in\{1,2,3,4,5\}
\end{equation}
where $h_k\in\mathcal{H}_k$ are spherical harmonics of degree $k$, and the multipliers $m_j$ are given by
\begin{equation}\label{tete13}
\begin{split}
&m_1(k,t):=\frac{(k+1)e^{-ta_k}+ke^{-tb_k}}{2k+1},\qquad m_2(k,t):=\frac{k(k+1)(e^{-tb_k}-e^{-ta_k})}{2k+1},\\
&m_3(k,t):=\frac{e^{-tb_k}-e^{-ta_k}}{2k+1},\qquad m_4(k,t):=\frac{ke^{-ta_k}+(k+1)e^{-tb_k}}{2k+1},\\
&m_5(k,t):=e^{-tc_k}.
\end{split}
\end{equation}
The point is that the operators $T_j$ are defined by spectral multipliers, and we can use Lemma \ref{LiPaLem} to bound these operators on $L^p$ spaces (see Remark \ref{rem_lem61}). Indeed, recalling the definitions \eqref{tete4}, it follows from Lemma \ref{LiPaLem} that if $l\in\Z_+$ satisfies $t2^l\lesssim 1$ then 
\begin{equation}\label{tete16}
\begin{split}
&\|\mathcal{P}'_l(T_1-I)\|_{L^p\to L^p}+\|\mathcal{P}'_l(T_4-I)\|_{L^p\to L^p}+\|\mathcal{P}'_l(T_5-I)\|_{L^p\to L^p}\lesssim t2^l,\\
&\|\mathcal{P}'_lT_2\|_{L^p\to L^p}\lesssim t2^{2l}, \qquad \|\mathcal{P}'_lT_3\|_{L^p\to L^p}\lesssim t.
\end{split}
\end{equation}
Moreover, if $l\in\Z_+$ satisfies $t2^l\gtrsim 1$ then 
\begin{equation}\label{tete17}
\begin{split}
&\|\mathcal{P}'_lT_1\|_{L^p\to L^p}+\|\mathcal{P}'_lT_4\|_{L^p\to L^p}+\|\mathcal{P}'_lT_5\|_{L^p\to L^p}\lesssim (t2^l)^{-5},\\
&\|\mathcal{P}'_lT_2\|_{L^p\to L^p}\lesssim 2^l(t2^l)^{-5}, \qquad \|\mathcal{P}'_lT_3\|_{L^p\to L^p}\lesssim 2^{-l}(t2^l)^{-5}.
\end{split}
\end{equation}

Assume now that $p\in [2,\infty]$, $a\in\Z_+$, $t2^a\lesssim 1$, and $X\in L^p(\SS^2:\R^3)$. Then, using \eqref{tete10}, \eqref{tete11}, and \eqref{tete16}, and \eqref{rptDec2},
\begin{equation*}
\begin{split}
\|\mathcal{P}'_a(e^{t\mathcal{A}}-I)X\|_{L^p}&\lesssim \sum_{l\in\Z_+,\,|l-a|\leq 1}\Big[\|\mathcal{P}'_l(g_{\mathbf{Y}}-f_{\mathbf{Y}})\|_{L^p}+2^l\|\mathcal{P}'_l(g_{\mathbf{\Psi}}-f_{\mathbf{\Psi}})\|_{L^p}+2^l\|\mathcal{P}'_l(g_{\mathbf{\Phi}}-f_{\mathbf{\Phi}})\|_{L^p}\Big]\\
&\lesssim \sum_{l\in\Z_+,\,|l-a|\leq 1}\Big[\|\mathcal{P}'_l(T_1-I)f_{\mathbf{Y}})\|_{L^p}+\|\mathcal{P}'_lT_2f_{\mathbf{\Psi}}\|_{L^p}+2^l\|\mathcal{P}'_lT_3f_{\mathbf{Y}})\|_{L^p}\\
&\qquad\qquad\qquad+2^l\|\mathcal{P}'_l(T_4-I)f_{\mathbf{\Psi}})\|_{L^p}+2^l\|\mathcal{P}'_l(T_5-I)f_{\mathbf{\Phi}})\|_{L^p}\Big]\\
&\lesssim \sum_{l\in\Z_+,\,|l-a|\leq 2}\Big[t2^l\|\mathcal{P}'_lf_{\mathbf{Y}}\|_{L^p}+t2^{2l}\|\mathcal{P}'_lf_{\mathbf{\Psi}}\|_{L^p}+t2^{2l}\|\mathcal{P}'_lf_{\mathbf{\Phi}}\|_{L^p}\Big]\\
&\lesssim t2^a\|X\|_{L^p}.
\end{split}
\end{equation*}
Similarly, using \eqref{tete10}, \eqref{tete11}, \eqref{tete17}, and \eqref{rptDec2}, if $t2^a\gtrsim 1$ then
\begin{equation*}
\begin{split}
\|\mathcal{P}'_ae^{t\mathcal{A}}X\|_{L^p}&\lesssim \sum_{l\in\Z_+,\,|l-a|\leq 1}\Big[\|\mathcal{P}'_lg_{\mathbf{Y}}\|_{L^p}+2^l\|\mathcal{P}'_lg_{\mathbf{\Psi}}\|_{L^p}+2^l\|\mathcal{P}'_lg_{\mathbf{\Phi}}\|_{L^p}\Big]\\
&\lesssim \sum_{l\in\Z_+,\,|l-a|\leq 1}\Big[\|\mathcal{P}'_lT_1f_{\mathbf{Y}})\|_{L^p}+\|\mathcal{P}'_lT_2f_{\mathbf{\Psi}}\|_{L^p}+2^l\|\mathcal{P}'_lT_3f_{\mathbf{Y}})\|_{L^p}\\
&\qquad\qquad\qquad+2^l\|\mathcal{P}'_lT_4f_{\mathbf{\Psi}})\|_{L^p}+2^l\|\mathcal{P}'_lT_5f_{\mathbf{\Phi}})\|_{L^p}\Big]\\
&\lesssim \sum_{l\in\Z_+,\,|l-a|\leq 2}\Big[(t2^l)^{-5}\|\mathcal{P}'_lf_{\mathbf{Y}}\|_{L^p}+2^l(t2^l)^{-5}\|\mathcal{P}'_lf_{\mathbf{\Psi}}\|_{L^p}+2^l(t2^{l})^{-5}\|\mathcal{P}'_lf_{\mathbf{\Phi}}\|_{L^p}\Big]\\
&\lesssim (t2^a)^{-5}\|X\|_{L^p}.
\end{split}
\end{equation*}
The desired bounds \eqref{lipa80} follow from these last two inequalities and the observation that $\mathcal{P}'_ae^{t\mathcal{A}}X=\mathcal{P}'_ae^{t\mathcal{A}}X_{a;3}$ for any $a\in\Z_+$. This completes the proof of part (i).

(ii) For $Y\in (\ker \mathcal{N}_1)^\perp$, the values $k=0,1$ are excluded. Therefore,  \eqref{tete13} and Lemma \ref{LiPaLem} give that, 
if $l\in\Z_+$ satisfies $t2^l\gtrsim 1$, then 
\begin{equation}\label{tete17_exp}
\begin{split}
&\|\mathcal{P}'_lT_1\|_{L^p\to L^p}+\|\mathcal{P}'_lT_4\|_{L^p\to L^p}+\|\mathcal{P}'_lT_5\|_{L^p\to L^p}\lesssim (t2^l)^{-5}e^{-\lambda t},\\
&\|\mathcal{P}'_lT_2\|_{L^p\to L^p}\lesssim 2^l(t2^l)^{-5}e^{-\lambda t}, \qquad \|\mathcal{P}'_lT_3\|_{L^p\to L^p}\lesssim 2^{-l}(t2^l)^{-5}e^{-\lambda t},
\end{split}
\end{equation}
and the conclusion follows as in (i).

(iii) We fix an integer $b_t$ such that $2^{b_t}t\in[1/2,2]$. Then we estimate, using \eqref{lipa80} and \eqref{lipa5.6} and the convention $\mathcal{P}'_{\leq b_t}=0$ if $b_t\leq -1$,
\begin{equation*}
\begin{split}
\|e^{t\mathcal{A}}X\|_{L^p}&\leq \|\mathcal{P}'_{\leq b_t}(e^{t\mathcal{A}}-I)X\|_{L^p}+\|\mathcal{P}'_{\leq b_t}X\|_{L^p}+\sum_{b\in\Z_+,\,b>b_t}\|\mathcal{P}'_be^{t\mathcal{A}}X\|_{L^p}\\
&\lesssim \|X\|_{L^p}+\sum_{b\in\Z_+,\,b\leq b_t}\|\mathcal{P}'_b(e^{t\mathcal{A}}-I)X\|_{L^p}+\sum_{b\in\Z_+,\,b>b_t}\|\mathcal{P}'_be^{t\mathcal{A}}X\|_{L^p}\\
&\lesssim \|X\|_{L^p}+\sum_{b\in\Z_+,\,b\leq b_t}(t2^b)^{1/2}\|X\|_{L^p}+\sum_{b\in\Z_+,\,b>b_t}(t2^b)^{-4}\|X\|_{L^p}\\
&\lesssim \|X\|_{L^p},
\end{split}
\end{equation*}
as desired.
\end{proof}

\begin{remark}\label{rem_lem61}
To apply the uniform bounds of Lemma \ref{LiPaLem} to the spectral operators $T_j$, whose multipliers $m_j(k,t)$ do not satisfy the order-zero condition \eqref{lipa4.1}, we employ the following procedure.
For a fixed $l \in \mathbb{Z}_+$ and operator index $j$, we define the (localized) multiplier
\begin{equation}
M_{j,l}(k) := \varphi_l(k) m_j(k,t).
\end{equation}
By the compact support of $\varphi_l$, $M_{j,l}$ vanishes outside the dyadic annulus $k \sim 2^l$. We extract the maximal amplitude $A_{j,l}$ of this composite multiplier by bounding its derivatives up to order $6$
\begin{equation}
A_{j,l} := \max_{0 \le p \le 6} \left( \sup_{k \in \mathbb{R}}  2^{pl} |\partial_k^p M_{j,l}(k)| \right).
\end{equation}
We then define the normalized multiplier $\widetilde{m}_{j,l}(k) := A_{j,l}^{-1} M_{j,l}(k)$. By construction, $\widetilde{m}_{j,l}$ is compactly supported in $[-2^{l+1}, 2^{l+1}]$ and satisfies the hypothesis \eqref{lipa4.1} of Lemma \ref{LiPaLem}
\begin{equation}
2^{pl}|\partial_k^p \widetilde{m}{j,l}(k)| \leq 1 \qquad \text{for all } p \in {0, \dots, 6}.
\end{equation}
Consequently, Lemma \ref{LiPaLem} implies that the physical-space kernel associated with $\widetilde{m}{j,l}$ is uniformly bounded in $L^1(\mathbb{S}^2)$, yielding the uniform operator bound on $L^p(\mathbb{S}^2)$
\begin{equation}
\|T_{\widetilde{m}} Y\|_{L^p} \lesssim \|Y\|_{L^p}.
\end{equation}
By the linearity of the spectral operators, $\mathcal{P}'_l T_j = A_{j,l} T_{\widetilde{m}}$. Multiplying the uniform bound by the extracted amplitude yields 
\begin{equation}
\|\mathcal{P}'_l T_j\|_{L^p \to L^p} \lesssim A_{j,l}.
\end{equation}
    
\end{remark}

\subsubsection{Local coordinates} It is useful sometimes to be able to work in local coordinates on the sphere. Precisely, we fix two smooth coordinate charts $\psi_1:B_1\to U_1$ and $\psi_2:B_1\to U_2$, where $B_1:=\{x\in\R^2:\,|x|<1\}$, $U_1:=\{x\in\SS^2:\,x_3>-1/2\}$, and $U_2:=\{x\in\SS^2:\,x_3<1/2\}$, and an associated smooth partition of unity 
\begin{equation}\label{Pes50}
\begin{split}
&\phi_1,\phi_2:\SS^2\to[0,1],\qquad \phi_1+\phi_2\equiv 1\,\,\text{ on }\,\,\SS^2,\\
&\mathrm{supp}\,\phi_1\subseteq U'_1:=\{x\in\SS^2:\,x_3>-1/4\},\quad \mathrm{supp}\,\phi_2\subseteq U'_2:=\{x\in\SS^2:\,x_3<1/4\}.
\end{split}
\end{equation}
Given any function $f:\SS^2\to\C$ we can decompose it and write it in local coordinates
\begin{equation}\label{Pes51}
\begin{split}
&f=f_1+f_2, \qquad f_a(x):=f(x)\phi_a(x),\,a\in\{1,2\},\\
&\widetilde{f}_a(x):=f_a(\psi_a(x)),\qquad \widetilde{f}_a:B_1\to\C.
\end{split}
\end{equation}


\section{Numerical verification}\label{sec:numerical}

The numerical computation is performed with a curved quadratic triangulation
$\Gamma_h$ of the reference sphere.  On each element $T$, with quadratic
basis $\{\varphi_\alpha\}_{\alpha=1}^6$, we represent the surface as
\[
  \mathbf X_h|_T
  =
  \sum_{\alpha=1}^6 \mathbf X_{T,\alpha}\varphi_\alpha .
\]
The topology and quadrature remain fixed on the reference mesh.  In
particular, the elastic force is assembled from the reference
Laplace--Beltrami operator: if
\[
  M^T_{\alpha\beta}
  =
  \int_T\varphi_\alpha\varphi_\beta\,d\mu_h,
  \qquad
  K^T_{\alpha\beta}
  =
  \int_T
  \nabla_{\SS^2}\varphi_\alpha\cdot\nabla_{\SS^2}\varphi_\beta\,d\mu_h ,
\]
then $L_h f_h$, the discrete approximation of $\Delta_{\SS^2} f$, is
defined by $M_h L_h f_h=-K_h f_h$.  At a reference degree of freedom
$\widehat{\boldsymbol\alpha}_m$, the velocity is then evaluated by the
reference-measure single-layer formula
\begin{equation}\label{eq:verification-discrete-velocity}
  U_h^i(\widehat{\boldsymbol\alpha}_m)
  =
  \sum_T
  \int_T
  G_{ij}(\mathbf X_h(\widehat{\boldsymbol\alpha}_m)-\mathbf X_h(p))
  (L_h\mathbf X_h)^j(p)\,d\mu_h(p).
\end{equation}
Thus the source measure and the elastic force are tied to the reference
sphere, while the Stokeslet kernel sees the current deformed surface.  This is
also the reference sphere used for the structural projection below.

For $\mathbf q\in\R^6$, let $\Lambda_{\mathbf q}$ and $\mathcal G_{\mathbf q}$ as in \eqref{eq:ExpMap} and define
\[
  \boldsymbol{\mathcal S}_{\mathbf q}(\boldsymbol\alpha)
  =
  \frac{\mathbb P_3\Lambda_{\mathbf q}\widetilde{\boldsymbol\alpha}}
       {\mathbf{e}_0^\top\Lambda_{\mathbf q}\widetilde{\boldsymbol\alpha}},
  \qquad
  \widetilde{\boldsymbol\alpha}=(1,\boldsymbol\alpha^\top)^\top ,
  \quad
  \mathbb P_3=(\mathbf 0_3\;\mathbb I_3),
  \quad
  \mathbf{e}_0=(1,\mathbf 0_3^\top)^\top .
\]
Writing
$\mathbf v(\mathbf q)=(q_1,q_2,q_3)^\top$, the discrete structural
coordinates are defined by
\begin{equation}\label{eq:verification-structural-decomposition}
  \mathbf X_h(\boldsymbol\alpha,t)
  =
  \mathbf Z_{\rho_h}(\boldsymbol{\mathcal S}_{\mathbf q_h(t)}
     (\boldsymbol\alpha),t)
  -\frac{2}{3}\mathbf v(\mathbf q_h(t))
  +\mathbf a_h(t),
  \qquad
  \mathbf Z_{\rho_h}(\boldsymbol\beta,t)
  =
  (1+\rho_h(t))\boldsymbol\beta+\mathbf Y_h(\boldsymbol\beta,t).
\end{equation}
Thus
\[
  \mathbf p_h(t)=(\rho_h(t),\mathbf a_h(t),\mathbf q_h(t))
  \in \R\times\R^3\times\R^6
\]
is the finite-dimensional coordinate vector, and
$\mathbf Y_h$ is the perturbation field in the conformal variable
$\boldsymbol\beta$.

It remains to specify how the parameters are extracted from a given numerical 
surface.  For fixed $t$, they are obtained by the $L^2(\SS^2)$ projection
onto the structural manifold:
\begin{equation}\label{eq:verification-continuum-optimization}
  \mathbf p_h(t)
  =
  \operatorname*{argmin}_{\rho>-1,\,\mathbf a\in\R^3,\,\mathbf q\in\R^6}
  \frac12
  \int_{\SS^2}
  \left|
  \mathbf X_h(\boldsymbol{\mathcal S}_{\mathbf q}^{-1}
       (\boldsymbol\beta),t)
  +\frac{2}{3}\mathbf v(\mathbf q)-\mathbf a
  -(1+\rho)\boldsymbol\beta
  \right|^2
  d\mu(\boldsymbol\beta).
\end{equation}
Equivalently, pulling this objective back to the reference variable
$\boldsymbol\alpha$ gives the Jacobian
\[
  J_{\mathbf q}(\boldsymbol\alpha)
  =
  \left(\mathbf{e}_0^\top\Lambda_{\mathbf q}
        \widetilde{\boldsymbol\alpha}\right)^{-2}.
\]
If $\hat{\boldsymbol\alpha}_{T,\alpha}$ denotes a reference degree of freedom and
$\{\varphi_\alpha\}$ is the local quadratic basis, the discrete objective is
\begin{equation}\label{eq:verification-structural-optimization}
  \frac12
  \sum_T\sum_{\alpha,\beta=1}^6
  \overline M^T_{\alpha\beta}(\mathbf q)\,
  \eta_{T,\alpha}(\rho,\mathbf a,\mathbf q,t)\cdot
  \eta_{T,\beta}(\rho,\mathbf a,\mathbf q,t),
\end{equation}
where
\[
  \eta_{T,\alpha}
  =
  \mathbf X_{T,\alpha}(t)
  +\frac{2}{3}\mathbf v(\mathbf q)
  -\mathbf a
  -(1+\rho)\boldsymbol{\mathcal S}_{\mathbf q}
     (\hat{\boldsymbol\alpha}_{T,\alpha}),
  \qquad
  \overline M^T_{\alpha\beta}(\mathbf q)
  =
  \int_T
  \varphi_\alpha\varphi_\beta
  J_{\mathbf q}\,d\mu_h .
\]
This is the finite-dimensional analogue of imposing the
$L^2(\SS^2)$-orthogonality conditions against the neutral modes. We solve this optimization problem numerically with the Gauss-Newton method. 
Once $\mathbf p_h(t)$ has been found, the deformation is defined as the residual in
\eqref{eq:verification-structural-decomposition}:
\begin{equation}\label{eq:verification-Y}
  \mathbf Y_h(\boldsymbol\beta,t)
  =
  \mathbf X_h(\boldsymbol{\mathcal S}_{\mathbf q_h(t)}^{-1}
      (\boldsymbol\beta),t)
  +\frac{2}{3}\mathbf v(\mathbf q_h(t))-\mathbf a_h(t)
  -(1+\rho_h(t))\boldsymbol\beta .
\end{equation}
We then compute the convergence rate of the perturbation field and the parameters.

We choose the following 4 smooth surfaces as initial conditions,
\[
  \mathbf X_h^{(k)}(0,x)
  =
  \begin{pmatrix}
  (1+\varepsilon e^{\cos x^2})x^1\\
  (1+\varepsilon e^{\sin x^3})x^2\\
  (1+\varepsilon e^{\cos(kx^1)})x^3
  \end{pmatrix},
  \qquad k=1,2,3,4,\qquad \varepsilon=0.25 .
\]
The mesh is obtained by three successive refinements of an icosahedron.  We use the time step $\Delta t=0.01$, Heun's second-order
Runge--Kutta method, and final time $T=15$. The numerical results are shown in Figure~\ref{fig:numerical-verification-structural-convergence}, which shows the decay of the fitted perturbation field and the parameters for these four initial surfaces.

\begin{figure}[htbp]
\centering
\includegraphics[width=0.95\linewidth]{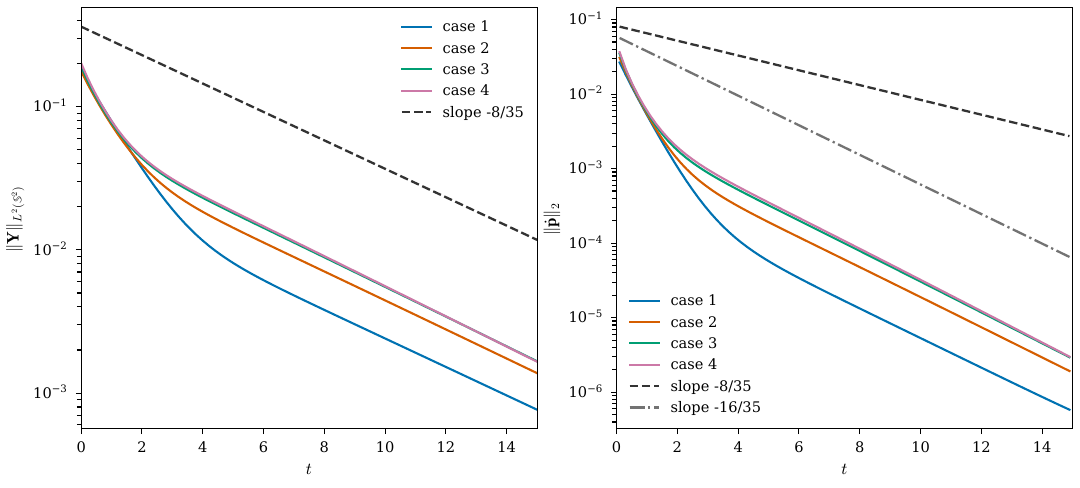}
\caption{Decay of $\norm{\mathbf Y_h(t)}_{L^2(\SS^2)}$ and $\norm{\dot{\mathbf p}_h(t)}_2$ on $0\le t\le15$.}
\label{fig:numerical-verification-structural-convergence}
\end{figure}

\section*{Acknowledgements}
EGJ was partially supported by the RYC2021-032877 research grant, the RED2022-134784-T funded by  MCIN/AEI/10.13039/50110001103, the AEI project PID2022-140494NA-I00, and the IMUS-Maria de Maeztu grant CEX2024-001517-M, funded by MICIU/AEI/ 10.13039/501100011033. SVH was partially supported by the NSF grant DMS-2511086. YM was partially supported by NSF DMR-2309034, Materials Research Science and Engineering Center (MRSEC) awarded to the University of Pennsylvania.

\bibliographystyle{acm}
\bibliography{references}

\begin{thebibliography}{10}

\bibitem{AlazardNguyen21}
{\sc Alazard, T., and Nguyen, Q.-H.}
\newblock On the {C}auchy problem for the {M}uskat equation. {II}: {C}ritical initial data.
\newblock {\em Ann. PDE 7}, 1 (2021), Paper No. 7, 25.

\bibitem{AlazardNguyen22}
{\sc Alazard, T., and Nguyen, Q.-H.}
\newblock Quasilinearization of the 3{D} {M}uskat equation, and applications to the critical {C}auchy problem.
\newblock {\em Adv. Math. 399\/} (2022), Paper No. 108278, 52.

\bibitem{AlazardNguyen23}
{\sc Alazard, T., and Nguyen, Q.-H.}
\newblock Endpoint {S}obolev theory for the {M}uskat equation.
\newblock {\em Comm. Math. Phys. 397}, 3 (2023), 1043--1102.

\bibitem{BaldiJulinLaManna26}
{\sc Baldi, P., Julin, V., and La~Manna, D.~A.}
\newblock Liquid drop with capillarity and rotating traveling waves.
\newblock {\em Arch. Ration. Mech. Anal. 250}, 1 (2026), Paper No. 4, 57.

\bibitem{CameronStrain24}
{\sc Cameron, S., and Strain, R.~M.}
\newblock Critical local well-posedness for the fully nonlinear {P}eskin problem.
\newblock {\em Comm. Pure Appl. Math. 77}, 2 (2024), 901--989.

\bibitem{ChenHuNguyen26}
{\sc Chen, K., Hu, R., and Nguyen, Q.-H.}
\newblock Schauder-type estimates and well-posedness for nonlocal quasilinear evolution equations in fluid dynamics.
\newblock {\em Preprint arXiv:2604.10682\/} (2026).

\bibitem{ChenNguyen23}
{\sc Chen, K., and Nguyen, Q.-H.}
\newblock The {P}eskin problem with {$\dot B^1_{\infty,\infty }$} initial data.
\newblock {\em SIAM J. Math. Anal. 55}, 6 (2023), 6262--6304.

\bibitem{ChenNguyenXu21}
{\sc Chen, K., Nguyen, Q.-H., and Xu, Y.}
\newblock The {M}uskat problem with ${C}^1$ data.
\newblock {\em Trans. Am. Math. Soc. 375\/} (2021), 3039--3060.

\bibitem{GancedoGJPatelStrain23}
{\sc Gancedo, F., Garc\'ia-Ju\'arez, E., Patel, N., and Strain, R.~M.}
\newblock Global regularity for gravity unstable {M}uskat bubbles.
\newblock {\em Mem. Amer. Math. Soc. 292}, 1455 (2023), v+87.

\bibitem{GancedoGJPatelStrain25}
{\sc Gancedo, F., Garc\'ia-Ju\'arez, E., Patel, N., and Strain, R.~M.}
\newblock On nonlinear stability of {M}uskat bubbles.
\newblock {\em J. Math. Pures Appl. (9) 194\/} (2025), Paper No. 103664, 31.

\bibitem{GancedoGraneroBelinchonScrobogna23}
{\sc Gancedo, F., Granero-Belinch\'on, R., and Scrobogna, S.}
\newblock Global existence in the {L}ipschitz class for the {N}-{P}eskin problem.
\newblock {\em Indiana Univ. Math. J. 72}, 2 (2023), 553--602.

\bibitem{GJGomezSerranoHaziotPausader24}
{\sc Garc\'ia-Ju\'arez, E., G\'omez-Serrano, J., Haziot, S.~V., and Pausader, B.}
\newblock Desingularization of small moving corners for the {M}uskat equation.
\newblock {\em Ann. PDE 10}, 2 (2024), Paper No. 17, 71.

\bibitem{GJHaziot25}
{\sc Garc\'ia-Ju\'arez, E., and Haziot, S.~V.}
\newblock Critical well-posedness for the 2{D} {P}eskin problem with general tension.
\newblock {\em Adv. Math. 460\/} (2025), Paper No. 110047, 46.

\bibitem{GJPoChunMori25}
{\sc Garc\'{\i}a-Ju\'{a}rez, E., Kuo, P.-C., and Mori, Y.}
\newblock The immersed inextensible interface problem in 2{D} {S}tokes flow.
\newblock {\em SIAM J. Math. Anal. 57}, 4 (2025), 3454--3487.

\bibitem{GJPChunMoriStrain25}
{\sc Garc\'ia-Ju\'arez, E., Kuo, P.-C., Mori, Y., and Strain, R.~M.}
\newblock Well-posedness of the 3{D} {P}eskin problem.
\newblock {\em Math. Models Methods Appl. Sci. 35}, 1 (2025), 113--216.

\bibitem{GJMoriStrain23}
{\sc Garc\'ia-Ju\'arez, E., Mori, Y., and Strain, R.~M.}
\newblock The {P}eskin problem with viscosity contrast.
\newblock {\em Anal. PDE 16}, 3 (2023), 785--838.

\bibitem{GigaGu26}
{\sc Giga, Y., and Gu, Z.}
\newblock On existence of a collapsed bubble with surface tension in viscous incompressible fluid.
\newblock {\em Preprint, arXiv:2606.31266\/} (2026).

\bibitem{KandalamSpirn26}
{\sc Kandalam, A.~T., and Spirn, D.}
\newblock {A}nchored {P}eskin {P}roblem.
\newblock {\em Preprint, arXiv:2604.27219\/} (2026).

\bibitem{LaiWeinstein23}
{\sc Lai, C.-C., and Weinstein, M.~I.}
\newblock Free boundary problem for a gas bubble in a liquid, and exponential stability of the manifold of spherically symmetric equilibria.
\newblock {\em Arch. Ration. Mech. Anal. 247}, 5 (2023), Paper No. 100, 87.

\bibitem{Li21}
{\sc Li, H.}
\newblock Stability of the {S}tokes immersed boundary problem with bending and stretching energy.
\newblock {\em J. Funct. Anal. 281}, 9 (2021), Paper No. 109204, 65.

\bibitem{LinTong19}
{\sc Lin, F.-H., and Tong, J.}
\newblock Solvability of the {S}tokes immersed boundary problem in two dimensions.
\newblock {\em Comm. Pure Appl. Math. 72}, 1 (2019), 159--226.

\bibitem{MeyerNiebelChristian25}
{\sc Meyer, D., Niebel, L., and Seis, C.}
\newblock Steady bubbles and drops in inviscid fluids.
\newblock {\em Calc. Var. Partial Differential Equations 64}, 9 (2025), Paper No. 299, 30.

\bibitem{MoriOhmSpirn20}
{\sc Mori, Y., Ohm, L., and Spirn, D.}
\newblock Theoretical justification and error analysis for slender body theory.
\newblock {\em Comm. Pure Appl. Math. 73}, 6 (2020), 1245--1314.

\bibitem{MoriOhmSpirn202}
{\sc Mori, Y., Ohm, L., and Spirn, D.}
\newblock Theoretical justification and error analysis for slender body theory with free ends.
\newblock {\em Arch. Ration. Mech. Anal. 235}, 3 (2020), 1905--1978.

\bibitem{MoriRodenbergSpirn19}
{\sc Mori, Y., Rodenberg, A., and Spirn, D.}
\newblock Well-posedness and global behavior of the {P}eskin problem of an immersed elastic filament in {S}tokes flow.
\newblock {\em Comm. Pure Appl. Math. 72}, 5 (2019), 887--980.

\bibitem{Ohm26}
{\sc Ohm, L.}
\newblock A free boundary problem for an immersed filament in 3{D} {S}tokes flow.
\newblock {\em To appear in Ann. PDE\/} (2025).

\bibitem{Ohm25}
{\sc Ohm, L.}
\newblock The slender body free boundary problem.
\newblock {\em Preprint, arXiv:2509.16800\/} (2025).

\bibitem{Peskin77}
{\sc Peskin, C.~S.}
\newblock Numerical analysis of blood flow in the heart.
\newblock {\em J. Comput. Phys. 25}, 3 (1977), 220--252.

\bibitem{Peskin02}
{\sc Peskin, C.~S.}
\newblock The immersed boundary method.
\newblock {\em Acta Numer. 11\/} (2002), 479--517.

\bibitem{PrussSimonett16}
{\sc Pr\"{u}ss, J., and Simonett, G.}
\newblock {\em Moving interfaces and quasilinear parabolic evolution equations}, vol.~105 of {\em Monographs in Mathematics}.
\newblock Birkh\"{a}user/Springer, [Cham], 2016.

\bibitem{Shao26}
{\sc Shao, C.}
\newblock On the {C}auchy problem of spherical capillary water waves.
\newblock {\em Forum Math. 38}, 3 (2026), 727--769.

\bibitem{Steinbook71}
{\sc Stein, E.~M., and Weiss, G.}
\newblock {\em Introduction to {F}ourier analysis on {E}uclidean spaces}, vol.~No. 32 of {\em Princeton Mathematical Series}.
\newblock Princeton University Press, Princeton, NJ, 1971.

\bibitem{Tong21}
{\sc Tong, J.}
\newblock {{R}egularized {S}tokes immersed boundary problems in two dimensions: Well-posedness, singular limit, and error estimates}.
\newblock {\em Comm. Pure Appl. Math. 74(2):366–449\/} (2021).

\bibitem{Tong24}
{\sc Tong, J.}
\newblock Global solutions to the tangential {P}eskin problem in 2-{D}.
\newblock {\em Nonlinearity 37}, 1 (2024), Paper No. 015006, 52.

\bibitem{TongDongyi24}
{\sc Tong, J., and Wei, D.}
\newblock Geometric properties of the 2-{D} {P}eskin problem.
\newblock {\em Ann. PDE 10}, 2 (2024), Paper No. 24, 62.

\bibitem{TongWei25}
{\sc Tong, J., and Wei, D.}
\newblock The immersed boundary problem in 2-{D}: the {N}avier-{S}tokes case.
\newblock {\em Preprint arXiv:2511.16189\/} (2025).

\end{thebibliography}

\end{document}